\documentclass[preprint,3p]{amsart}
\usepackage{amsmath}
\usepackage{amsfonts}
\usepackage{amsthm}
\usepackage{amssymb, textcomp}
\usepackage{url}
\usepackage[usenames]{color}
\usepackage{xcolor}
\usepackage{stmaryrd}
\usepackage{verbatim}

\usepackage{mathtools}

\usepackage[T1]{fontenc}
\usepackage[utf8]{inputenc}
\usepackage{graphicx}
\usepackage{caption}
\usepackage{subcaption}
\usepackage[a4paper,width=150mm,top=30mm,bottom=25mm,bindingoffset=6mm]{geometry}
\usepackage[hidelinks]{hyperref}
\usepackage{mathdots}
\usepackage[cleardoublepage=empty]{scrextend}

\usepackage{fancyhdr}

\fancypagestyle{fancy}{
\fancyhf{}
\fancyhead[RO,LE]{\thepage}
\fancyhead[RE]{\nouppercase{\leftmark}}
\fancyhead[LO]{\nouppercase{\rightmark}}

}

\fancypagestyle{fancy_biblio}{
\fancyhf{}
\fancyhead[RO,LE]{\thepage}
\fancyhead[LO,RE]{Bibliography}

}

\usepackage[backend=biber,style=trad-alpha,natbib=true, safeinputenc, doi=false, isbn=false, url=false, eprint=true]{biblatex}
\usepackage[autostyle=true]{csquotes} 

\AtEveryBibitem{%
	\ifentrytype{book}{%
		\clearfield{pages}%
		\clearfield{pagetotal}%
	}{}%
}

\newcommand{\SA}{\mathfrak{R}}
\newcommand{\tmm}{\boldsymbol{\tilde{m}}}

\newcommand{\mm}{\boldsymbol{m}}

\newcommand{\ord}{\ell}

\usepackage{enumitem}

\usepackage{amsmath, mathtools}
\usepackage{amstext}
\usepackage{amsthm}
\usepackage{amssymb}
\usepackage{stmaryrd}
\usepackage{graphicx}
\usepackage{bbm}
\usepackage{comment}
\usepackage{faktor,url}
\usepackage{xfrac}
\usepackage{hyperref}

\numberwithin{equation}{section}
\theoremstyle{plain}
\newtheorem{thm}{\protect\theoremname}[section]
  \theoremstyle{definition}
  \newtheorem{defn}[thm]{\protect\definitionname}
  \theoremstyle{remark}
  \newtheorem{rem}{\protect\remarkname}
  \theoremstyle{plain}
  \newtheorem{lem}[thm]{\protect\lemmaname}
  \theoremstyle{plain}
  \newtheorem{prop}[thm]{\protect\propositionname}
 \theoremstyle{plain}

\providecommand{\definitionname}{Definition}
\providecommand{\lemmaname}{Lemma}
\providecommand{\propositionname}{Proposition}
\providecommand{\remarkname}{Remark}
\providecommand{\theoremname}{Theorem}
\providecommand{\conjecturename}{Conjecture}

\newcommand{\R}{\mathbb{R}}
\newcommand{\Z}{\mathbb{Z}}
\newcommand{\N}{\mathbb{N}}
\newcommand{\T}{\mathbb{T}}

\newcommand{\D}{\mathcal{D}}

\newcommand{\OO}{\mathcal{O}}

\newcommand{\C}{\mathbb{C}}

\newcommand{\MM}{\mathcal{M}}

\newcommand{\Id}{\mathrm{Id}}

\newcommand{\M}{\mathcal{M}}

\newcommand{\bigO}{\mathcal{O}}

\usepackage{color}

\title{Universal Dynamics in a family of Celestial Mechanics models}

\author[M. Garrido]{Miguel Garrido}
\address[MG]{Institute of Science and Technology Austria, Am Campus 1, 3400 Klosterneuburg, Austria}
\email{miguel.garrido@ist.ac.at}

\author[P. Mart\'in]{Pau Mart\'in}
\address[PM]{ Departament de Matem\`atiques, Universitat Polit\`ecnica de Catalunya, Diagonal 647, 08028 Barcelona, Spain \& Centre de Recerca Matemàtica, Edifici C, Campus Bellaterra, 08193 Bellaterra, Spain}
\email{p.martin@upc.edu}
\date{\today}

\begin{document}

\begin{abstract}
%
%
Universal maps (maps whose renormalized iterations approximate every map in a given class) are locally generic in several spaces of diffeomorphisms~\cite{BonattiDiazPMIHES2003,TuraevCMP2015}. In fluid dynamics, steady Euler flows whose Poincar\'e maps are universal are also known to be locally dense~\cite{BergerFlorioPeraltaSalas}. Motivated by Arnold's vision~\cite{Arnold66} that the complexity of orbits in celestial mechanics and that in fluids should be comparable, in the present paper we address the question of the existence of universal maps in celestial mechanics. We give a partial answer by proving a weak, finite-dimensional form of universal dynamics. More concretely, we show that every symplectic embedding of the disk into $\mathbb{R}^2$ can be approximated, with arbitrary precision, by a renormalization of a Poincar\'e map of a restricted planar circular $(n+1)$-body problem, for suitable $n$ and appropriate choice of the masses of the primaries. The proof relies on Gonchenko--Shilnikov--Turaev theory~\cite{Turaev03,GSTtangencies} and requires control of the dynamics near homoclinic tangencies of arbitrarily high order; the techniques developed in~\cite{GMP2025} are essential for such control.
\end{abstract}

\maketitle

\tableofcontents

\section{Introduction}
In \cite{Arnold66}, V.I. Arnold expressed the following view:
\begin{quote}
    ``The flow of a vector field $X$ satisfying $\mathrm{curl}\, X = \lambda \cdot X$ probably displays stream lines with topologies as complicated as those of orbits in Celestial Mechanics.''
\end{quote}

If a vector field $X$ satisfies $\mathrm{curl} X = \lambda \cdot X$, for some real number $\lambda \neq 0$, then $X$ is called $\lambda$-Beltrami. This type of fields are important in the study of ideal fluids in equilibrium.

In \cite{BergerFlorioPeraltaSalas}, Berger, Florio and Peralta-Salas proved that Beltrami fields can exhibit universal dynamics: there exists a Beltrami field $X$ such that any area preserving diffeomorphism of the disk can be approximated, with arbitrary precision, by the Poincaré map of $X$ at some transverse section. In fact, they proved that there exists a locally dense set comprised of Beltrami fields displaying this property.

Moreover, the authors of \cite{BergerFlorioPeraltaSalas} state that their result is indeed stronger than Arnold's vision, since at the moment there was no example of universal dynamics in Celestial Mechanics. In the present work we show that inside a specific family of celestial mechanics models one can find, to any desired degree of accuracy, any conservative dynamics, thereby completing Arnold's vision.

We remark that establishing the occurrence of phenomena for maps coming from Celestial Mechanics is usually much more involved than when one has the freedom to choose among all symplectic maps, due to the presence of singularities and degeneracies. For instance, Arnold diffusion was already proved for an ad hoc model in 1964 (see \cite{ArnoldDiffusion_1964}), whereas in the setting of Celestial Mechanics this has only been achieved recently (see \cite{DKdlRS2019, Diffusion_4body_PBIHES} for the first complete analytical proofs).

\subsection{Universal dynamics.}
A recent survey of results on wild dynamics, including universality, can be found in \cite{BergerICM2026}; we follow it for the next definitions:
\begin{rem}
The terms in bracket refer to the definitions in the symplectic category. Ignore the brackets for the general case.
\end{rem}

\begin{defn}\label{def:primitive_renormalization}
 Let $M$ be a (symplectic) $d$-dimensional manifold and $f \colon M \to M$ a (symplectic) diffeomorphism. A map $F: U\subset \R^d \to V\subset \R^d$ is a (symplectic) \emph{primitive renormalization} of $f$ if there is a (conformally symplectic) diffeomorphism $\phi: U\cup V\hookrightarrow M$ and an integer $k$ such that:
\[ f^k\circ \phi|_{U} = \phi \circ F\]
and the sets $(f^\ell\circ \phi(U))_{0\le \ell< k}$ are disjoint.
\end{defn}

\begin{defn}\label{def:universal_diffeo}
A (symplectic) diffeomorphism $f$ is $C^r$-\emph{universal} if the set of its (symplectic) primitive renormalizations is dense in the space of orientation preserving (symplectic) $C^r$-embeddings of the $d$-dimensional ball $\mathbb{B}$ into $\mathbb{R}^d$
\end{defn}

The notion of universal dynamics was first introduced by Bonatti and Díaz in \cite{BonattiDiazPMIHES2003}, where they proved that, for any compact manifold $M$ of dimension $d \ge 3$, $C^1$-universality is a locally generic property for $C^1$-diffeomorphisms on $M$. Their mechanism consists in constructing a \textit{wild homoclinic class} (wild meaning that it does not admit any dominated splitting) using \textit{blenders} obtained from heterodimensional cycles (a cycle connecting two saddle points with different number of stable directions). The notion of \textit{blender} was first introduced in \cite{BonattiDiazAnnals1996} and has found widespread use within smooth dynamics. We highlight the recent application of symplectic blenders to celestial mechanics models by Guàrdia and Paradela \cite{GuardiaParadela2025arxiv} to show a strong form of topological instability; see also \cite{paradela2026arxiv} for a higher dimensional version. Li and Turaev \cite{LiTuraev2026arxiv} have also recently proved, through a different mechanism, the emergence of symplectic blenders close to symplectic maps displaying a one-dimensional KAM curve with a transverse homoclinic orbit to it; they discuss a possible application of their results to celestial mechanics too. Finally, in the contact flow setting, blenders have recently been constructed by Chaidez~\cite{chaidez2025blenders}.

In the case of surface diffeomorphisms, since a surface cannot support saddles with different number of stable directions, heterodimensional cycles and blenders are topologically precluded. Thus, in dimension $d=2$, universality can instead be achieved through renormalization of the dynamics close to homoclinic tangencies of arbitrarily high order: rescaled first-return maps near a chain of such tangencies converge to Hénon-like polynomial maps; in turn, by Turaev's polynomial-approximation theorem \cite{Turaev03}, every symplectic diffeomorphism can be $C^r$-approximated by iterates of such Hénon-like maps. This approach was carried out in \cite{GSTtangencies} by Gonchenko, Shilnikov and Turaev to show that, in the case of symplectomorphisms, for $d=2$ and $r \in \{2,\ldots,\infty,\omega\}$ (where $r=\omega$ corresponds to the real-analytic case), $C^r$-universal maps form a residual subset in the \textit{Newhouse domains} of the space of symplectic two-dimensional $C^r$ diffeomorphisms. Newhouse domains are open subsets in the space of diffeomorphisms consisting of those which exhibit homoclinic tangencies. With regard to celestial mechanics, Newhouse domains were first constructed, in the restricted planar circular three-body problem (RPC3BP), by Gorodetski and Kaloshin in the unpublished manuscript \cite{GorodetskiK12}; in this work Newhouse domains appear at the unfolding of a homoclinic quadratic tangency between the invariant manifolds of a degenerate parabolic saddle ``at infinity''. In \cite{Giraltcoorbitalchaos}, Baldomá, Giralt and Guàrdia also obtained Newhouse domains in the RPC3BP, in their case related to a hyperbolic periodic point. Finally, in \cite{GMP2025}, the aforementioned result in \cite{GorodetskiK12} was obtained using different techniques, which will be crucial for the present paper.

To complete the picture on universality in dynamics, in \cite{GelfreichTuraev2010RCD} Gelfreich and Turaev showed that that, for $d=2$, a generic symplectic map with an elliptic periodic point is $C^\omega$-universal. This has been recently generalized to arbitrary (even) dimension $d$ in \cite{BergerTuraev2025arxiv}. Finally, in the general non-symplectic setting, Turaev \cite{TuraevCMP2015} showed that, for $d=2$ and $2 \le r \le \infty$, $C^r$-universal maps form a locally generic subset of the space of $C^r$ diffeomorphisms on any surface. See also \cite{BergerGH2025Inventiones} for a recent improvement in Turaev's argument.

As can be noted from the previous exposition, it is still an open question whether $C^r$-universal maps are locally generic in the space of $C^r$-diffeomorphism on a manifold of dimension $d \ge 3$ when $2 \le r \le \infty$.

Turning back to the framework of the present work, universality is obtained, similarly to \cite{BergerFlorioPeraltaSalas}, via Gonchenko--Shilnikov--Turaev (GST) Theory~\cite{Turaev03, GSTtangencies}. More precisely,  we prove the following theorem.
\begin{thm}
    \label{thm:main_theorem_introduction} For every $r \in \{2,\ldots,\infty,\omega\}$, the space of symplectic  $C^r$-embeddings of the disk into $\mathbb{R}^2$ is contained in the $C^r$-closure of the set of symplectic primitive renormalizations of symplectic maps induced by a family of Celestial Mechanics models.
\end{thm}
This theorem is an immediate consequence of Theorem~\ref{thm:uni_main_result}.
In a sense, such \textit{family} of maps can be said to be $C^r$-universal (compare with Definition~\ref{def:universal_diffeo}).

We remark that, while Beltrami fields are solutions to a PDE, which can be considered a system with infinite degrees of freedom, models in classical celestial mechanics consist of a finite number $n$ of discrete particles following Newton's law of gravitation and therefore have $6n$-degrees of freedom and a finite number of possible free parameters. Since the applicability of GST universality theory requires an arbitrary number of free parameters, we are forced to consider a \textit{family} of models instead of a concrete one.

It might still be possible to obtain universality for a specific celestial mechanics model (rather than for a family of models) by extending GST theory or by applying a different framework. Nevertheless, we believe that the intermediate steps that we carry out in the present work in order to apply GST theory might be of independent interest.

\subsection{Central configurations.} As previously explained, in order to apply GST universality theory, an arbitrarily high number of free parameters is required. We will consider a family of celestial mechanics models consisting of restricted $n+1$ body problems, where the dynamics of an infinitesimal body under the gravitational attraction of other $n$ massive bodies (called the primaries) is studied. In these models, the infinitesimal body is affected by, but does not affect, the $n$ primaries. Hence, a solution of the $n$ body problem has to be prescribed for the motion of the $n$ primaries. The only known explicit solutions of the $n$ body problem arise from (planar) \textit{central configurations}. A central configuration is a special arrangement of $n$ point masses with the property that the net acceleration vector exerted on each body by all the others points towards the center of mass of the system and has a magnitude proportional to the distance to the center. Planar central configurations yield a type of periodic solutions of the $n$ body problem, called \textit{homographic}, where the whole configuration just rotated rigidly around its center of mass with constant angular speed; for this reason, planar central configurations are also named \textit{relative equilibria} (i.e., equilibria in a rotating frame of reference).  An excellent exposition on central configurations can be found in \cite{Moeckel2015CourseCC}. 

The algebraic equations defining central configurations are invariant under the group of similarities of Euclidean space, which comprises translations, rotations, reflections and homotheties. The question of the finiteness, modulus similarities, of the number of planar central configuration of $n$ bodies for any fixed choice of $n$ positive masses appears as Problem 6 in Smale's list of problems for the 21st century \cite{SmaleProblems1998, SmaleProblems2000}. Hampton and Moeckel \cite{HamptonMoeckel2006Inventiones} proved finiteness in the case $n=4$, and Albouy and Kaloshin \cite{AlbouyKaloshin2012Annals} proved finiteness when $n=5$ except for masses in an exceptional closed algebraic subset of codimension 2.\footnote{They also gave an alternative proof of the case $n=4$ which is less computation-heavy than Hampton-Moeckel's original computed-assisted proof.} Moreover, Lindstrom \cite{Lindstrom2001TAMS} showed that the number of $n$-body central configurations is finite when $n-1$ masses are located at fixed positions in the plane, for $n>2$ and $n \neq 4$; the case $n=4$ is proved in \cite{AlvarezCorberaDelgadoLlibre2005PAMS}.

Remarkably, the conjecture on finiteness is false if one allows negative masses: Roberts \cite{Roberts1999PhysicaD} showed that for a $5$-body system with ``masses'' $(1,1,1,1,-1/4)$ there is a continuous $1$-parameter family of non-equivalent central configurations.

Since we are concerned with families of central configurations with an arbitrarily high number of bodies $n \in \mathbb{N}$, we mention two classical examples of such families\footnote{In \cite{Xia1991JDE}, Xia also presents a method for adding an arbitrary number of small masses to an existing nondegenerate central configuration via an iterative process.}: Moulton's collinear \cite{MoultonCollinear1910Annals}, where all the bodies lay in a line; and Maxwell's (centered) regular $N$-gon \cite{MaxwellRing}, where $N=n-1$ bodies of equal masses are located at the vertices of a regular $N$-gon (together with an optional additional body at its center). This latter type of configurations is the one we consider in the present paper.

In particular, we will take the central body to be much more massive than the ones in the ring. The limiting case of a central configuration of $n=N+1$ masses where $N$ of them tend to zero is usually referred to in the literature as \textit{the $(1+N)$-body problem} (the order of the summands indicating the different magnitude of the masses). There is a vast literature in the $(1+N)$-body problem, see for instance the recent paper \cite{ChenHanLvHouYang2026CMDA} and references therein. Some of the fundamental works in this subject are due to Moeckel \cite{MoeckelJDDE1994, Moeckel1997JDDE}. When $N \gtrsim e^{73}$ and under the assumption that the $N$ masses tending to zero are equal, it was proved in \cite{Casasayas1994CentralConfigs} that the only central configuration of the $1+N$-body problem is the regular $N$-gon with the large body in its barycenter (in \cite{Hall1988} this had been proved for $N \gtrsim e^{27000}$); there is numerical evidence \cite{CorsLlibreOlle2004CMDA} that this is also true for $N \ge 9$. In \cite{RennerSicardy2004CMDA}, the existence of central configurations of the $(1+N)$-body problem, for given angular separation between the bodies, is studied. 

Another key property used in our arguments is that the centered $N$-gon is a non-degenerate central configuration for a large enough central mass. The Implicit Function Theorem then guarantees the existence of nearby central configurations under small perturbations of the masses in the ring. The non-degeneracy of this configuration was first analyzed in \cite{palmore1976}; subsequently Meyer and Schmidt \cite{MeyerSchmidt1988} (see also \cite[Section 6]{SchmidtRecifeLectures}) provided corrections to that paper. The non-degeneracy has also been addressed recently in \cite{zhouxia2025, zhouxia2026}. While the existence of nearby central configurations is established in the aforementioned works, in the present paper we are concerned with obtaining explicit expressions for the first order (in the masses) of the positions of the bodies in the configuration. To achieve this, ideas from \cite{MoeckelJDDE1994} play an essential role (see also \cite{Hall1988} and \cite{Casasayas1994CentralConfigs}).

Regarding the (non-centered) regular $N$-gon central configuration, Perko and Walter \cite{PerkoWalter1985PAMS} proved that placing $N$ bodies at the vertices of regular $N$-gon yields a central configuration if and only if all their masses are equal. Moreover, Chenciner \cite[Section~7]{Chenciner2004ConjectureCC} conjectured that it is the sole central configuration such that all the bodies lie on a circle and the center of mass coincides with the center of the circle. This has been proved affirmatively, apart from the classical case $N = 3$, for $N=4$ in \cite{Hampton2003Proceedings}, for $N=5$ in \cite{LlibreValls2015JDDE} and for $N\le 6$ in \cite{Zhiqiand2023JDE}. Moreover, for any $N \in \mathbb{N}$ and under the additional assumption that all the masses in the circle are equal, Cors, Hall and Roberts \cite{CorsHallRoberts2014PhysicaD} gave also a positive answer to the conjecture. Other related results can be found in \cite{TangZhu2024PhysicaD, CorberaValls2019JDDE}. In any case, by a result of Hampton \cite{Hampton2016CMDA}, the central configurations that we consider in Section~\ref{sec:family_CC} cannot provide a counterexample to Chenciner's conjecture.

For our construction of homoclinic tangencies of high order, the symmetries of both the regular $N$-gon and its continuation for small variations of the masses have a crucial role. We mention that in \cite{Montaldi2015CMDA} Montaldi shows the existence of central configurations for every possible symmetry type and for any symmetric choice of masses.

We will find central configurations with $N+1$ masses, with $N$ arbitrarily large, but fixed. We do not explore how the central configurations we obtain behave when $N$ goes to infinity, although our result suggests that the first order dependence on the masses of the positions in the central configuration has a well defined limit as $N\to \infty$.

Central configurations have several applications within celestial mechanics (see \cite{Saari1980CM}), we highlight their role in recent works in the analysis of collision dynamics \cite{Yu2024MAMS,MoeckelMontgomery2025JAMS,GierzkiewiczSchaeferZgliczynski2025CMP}, expanding parabolic orbits \cite{MadernaVenturelli2009ARMA} and non-integrability \cite{MaciejewskiPrzybylskaCombot2025JEMS}\footnote{In \cite{MaciejewskiPrzybylskaCombot2025JEMS}, the authors state that they use Moulton's collinear central configurations because of some of its known nice properties; it would be interesting to assess whether an analogous argument could be carried out for Maxwell's $N$-gon, using ideas from Section~\ref{sec:family_CC} of the present paper.}.

\subsection{Renormalization and Shilnikov maps.} Universality is obtained by the renormalization of a suitable map. This renormalization procedure, similar to the one in~\cite{GSTtangencies} (and also~\cite{Duarteabundance}) relies in the existence of a hyperbolic fixed point of a 2-d map whose invariant manifolds have tangencies. Roughly speaking, renormalization is a consequence of the local dynamics close to the hyperbolic point and the tangency. In our setting, where the map we consider is a suitable Poincaré map of a restricted planar $(n+1)$-body problem, the role of the hyperbolic point is played by a fixed point at which the map is tangent to the identity. However, it was already proven in~\cite{McGeheestablemfold}, although degenerate, the fixed point possesses stable and unstable manifolds that depend analytically on the values of the masses of the primaries. This fact and a $C^1$ control of the dynamics close to the fixed point was used in~\cite{Moserbook} to obtain oscillatory motions in the Sitnikov problem. It should be remarked that the standard $\lambda$-lemma~\cite{PalisdeMelo} does not hold in this situation. Consequently, one needs additional arguments to obtain a $C^k$ control of the dynamics close to the degenerate saddle. Here we use the generalization of the Shilnikov maps~\cite{Shilnikov67a} to this situation developed in~\cite{GMP2025} to obtain Newhouse domains in several instances of the restricted $n$-body problem. Here we generalize the renormalization argument in~\cite{GMP2025} to our setting.

\subsection{Detecting tangencies.} To detect  primary tangencies of arbitrary high order, it is necessary to perform a careful analysis of the invariant manifolds of some periodic orbit at infinity in the restricted planar circular $(n+1)$-body problem. This is done by considering a regime in which a suitable Melnikov function provides a \emph{first order} of the distance between the invariant manifolds. As a matter of fact, it is possible to apply the arguments in~\cite{GMS2016} to the present situation without modification to show that the Melnikov function is useful to detect tangencies of order 2 for any values of the masses in the central configuration, if some energy is large enough. However, for technical reasons, the arguments in~\cite{GMS2016} cannot ensure that this is the case for tangencies of higher order for any set of masses. Although it may be possible to extend those results to the present case, we use the approach in~\cite{MartinezPinyol,llibre1980oscillatory}, where,  taking large energy and using the masses of the primaries on the ring as parameters, while the central mass of the central configuration is taken large enough. In this situation, the Melnikov function allows us to obtain tangencies of arbitrary order.

\subsection{Organization of the article.} The rest of the article is organized as follows.

In Section~\ref{sec:main_results} we outline the proof of the main result and present the key intermediate theorems.

In Section~\ref{sec:family_CC} we provide the explicit first-order expansion (in the masses of the ring) of the positions of the bodies in the local continuation of the centered $N$-gon central configuration.

In Section~\ref{sec:RPCN2BP}, we consider the corresponding Restricted $(N+1)+1$ Body Problem where the $N+1$ primaries follow the homographic solutions corresponding to the continued central configurations, and study the motion of a massless particle which is affected by the gravitational attraction of the $N+1$ primaries but which does not affect them. Afterwards, using the classical Melnikov method, we analyze the splitting of the invariant manifolds of ``infinity''.

In Section~\ref{sec:tangencies}, we exploit the symmetries that our $(N+1)$-parametric family of periodic solutions for the $N+1$ primaries exhibits to obtain a tangency between the invariant manifolds of order $\ord(N)$, where $\{\ord(N) \mid N \in \mathbb{N}\}$ is unbounded.

In Section~\ref{uni_sec:renormalization}, we adapt the required renormalization scheme to the degenerate saddle at infinity. We remark that in \cite{GorodetskiK12}, Newhouse phenomenon for this same degenerate saddle was studied, but their method requires directly analyzing the $2$-jets of the equations in order to handle the case of a quadratic tangency. Because GST universality requires handling tangencies of arbitrarily high order, the methods developed in \cite{GMP2025} seem crucial, since they look more readily generalizable to tangencies of any order. The essential tools required from \cite{GMP2025} are collected in Appendix~\ref{uni_sec:Shilnikov}.

Once the renormalization scheme is implemented, a straightforward application of a theorem by Turaev \cite{Turaev03} yields our Main Result on universality in celestial mechanics, stated in Theorem~\ref{thm:uni_main_result}.

\section*{Acknowledgements}
M.\ G. has been partially supported by the Spanish Government grants PID2022-136613NB-I00 and PRE2020-096613; part of this work was completed while being part of Kaloshin Group at IST Austria, whose support is gratefully acknowledged. P. M. has been partially supported by the Spanish MINECO-FEDER Grant PID2024-158570NB-I00. 
This work is also supported by the Spanish State Research Agency, through the Severo Ochoa and Mar\'ia de Maeztu Program for Centers and Units of Excellence in R\&D (CEX2020-001084-M).

\section{Main results}\label{sec:main_results}




\subsection{A parametric family of planar central configurations}
\label{sec:configuacions_centrals_planes}


The planar restricted $(N+1)+1$-body problem that we will use to approximate any map
with prescribed accuracy is built on a periodic orbit of the $(N+1)$-body problem arising from a family of \emph{central configurations} depending on many parameters. The type of central configurations we are interested in are close to a regular $N$-gon, with masses equal to 1 in the vertices of the ring, and a large central mass (see Figure~\ref{fig:Ngon}). 
\begin{figure}[h]
\begin{center}
\includegraphics[width=.5\textwidth]{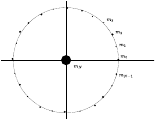}
\end{center}
\caption{The masses $m_j$, $j=0,\dots, N-1$, close to 1, at positions close to $e^{2\pi i j}$. The mass $m_N \gg 1$, close to the origin.}\label{fig:Ngon}
\end{figure}


In what follows, we will identify $\R^2$ with $\C$. In particular, given $q\in \C \equiv \R^2$, $\bar q$ will denote its complex conjugate and $|q|^2 = q \bar{q}$.

Let $N \ge 2$ and $\mm=(m_0,\ldots,m_N) \in \mathbb{R}_+^{N+1}$, a set of masses.
Let 
\begin{equation}\label{eq:center_mass_at_origin}
    M=\left\{ (q_0,\ldots,q_N) \in (\mathbb{R}^2)^{N+1} \mid \sum_{i=0}^{N} m_i q_i = 0 \right\} ,
\end{equation}
the set of configurations with the center of mass at the origin.

For $0 \le i < j \le N$, we define 
\[
\Delta_{i j} = \{ (q_0, \ldots, q_N) \in M \mid q_i = q_j \}
\]
and $\Delta = \bigcup_{i < j} \Delta_{i j}$, the collision manifold.

The \emph{potential function} of the $(N+1)$-body problem $V_{\mm} \colon M \setminus \Delta \to \mathbb{R}$ is defined by
\begin{equation}\label{eq:def_potential_N_body}
   V_{\mm}(q_0, \ldots, q_N) =  \sum_{ 0 \le i < j \le N} \frac{m_i m_j}{| q_i - q_j |}.
\end{equation}
We define the \emph{moment of inertia} as 
\begin{equation}\label{eq:def_moment_inertia}
    I_{\mm}(q_0, \ldots,q_N)= \sum_{i=0}^{N} m_i | q_i |^2 = \sum_{i=0}^{N} m_i  q_i \bar{q}_i.
\end{equation}
We consider the $N$ level set of $I_{\mm}$,
\begin{equation}\label{eq:normalization_moment_inertia}
    S_{\mm} = \{ (q_0, \ldots, q_N) \in M \mid I_{\mm}(q_0, \ldots,q_N) = N\} 
\end{equation}
and write $S_{\mm} \setminus (S_{\mm} \cap \Delta)$ as $S_{\mm} \setminus \Delta$. 

\begin{defn}\label{def:CC_critical_point}
$(q_0, \ldots, q_N) \in S_{\mm} \setminus \Delta$ is a \emph{central configuration} if and only if it satisfies the following system of algebraic equations
    \begin{equation}
    \label{def:sistemaCC}
    \partial_{q_j} V_{\mm} (q_0,\dots,q_N) = \frac{\lambda m_j}{2} \, \overline{q_j} ,\qquad j=0,\dots, N,
    \end{equation}
for some $\lambda \in \mathbb{R}$, that is, if it is a critical point of $V_{\mm}$ restricted to $S_{\mm} \setminus \Delta$.
\end{defn}

It is well known that, if $(q_0,\dots,q_N)\in S_{\mm}$ is a central configuration, then 
\begin{equation}
    \label{def:lambda_cc}
    \lambda =  -\frac{V_{\mm}(q_0,\dots,q_N)}{I_{\mm}(q_0,\dots,q_N)} = - 
    \frac{V_{\mm}(q_0,\dots,q_N)}{N}.
\end{equation}

\begin{rem} 
\label{rem:periode_solucio_circular}
    Given $(q_0,\ldots,q_N)$ a planar central configuration of $N+1$ bodies, then the functions
    \begin{equation}
    \label{def:Qj}
    Q_j(t)=e^{i \sqrt{-\lambda} \, t} \, q_j, \quad j=0,\ldots,N,
    \end{equation}
    define a solution of the corresponding $(N+1)$-Body Problem.
\end{rem}





Since we want to consider central configurations close to the polygonal one, with $N$ equal masses on a ring, we introduce $\mu_j$ by setting $m_j = 1+ \mu_j$, $j=0,\dots,N-1$, for the bodies in the ring. We emphasize that, while $m_j >0$, $\mu_j$ can take negative values. We will write $\mu = (\mu_0,\dots,\mu_{N-1})^\top$. 


Next theorem, whose proof is placed in Section~\ref{sec:family_CC}, provides the existence of a central configuration for each value of $\mu$ and $m_n$ and its explicit first-order dependence of the positions $q_j$ on the the values masses.

\begin{thm}
\label{thm:posicions_respecte_a_les_masses_s2}
For any $N\ge 3$, there exist $m^*,  \mu^* >0$ such that, for all $\|\mu\| < \mu^*$ and $m > m^*$, there exists a unique solution $q(\mu,m) \in S_{\mm} \setminus \Delta$ of~\eqref{def:sistemaCC} with masses $m_j = 1+\mu_j$, $j=0,\dots, N-1$, $m_N = m$, with $\lambda$ given by~\eqref{def:lambda_cc}, analytic in $\mu$ and $\varepsilon= 1/m$,  of the form
\[
q_j(\mu,\varepsilon) = \rho_j(\mu,\varepsilon) e^{i\theta_j(\mu,\varepsilon)}, \qquad j = 0,\dots, N-1,
\]
with $\theta_0(\mu,\varepsilon) = 0$, where, with the normalizing condition $\sum_{\ell= 0}^{N-1} \mu_\ell = 0$, 
\[
\begin{aligned}
\rho_j(\mu,\varepsilon) & =  1 + \OO(\mu) \OO(\varepsilon), &   j & = 0,\dots, N-1, \\ 
\theta_j(\mu,\varepsilon) & = \frac{2\pi j}{N} + \sum_{\ell=0}^{N-1} \left(\frac{1}{N} \sum_{k=1}^{N-1} c_k (\omega^{k(j-\ell)} - \omega^{-k \ell}) \right) \mu_\ell + \mathcal{O}(\varepsilon) + \mathcal{O}_2(\mu, \varepsilon) ,  & j & = 1,\dots, N-1,
\end{aligned}
\]
$c_k$ are some explicit constants that depend only on $N$, given in Theorem~\ref{thm:posicions_respecte_a_les_masses}, and
\[
q_N(\mu,\varepsilon) = - \varepsilon \sum_{\ell=0}^{N-1} (1+\mu_\ell) q_\ell(\mu,\varepsilon) = \OO(\mu) \OO(\varepsilon).
\]

Moreover, if the masses satisfy the symmetry condition $\mu_{j}=\mu_{N-j}$ for $j=1,\ldots,N-1$, then the corresponding central configuration is symmetric with respect to the real axis, with positions satisfying $q_j=\overline{q_{N-j}}$ for $j=1,\ldots,N-1$ and $q_0, q_N \in \mathbb{R}$.
\end{thm}

We fix $\theta_0 = 0$ and take the $(q_0,\dots,q_N)\in S_{\mm}$ to get rid of freedom of central configurations with respect to  rotations and scalings of the positions. Once these are fixed, we exploit the fact that the $N$-gon configuration is non-degenerate if $m\gg 1$ (see~\cite{palmore1976,MeyerSchmidt1988}). 

As a corollary of Theorem~\ref{thm:posicions_respecte_a_les_masses}, we have that, if $\varepsilon \to 0$, that is, if the central mass $m \to \infty$, the bodies in the ring tend to  the unit circle while the central mass tends to the origin. This was proven in~\cite{MoeckelJDDE1994}), assuming that there was a limit central configuration when $m \to \infty$. Theorem~\ref{thm:posicions_respecte_a_les_masses} ensures such limit exists.

The novelty and importance of Theorem~\ref{thm:posicions_respecte_a_les_masses_s2} is twofold. On the one hand, the parameters $\mu^*$ and $\varepsilon^*$ are fixed and independent. This will allow as to take $\varepsilon = 1/m$ as small as necessary while keeping $\mu$ in a fixed ball. On the other, it provides the explicit first order dependence on $\mu$, given by the coefficients $c_k$. They  are given in Lemmas~\ref{lem:DmutildeG} and~\ref{lem:DthetatildeG_rang}. This first order will appear later in the proof in an essential way.

\subsection{The Restricted Planar Circular $(N+1)+1$-Body Problem}

Let $(q_0(\mu,\varepsilon),\dots,q_N(\mu,\varepsilon))$ be the positions of the central configuration corresponding to masses $\mm = (1+\mu_0,\dots,1+\mu_{N-1},1/\varepsilon)$ given by Theorem~\ref{thm:posicions_respecte_a_les_masses_s2}. We will take  $\varepsilon$ and $\mu$ as small parameters.

In view of Remark~\ref{rem:periode_solucio_circular}, the period of the solution of the $(N+1)$-body problem with masses $\mm = (1+\mu_0,\dots,1+\mu_{N-1},1/\varepsilon)$ is $2\pi/\sqrt{-\lambda}$, where $\lambda$, given in~\eqref{def:lambda_cc}. Hence, it depends on $(\mu,\varepsilon)$.
We rescale the masses and positions in such a way that the chosen central configuration has total mass 1 and the corresponding circular solution of the $(N+1)$-body problem has period~$2\pi$. The following proposition collects the appropriate the scalings.

\begin{prop}
\label{prop:solucio_circular_periode_2pi}
Let $(q_0(\mu,\varepsilon),\dots,q_N(\mu,\varepsilon))$ be the positions of the central configuration corresponding to masses $\mm = (1+\mu_0,\dots,1+\mu_{N-1},1/\varepsilon)$ given by Theorem~\ref{thm:posicions_respecte_a_les_masses_s2}.
Let $\tmm = \mm/M$ and $\tilde{q}_j(\mu,\varepsilon) = M^{-1/3} \SA(M \tmm, N) q_j(\mu,\varepsilon)$, $j=0,\dots,N$, where $M = \sum_{j=0}^N m_j = N+1/\varepsilon$ and
\begin{equation}\label{def:SA}
    \SA(\mm, N) 
    = \left(\frac{1}{N} \sum_{0 \le j < k \le N} \frac{m_j m_k}{|q_j(\mu,\varepsilon)-q_k(\mu,\varepsilon)|} \right)^{1/3}.
\end{equation}
Then
\begin{equation}
\label{def:solucio_circular}
e^{it}(\tilde{q}_0(\mu,\varepsilon),\dots,\tilde{q}_N(\mu,\varepsilon))
 \end{equation}
is a solution of the $(N+1)$-body problem with potential $V_{\tmm}$ in~\eqref{eq:def_potential_N_body}.
\end{prop}

\begin{proof}
Since, by Theorem~\ref{thm:posicions_respecte_a_les_masses_s2}, $(q_0(\mu,\varepsilon),\dots,q_N(\mu,\varepsilon))$ is a central configuration corresponding to masses $\mm=(m_0,\ldots,m_N)= (1+\mu_0,\dots,1+\mu_{N-1},1/\varepsilon)$ with $\lambda $ given by~\eqref{def:lambda_cc}, for any $\alpha, \beta >0$, $(\alpha q_0(\mu,\varepsilon),\dots,\alpha q_N(\mu,\varepsilon))$
is a central configuration corresponding to masses $\beta\mm$ with $\tilde \lambda = \lambda \beta/\alpha^3$. The circular solution corresponding to this central configuration has period $2\pi/\sqrt{-\tilde \lambda}$. Hence, choosing $\beta = 1/M$ and using the value of $\lambda$ in~\eqref{def:lambda_cc}, we have that $\alpha = \SA(\mm, N)$.
\end{proof}

Notice that, since $\varepsilon = 1/m_N$ and $\sum_{j=0}^{N-1} \mu_j = 0$, we have that
\begin{equation}
    \label{relacions_masses_tilde_i_epsilon}
\begin{aligned}
    \tilde m_j & = \frac{\varepsilon}{1+N\varepsilon} (1+\mu_j), \qquad j =0,\dots N-1, \\
    \tilde m_N & = \frac{1}{1+N\varepsilon}.
\end{aligned}
\end{equation}
Moreover, since $N$ is fixed, $\|\mu\| < \mu^*$,  and $M = (1+N\varepsilon)/\varepsilon$, Theorem~\ref{thm:posicions_respecte_a_les_masses} implies by a straightforward computation that
\begin{equation}
    \label{expansio_en_epsilon_R}
    M^{-1/3} \SA(\mm, N) = 1+\mathcal{O}(\varepsilon).
\end{equation}
In particular,  since
\[
\tilde{q}_N = - \varepsilon \sum_{j=0}^N m_j {\tilde{q}_j}
\]
and $\tilde{q}_{j \mid \varepsilon = 0} = q_j(\mu,\varepsilon)_{\mid \varepsilon = 0}$ only depends on $\mu$, analytically around $\mu = 0$, we have that $\tilde{q}_{N \mid \varepsilon = 0} = 0$.

Hence, the position and momentum $(q,p)\in \R^2\times \R^2$  of a massless body under the gravitational attraction of the primaries with mass $\tmm = (\tilde m_0,\dots, \tilde m_N)$ describing the circular solution of the $(N+1)$-body problem~\eqref{def:solucio_circular} is given by the Hamiltonian 
\begin{equation}\label{eq:Hamiltonian_ring_cart}
    H_{\mathrm{cart}}(q,p,t)= \frac{|p|^2}{2} - U_{\mathrm{cart}}(q,t;\mu,\varepsilon) ,
\end{equation}
 and 
\begin{equation}
\label{def:Ucart}
U_{\mathrm{cart}}(q,t;\mu,\varepsilon) = \sum_{j=0}^{N} \frac{\tilde{m}_j}{|q-e^{it} \tilde{q}_j(\mu,\varepsilon)|} .
\end{equation}
Introducing polar coordinates in the plane, $q=r e^{i \alpha}$, and $p = y \, e^{i \alpha} + i \, \frac{G}{r} \, e^{i \alpha}$, their conjugate momenta, 
Hamiltonian $H_{\mathrm{cart}}$ in~\eqref{eq:Hamiltonian_ring_cart} becomes
\begin{equation}\label{eq:Hamiltonian_ring_polar}
    H_{\mathrm{pol}}(r,y,\alpha,G, t; \mu,\varepsilon)= \frac{y^2}{2} + \frac{G^2}{2 r^2} - U_{\mathrm{pol}}(r,\alpha- t;\mu,\varepsilon) ,
\end{equation}
%
%
where
\begin{multline}\label{eq:potential_ring_polar}
    U_{\mathrm{pol}}(r,\phi;\mu,\varepsilon) = U_{\mathrm{cart}}(r e^{i(\phi+t)},t;\mu,\varepsilon) \\
    = \sum_{j=0}^{N} \frac{\tilde{m}_j}{|r e^{i (\phi+t)}-\tilde{q}_j(\mu,\varepsilon) e^{i t}|} = \sum_{j=0}^{N} \frac{\tilde{m}_j}{r|1-\frac{\tilde{q}_j(\mu,\varepsilon)}{r} e^{-i \phi}|} ,
\end{multline}
with $\tilde m_j$ given by~\eqref{relacions_masses_tilde_i_epsilon},
 is analytic on $(\mu,\varepsilon)$ in a neighborhood of $(0,0)$. Clearly, using~\eqref{relacions_masses_tilde_i_epsilon} and $\tilde{q}_{N \mid \varepsilon = 0} = 0$, we have that ${U_{\mathrm{pol}}}_{\mid \varepsilon = 0} = 1/r$. This fact, in combination with the standard expansion in Legendre polynomials (the center of mass of the central configuration is at the origin), gives
\begin{equation}
\label{expansio_U_pol}
U_{\mathrm{pol}}(r,\phi;\mu,\varepsilon) = \frac{1}{r}+\varepsilon \mathcal{O}\left( \frac{1}{r^3}\right).
\end{equation}

Although $H_{\mathrm{pol}}$ is a $2 \frac{1}{2}$ degrees of freedom Hamiltonian, we emphasize it depends on $(\alpha, t)$ only through $\phi=\alpha-t$. Hence, by using the new angle $\phi$, that is, writing $H_{\mathrm{pol}}$ in rotating coordinates, we rewrite $H_{\mathrm{pol}}$ as the 2~degrees of freedom Hamiltonian
\begin{equation}\label{eq:Hamiltonian_ring_polar_rot}
    H_{\mathrm{rot}}(r,y,\phi,G; \mu,\varepsilon)= \frac{y^2}{2} + \frac{G^2}{2 r^2} - U_{\mathrm{pol}}(r,\phi;\mu,\varepsilon) -G.
\end{equation}

Thanks to the last claim of Theorem~\ref{thm:posicions_respecte_a_les_masses_s2}, we will choose the central configuration to be symmetrical with respect to the real axis. This implies that the equations of motion of the RPC(N+2)BP are reversible with respect to the involution 
\[
(r,y,\phi, G) \mapsto (r,-y,-\phi,G) ,
\]
which translates into a symmetry in the parameterizations of the invariant manifolds $W_\infty^{s,u}$.

\subsection{Invariant manifolds of infinity and parameter dependence}

In view of~\eqref{expansio_U_pol}, the equations of motion associated to Hamiltonian $H_{\mathrm{rot}}$ in~\eqref{eq:Hamiltonian_ring_polar_rot} are
\begin{equation}\label{eq:equations_motion_rotating}
    \left\{\begin{aligned}
        \dot r &= y \\
        \dot \phi &= -1+\frac{G}{r^2} \\
        \dot y &= -\frac{1}{r^2}+\frac{G^2}{r^3} + \varepsilon \mathcal{O}\left( \frac{1}{r^4}\right)\\
        \dot  G &= \varepsilon \mathcal{O}\left( \frac{1}{r^3}\right)
    \end{aligned}\right.
\end{equation}
From the above equations, it is clear that, for any $G_0 \in \R$, the set
\[
\gamma_{G_0} =\{(r,y,\phi,G) \mid \; r = +\infty,\; y = 0,\; G = G_0\}
\]
is a periodic orbit (at infinity), with period $2\pi$. Hence, $\tilde{\Gamma}_{\infty} = \cup_{G_0 \in \R} \gamma_{G_0}$ is an invariant cylinder at infinity, foliated with periodic orbits. It is often called the \emph{parabolic} infinity, since $y_{\mid \tilde{\Gamma}_{\infty}} = 0$. In view of~\eqref{eq:Hamiltonian_ring_polar_rot} and~\eqref{expansio_U_pol}, it  is clear that
\[
\tilde{\Gamma}_{\infty} \cap \{H_{\mathrm{rot}} = - G_0\} = \gamma_{G_0}.
\]

Let $\Phi_{t}=(\Phi_{t}^r, \Phi_{t}^y, \Phi_{t}^\phi, \Phi_{t}^G)$ denote the flow associated to the Hamiltonian system corresponding to Hamiltonian $H_{\mathrm{rot}}$ in~\eqref{eq:Hamiltonian_ring_polar_rot}. It was proven in~\cite{McGeheestablemfold}  that the sets
\[
W_{\infty,G_0}^{s,u}= \left\{(r,y,\phi, G) \in \mathbb{R}^+ \times \mathbb{R} \times \mathbb{T} \times \mathbb{R} \colon \lim_{t \to \pm \infty} (\Phi_{t}^r,\Phi_{t}^y,\Phi_{t}^G)(r,y,\phi,G)= (\infty, 0,G_0) \right\}.
\]
are two dimensional analytic manifolds, which are called the stable and unstable manifolds of $\gamma_{G_0}$, respectively. They depend analytically on $\mu$, $\varepsilon$ and $G_0$. Since $\gamma_{G_0} \subset \{H_{\mathrm{pol}} = - G_0\}$, $W_{\infty,G_0}^{s,u}\subset \{H_{\mathrm{pol}} = - G_0\}$. The sets 
\[
W_\infty^{s,u} = \cup_{G_0 \in \R} W_{\infty,G_0}^{s,u}
\]
are the invariant manifolds of $\tilde{\Gamma}_{\infty}$.

Our aim is to find homoclinic tangencies between $W_{\infty,G_0}^{s,u}$ when varying the parameters of the system, that is, $\mu$, the values of the masses of the ring, $\varepsilon$, the value of the central mass, and $G_0$, the energy level. In our construction, unlike~\cite{GMS2016}, $\varepsilon$ and $G_0$ will not be independent: we will fix $G_0$ large enough, then choose accordingly $\varepsilon$ small enough. The fact that, by Theorem~\ref{thm:posicions_respecte_a_les_masses}, the mass parameters $\mu$ can be chosen independently of $\varepsilon$ will allow us to choose them in some appropriate ball. To do so, we will use the standard Melnikov approach, which will be valid in our setting because $\varepsilon$, unlike $G_0$ is a regular parameter of the problem. This fact will be clear in Section~\ref{sec:limit_integrable}. 

In order to make explicit the dependence on the energy level and to normalize the size of the ``homoclinic loop'' of the integrable limit (as explained in Section~\ref{sec:limit_integrable}, below),  we perform the following scaling.
Given $G_0>0$, we rescale the variables by
\begin{equation}
\label{escalat_G0}
r = G_0^2 \tilde r, \qquad y = G_0^{-1} \tilde y, \qquad \phi = \tilde \phi, \qquad G = G_0 \tilde G, 
\end{equation}
which is a conformally symplectic map with multiplier $G_0$. Clearly
\[
H_{\mathrm{rot}}(G_0^2 \tilde r,G_0^{-1} \tilde y,\tilde \phi,G_0 \tilde G; \mu,\varepsilon) = G_0^{-2} \mathcal{H}(\tilde r, \tilde y,\tilde \alpha, \tilde G;\mu,\varepsilon, G_0)
\]
where
\begin{equation}
\label{def:Hamiltonian_rescaled_rotating}
\mathcal{H}(\tilde r, \tilde y,\tilde \phi, \tilde G;\mu,\varepsilon, G_0) = \frac{\tilde y^2}{2} + \frac{\tilde G^2}{2 \tilde r^2} - \frac{1}{\tilde r} - G_0^3 \tilde G - \varepsilon V(\tilde r, \tilde \phi;\mu,\varepsilon, G_0)
\end{equation}
and
\begin{equation}\label{def:V_N2}
   \varepsilon V(\tilde r, \tilde \phi; \mu,\varepsilon, G_0) = - \frac{1}{\tilde r} + G_0^2 U_{\mathrm{pol}}(G_0^2 \tilde r, \tilde \phi; \mu,\varepsilon) = - \frac{1}{\tilde r} + \sum_{j=0}^{N} \frac{\tilde{m}_j}{\tilde r|1-\frac{\tilde{q}_j(\mu,\varepsilon)}{G_0^2 \tilde r} e^{-i \tilde \phi}|} .
\end{equation}
We recall that the mass parameters $\tmm$ are a function of $\varepsilon$ and $\mu$ through~\eqref{relacions_masses_tilde_i_epsilon}. Since ${U_{\mathrm{pol}}}_{\mid \varepsilon = 0} = 1/r$, Schwarz lemma ensures  that $\varepsilon$ is indeed a common factor in~\eqref{def:V_N2}. 

In these rescaled variables, the periodic orbit at infinity, $\gamma_{G_0}$ (which we denote by the same letter) is now
\begin{equation}
\label{gamma_G0_reescalada}
\gamma_{G_0} = \{(\tilde r,\tilde y,\tilde \phi,\tilde G) \mid \; \tilde r = +\infty,\; \tilde y = 0,\; \tilde G = 1\} \subset \{\mathcal{H} = -G_0^3\}.
\end{equation}
We will denote their invariant manifolds again by $W_{\infty,G_0}^{s,u}$. They satisfy $W_{\infty,G_0}^{s,u} \subset \{\mathcal{H} = -G_0^3\}$.


From now on, we will drop the tildes in the variables.

\subsection{The Kepler problem as the integrable limit}\label{sec:Kepler_approx_N2BP}
\label{sec:limit_integrable}
We write Hamiltonian $\mathcal{H}$ in~\eqref{def:Hamiltonian_rescaled_rotating} as
\[
\mathcal{H} = \mathcal{H}_0+\varepsilon V
\]
where $\mathcal{H}_0 = \mathcal{H}_{\mid \varepsilon = 0}$ is the  Hamiltonian corresponding to the Kepler problem in rotating coordinates,
\begin{equation}
	\label{def:Kepler_hamiltonian}
	\mathcal{H}_0(r,\phi, y, G;  G_0)  = \frac{y^2}{2} + \frac{G^2}{2 r^2}  - \frac{1}{r} - G_0^3 G,
\end{equation}
and $\varepsilon V$ is given in~\eqref{def:V_N2}.

Since $\mathcal{H}_0$ does not depend on $\phi$, the angular momentum $G$ is a conserved quantity. Hence, $\mathcal{H}_0$ has 2 independent first integrals in involution, and so is integrable. Therefore, for the Kepler problem, the invariant manifolds of $\gamma_{G_0}$ coincide along a 2-dimensional homoclinic manifold $W^0_{G_0} = W_{\infty, G_0 \mid \varepsilon = 0}^{s}= W_{\infty, G_0 \mid \varepsilon = 0}^{u} \subset \{\mathcal{H}_0 = -G_0^3\}$, which is given by
\begin{equation}\label{eq:inv_manif_2d_Kepler}
    W^0_{G_0} = \left\{(r,y,\phi,G)\mid \frac{y^2}{2} + \frac{G^2}{2 r^2}  - \frac{1}{r} = 0, \; G = 1\right\}. 
\end{equation}

\begin{figure}[h]
\begin{center}
\includegraphics[width=.5\textwidth]{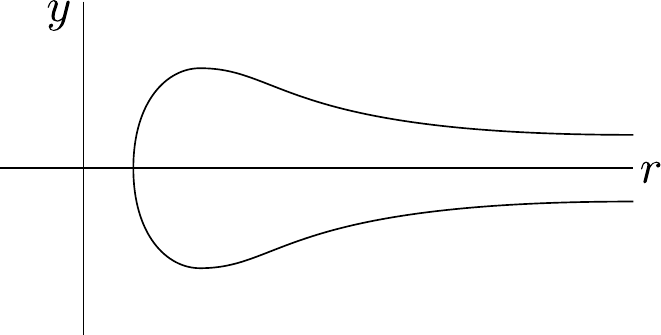}
\end{center}
\caption{The homoclinic manifold $W^0_{G_0}$ intersected with a section $\{\phi=\phi_0\}$.}\label{fig:separatrix}
\end{figure}
Now, the purpose of the scaling~\eqref{escalat_G0} becomes clear. In the rescaled variables, the homoclinic loop to the parabolic infinity at energy level $-G_0^3$ is independent of $G_0$. 
To use the Melnikov approach, as in~\cite{MartinezPinyol,llibre1980oscillatory}, we will need to take $G_0 \gg 1$.  A suitable parametrization of $W^0_{G_0}$ is obtained as follows.


The equations of motion corresponding to $\mathcal{H}_0$ in~\eqref{def:Kepler_hamiltonian} are 
\begin{equation}\label{eq:equations_motion_Kepler_rescaled_rotating}
    \left\{
\begin{aligned}
	\dot{r} &= y, \\
    \dot{\phi} &= \frac{G}{r^2} - G_0^3 ,\\
	\dot{y} &= \frac{G^2}{r^3} - \frac{1}{r^2}, \\
	\dot{G} &= 0 .\\
\end{aligned}\right.
\end{equation}
Notice that $\mathcal{H}_0$ is singular in the regime $G_0 \gg 1$ since, then, $\phi$ becomes a fast angle.

Let $(r_h(t),y_h(t),\phi_h(t),G_h(t))$ be the solutions of~\eqref{eq:equations_motion_Kepler_rescaled_rotating} satisfying
\[
y_h(0) = \phi_h(0) = 0, \qquad G_h(0) = 1,
\]
and $r_h(0)$ given by
\[
\mathcal{H}_0(r_h(0),y_h(0),\phi_h(0),G_h(0)) = -G_0^3,
\]
that is, $r_h(0) = 1/2$. In view of~\eqref{eq:equations_motion_Kepler_rescaled_rotating},
\[
\phi_h(t) = \alpha_h(t) - G_0^3 t,
\]
where $\dot \alpha_h = 1/r_h^2$. It belongs to $W^0_{G_0}$. It is clear that
\begin{equation}
	\Gamma_0(u,\beta) =
	\begin{pmatrix}
		r_h(u) \\\beta+ \alpha_h(u) - G_0^3 u\\ y_h(u) \\ 1 
	\end{pmatrix}, \qquad (u,\beta) \in \R\times \T .
\end{equation}
is a parametrization of $W^0_{G_0}$. It satisfies $\Phi_{t\mid \varepsilon = 0} \circ \Gamma_0 (u,\beta) = \Gamma_0(u+t,\beta)$.

The functions $r_h$, $y_h$, $\alpha_h$ are known. Next lemma, from~\cite{MartinezPinyol}, provides explicit expressions for them.
\begin{lem}
	\label{lem:rhyhalphah}
    Let $\tau(v)$ be the unique analytic function defined by
    \[
    v = \frac{1}{2} \left( \frac{1}{3} \tau^3 +\tau\right)
    \]
    such that $\tau$ is real for real values of $v$. Then
    \begin{enumerate}
        \item $r_h(v) = \tilde r_h(\tau(v))$, where
        \[
        \tilde \tau_h(\tau) = \frac{1}{2}(\tau^2+1),
        \]
        \item $\alpha_h(v) = \tilde \alpha_h(\tau(v))$, where
        \[
        \tilde \alpha_h(\tau) = 2 \arctan (\tau).
        \]
    \end{enumerate}
	
\end{lem}

\subsection{Distance between the invariant manifolds of the parabolic infinity}

It is clear from~\eqref{eq:equations_motion_Kepler_rescaled_rotating} that, if $G_0$ is large enough, $\dot \phi <0$ in a neighborhood of $W_{\infty,G_0}^{s,u}$. Hence, the restriction of the flow of $\mathcal{H}$ at the energy level $\{\mathcal{H} = -G_0^3\}$ induces a well defined Poincar\'e
map 
\begin{equation}
\label{def:mapa_de_Poincare}
\mathcal{P}_{G_0,\phi_0}:\{\phi = \phi_0\}\mapsto \{\phi = \phi_0+2\pi\},
\end{equation}
with natural coordinates $(r,y) \mapsto \mathcal{P}_{G_0,\phi_0}(r,y)$.
The intersection of the stable and unstable invariant manifolds of $\gamma_{G_0}$, with $y>0$, with the section $\{\phi = \phi_0\}$
can be parametrized as
\[
\begin{aligned}
r & = r_h(v) \\
y & = Y_{\phi_0}^{u,s}(v,\mu,\varepsilon,G_0).
\end{aligned}
\]
The following theorem, whose proof is placed in Section~\ref{sec:prova_distancia}, provides the difference between these invariant manifolds.
\begin{thm}
\label{thm:Melnikov_domina}
    Assume $G_0>1$. Then 
    \[
   Y_{\phi_0}^{u}(v,\mu,\varepsilon,G_0)- Y_{\phi_0}^{s}(v,\mu,\varepsilon,G_0) = - \frac{\varepsilon}{y_h(v)} \left(\left(G_0^3-\frac{1}{r_h(v)^2}\right)\partial_\beta L_{\mid G_0^3v+\phi_0-\alpha_h(v)}+ \mathcal{O}(\varepsilon)\right),
    \]
    where 
    \begin{equation}
    \label{def:Melnikov}
    L(\beta; \mu,\varepsilon, G_0) = \int_{-\infty}^{+\infty} V(r_h(s), \beta -G_0^3 s + \alpha_h(s); \mu,\varepsilon, G_0) ds
    \end{equation}
    is the \emph{Melnikov potential} associated to $V$.
\end{thm}

In Section~\ref{sec:RPCN2BP} we will prove Theorem~\ref{thm:Melnikov_domina} and we will give an explicit expression of the Melnikov potential $L$ (as a function of $G_0$ and $\mu$, see Propositions~\ref{prop:Fourier_expansion_Melnikov_RPCN2BP} and~\ref{prop:expansion_Melnikov_positions}), necessary to obtain next theorem.

\begin{thm}
\label{thm:tangencies_sec2}
    Let $\ord \in \mathbb{N}$ arbitrary. Then there exist $N \in \mathbb{N}$ such that, for any $G_0 \gg 1$, there exist values $\mu^\ast, \varepsilon^\ast$ for which the one-dimensional invariant stable and unstable invariant manifolds of the fixed point $\gamma_{G_0}^{\ast}$ have a tangency of order $\ord$, at some point $z^{\ast}=(r^\ast, y^\ast)$, which splits non-degenerately in $\mu$.
\end{thm}

This theorem is a rewording of Theorem~\ref{thm:tangencies}, in Section~\ref{sec:tangencies}, proven by varying the parameters $\mu$ to appropriately tune the perturbative potential $V$. 

\begin{figure}[h]
\begin{center}
\includegraphics[width=.7\textwidth]{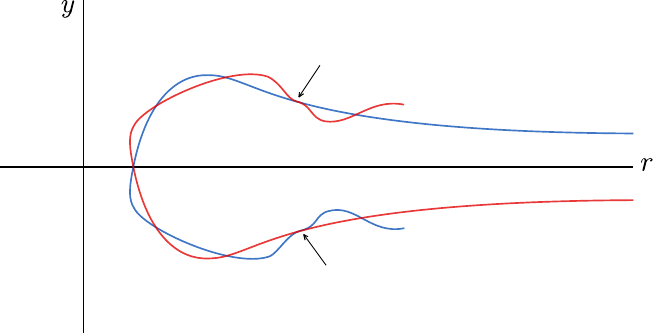}
\end{center}
\caption{Cubic homoclinic tangency (i.e., of order $\ord=2$) between the one-dimensional invariant manifolds ${W_{\infty,G_0}^{s,u} \cap \{\phi=\phi_0\}}$.}\label{fig:cubic_tangency}
\end{figure}

\subsection{Renormalization of the Poincar\'e map}

In Section~\ref{uni_sec:renormalization}, we extend the renormalization scheme for quadratic tangencies in the parabolic setting obtained in \cite{GMP2025} to tangencies of higher order, obtaining (see Proposition~\ref{uni_prop:mainresultrenormalization}):

\begin{thm}
\label{thm:renormalitzacio}
    Under the conditions stated above, there exists $L \in \mathbb{N}$ such that, for any $T \in \mathbb{N}$ sufficiently large, there exist:
    \begin{itemize}
        \item An affine change of parameters $\varphi_{T}:[-2,2]^\ord \to\mathbb{R}^\ord$, $\varphi_T \colon \kappa \mapsto \mu$, with $\kappa=(\kappa_0,\kappa_1,\ldots,\kappa_{\ord-1})$,
        \item A real-analytic change of variables $\Psi_{\kappa, T}\colon B_{10}\to \mathbb{C}^2$, defined in a complex ball of radius $10$ centered at the origin $B_{10} \subset \mathbb{C}^2$,
    \end{itemize}
    such that, as $T\to\infty$,
    \[
    \varphi_{T}([-2,2]^\ord)\to \mu^\ast,\qquad\qquad \Psi_{\kappa,T}(B_{10})\to z^\ast,
    \] 
    and,  uniformly for $(Q,P)\in B_{10}$, 
    \[
    \Psi_{\kappa,T}^{-1}\circ \mathcal{P}_{\varphi_{T}(\kappa)}^{L+T}\circ \Psi_{\kappa,T}(Q,P)=H_{g_\kappa}(Q,P)+\mathcal{O}(T^{-b}),
    \]
    for some positive real number $b>0$, where $H_{g_\kappa}\colon (Q,P) \mapsto (P, -Q+g_\kappa(P))$ is the Hénon-like map corresponding to the polynomial $g_\kappa(P)=\kappa_0+\kappa_1 P + \cdots + \kappa_{\ord-1} P^{\ord-1} -P^{\ord+1}$.
\end{thm}

Before stating the main result of the paper, we briefly outline the Gonchenko--Shilnikov--Turaev universality theory, following the exposition in \cite{BergerFlorioPeraltaSalas}. It is also worth recalling the definitions of primitive renormalization and universal map from the Introduction (Definitions \ref{def:primitive_renormalization} and \ref{def:universal_diffeo}, respectively). In what follows, let $r\in\{2,\ldots,\infty,\omega\}$.

Let $C^{r}_{\mathrm{Symp}}$ be the set of two-dimensional symplectic diffeomorphisms of class $C^{r}$.

Let $\mathbb{D}$ be the closed unit disk of $\mathbb{R}^2$ and let $\mathrm{Diff}^{r}_{\mathrm{Symp}}(\mathbb{D}, \mathbb{R}^2)$ be the set of symplectic diffeomorphisms of class $C^{r}$ from $\mathbb{D}$ onto their image in $\mathbb{R}^2$ .

Let $\mathbb{I} \subset \mathbb{R}$ denote a compact interval. For $J \ge 0$, let $\mathrm{Diff}^{r}_{\mathrm{Symp}}(\mathbb{D}, \mathbb{R}^2)_{\mathbb{I}^J}$ denote the space of $C^{r}$ families $(f_p)_{p \in \mathbb{I}^J}$ of maps $f_p \in \mathrm{Diff}^{r}_{\mathrm{Symp}}(\mathbb{D}, \mathbb{R}^2)$.

The following result is due to Turaev. It asserts that renormalizations of Hénon maps are dense in $\mathrm{Diff}^{r}_{\mathrm{Symp}}(\mathbb{D}, \mathbb{R}^2)$ and likewise for the families. See also its reformulation in \cite[Theorem~4.9]{BergerFlorioPeraltaSalas}.

\begin{thm}[{\cite[Theorem~1]{Turaev03}, \cite[Lemma~9]{GSTtangencies}}]\label{thm:Turaev03}
Let $r \in \{2,\ldots,\infty,\omega\}$. For every $J \ge 0$ and every non-empty open subset $\mathcal{O}_J \subset \mathrm{Diff}^{r}_{\mathrm{Symp}}(\mathbb{D}, \mathbb{R}^2)_{\mathbb{I}^J}$, there exist $k > 0$, a $C_{\mathrm{Symp}}^{r}$ family of polynomials $(g_p)_{p \in \mathbb{I}^J}$ of degree $< k$, an iteration $N \ge 1$ and an affine change of coordinates $\phi$ of $\mathbb{R}^2$ such that $(\phi^{-1} \circ H_{g_p}^N \circ \phi|_{\mathbb{D}})_{p \in \mathbb{I}^J}$ is in $\mathcal{O}_J$, where $H_{g_p}$ is the Hénon-like map $H_{g_p}\colon (Q, P) \mapsto (P, -Q+g_p(P))$.
\end{thm}

We can now state our main result, which is a immediate consequence of  Theorem~\ref{thm:renormalitzacio} together with Theorem~\ref{thm:Turaev03}.
\begin{thm}[Main Result]\label{thm:uni_main_result}
Let $r \in \{2,\ldots,\infty, \omega\}$. Let $\mathfrak{Y}$ be the set consisting of all Poincaré maps $\mathcal{P}$ corresponding to the RPC(N+2)BP (see \eqref{def:mapa_de_Poincare}) for all $N \in \mathbb{N}$, $N \ge 3$. Then the set of symplectic primitive renormalizations of maps in $\mathfrak{Y}$ is dense in $\mathrm{Diff}^r_{\mathrm{Symp}}(\mathbb{D}, \mathbb{R}^2)$.
\end{thm}

In a sense, the \textit{family} $\mathfrak{Y}$ of Poincaré maps can be said to be $C^r$-universal (compare with Definition~\ref{def:universal_diffeo}). An informal reformulation of our main result would be that any symplectic diffeomorphism of the disk can be approximated, with arbitrary precision, by a renormalization of the Poincaré map of some celestial mechanics model.

\section{A family of central configurations with arbitrary number of parameters}\label{sec:family_CC}

\subsection{The continuation of the centered N-gon central configuration}

In this section we prove the main part of Theorem~\ref{thm:posicions_respecte_a_les_masses_s2}, that is, that the centered $N$-gon central configuration, that is $N$ bodies of mass 1 in the vertexes of a regular $N$-gon, forming a ring, and a central body of mass $m$, can be continued as a central configuration when the masses of the ring are varied and the central mass is large enough. The proof is inspired by ideas from \cite{MoeckelJDDE1994}; see also \cite{Casasayas1994CentralConfigs}. The fact that the configuration is symmetric if the masses satisfy certain symmetry condition is proven in Section~\ref{sec:cc_simetriques}.

Let $q_j\in \R^2$,  $j=0,\dots,N$, be $N+1$ bodies, with masses $m_j$ for $j=0,\dots, N-1$, and $m$ for $j = N$. In what follows, we will assume that $m_j$ is close to $1$ while $m\gg 1$. Assume that their center of mass lies at the origin, that is,
\[
\sum_{j=0}^{N-1} m_j q_j +m q_N = 0,
\]
and that $(q_0,\dots,q_N) \in S_{\mm}\setminus \Delta_{\mm}$. We write the gravitational potential $V_{\mm}$ in~\eqref{eq:def_potential_N_body} as
\begin{equation}
\label{def:VmU1U2}
V_{\mm} = U_1+ U_2
\end{equation}
where
\begin{align}
\label{def:U1}
U_1(q_0,\dots,q_{N-1}) & = \sum_{0\le j < k \le N-1}\frac{m_j m_k}{|q_j-q_k|} = \sum_{0\le j < k \le N-1}\frac{m_j m_k}{((q_j-q_k)(\bar q_j-\bar q_k))^{1/2}},
\\
\label{def:U2}
U_2(q_0,\dots,q_{N}) & = \sum_{j=0}^{N-1} \frac{m_j m}{|q_j-q_N|} = \sum_{j=0}^{N-1} \frac{m_j m}{((q_j-q_N)(\bar q_j-\bar q_N))^{1/2}},
\end{align}
We are identifying $\R^2$ with $\C$. Hence, $\bar z$ denotes de complex conjugate of $z\in \C$ and $|z|^2 = z \bar z$.

We recall that, from Definition~\ref{def:CC_critical_point}, $(q_0,\dots,q_N)$ is a central configuration if there exists $\lambda \in \R$ such that
\begin{equation}
    \label{def:cc_system}
    \partial_{q_j} V_{\mm} (q_0,\dots,q_N) = \frac{\lambda m_j}{2} \bar q_j, \qquad j=0,\dots, N.
\end{equation}
Using the decomposition of $V_{\mm}$ in~\eqref{def:VmU1U2}, equation~\eqref{def:cc_system} is equivalent to
\begin{equation}
    \label{def:cc_system2}
    \left\{
    \begin{aligned}
      \partial_{q_j} U_1 (q_0,\dots,q_{N-1}) +   \partial_{q_j} U_2 (q_0,\dots,q_{N}) & = \frac{\lambda m_j}{2} \bar q_j, \qquad j = 0,\dots,N-1 \\
      \partial_{q_N} U_2 (q_0,\dots,q_{N}) & = \frac{\lambda m}{2} \bar q_N.
    \end{aligned}
    \right.
\end{equation}

It  is immediate to check that, when $m_j =1$, $j=0,\dots,N-1$, and $m> 0$, a solution of~\eqref{def:cc_system2} is given
by $q_j = e^{2\pi i j/N+\theta}$, $j=0,\dots,N-1$, and $q_N=0$, for any $\theta \in \R$. This is the Centered $N$-gon central configuration. The freedom on the choice of $\theta$ is a consequence of the preservation of the total angular momentum.

Since we want to consider central configurations close to the polygonal one, we introduce $\mu_j$ by setting $m_j = 1+ \mu_j$, $j=0,\dots,N-1$. 

Let $\omega = e^{2\pi i/N}$ and $\theta^0_j = 2\pi j/N$, $j = 0, \dots, N-1$, $\hat \theta^0 = (\theta_1^0, \dots, \theta_{N-1}^0)^\top$. We introduce
the matrix $\MM$ by\footnote{$\M$ is the matrix of the discrete Fourier transform. It was also used in \cite{palmore1976} and \cite{MeyerSchmidt1988, SchmidtRecifeLectures}.}
\begin{equation}
\label{def:M}
   \MM_{j,k} = \omega^{jk}, \qquad j,k = 0,\dots,N-1.
\end{equation}
It satisfies
\[
\M^\top = \M, \qquad \overline{\MM} \MM = N \,\Id.
\]

Next theorem implies the first part of Theorem~\ref{thm:posicions_respecte_a_les_masses_s2}.

\begin{thm}
\label{thm:posicions_respecte_a_les_masses}
For any $N\ge 3$, there exist $m^*,  \mu^* >0$ such that, for all $\|\mu\| < \mu^*$ and $m > m^*$, there exists a unique solution $q(\mu,m)$ of~\eqref{def:cc_system2} with $\lambda$ given by~\eqref{def:lambda_cc}, analytic in $\mu$ and $\varepsilon= 1/m$,  of the form
\[
q_j(\mu,\varepsilon) = \rho_j(\mu,\varepsilon) e^{i\theta_j(\mu,\varepsilon)}, \qquad j = 0,\dots, N-1,
\]
with $\theta_0(\mu,\varepsilon) = 0$, $q_N(\mu,\varepsilon) = - \varepsilon \sum_{\ell=0}^{N-1} (1+\mu_\ell) q_\ell(\mu,\varepsilon)$,
where, with the normalizing condition $\sum_{\ell= 0}^{N-1} \mu_\ell = 0$,
\[
\rho_j(\mu,\varepsilon) =  1 + \OO(\mu)\OO(\varepsilon), \qquad  j = 0,\dots, N-1, 
\]
 and, denoting $\hat \theta = (\theta_1,\dots,\theta_{N-1})^\top$,
\[
\hat \theta (\mu,\varepsilon) = \hat \theta^0 + \frac{1}{N} \M \begin{pmatrix}
    -c^\top \\ \hat \D_c
\end{pmatrix}
(0,\Id) \overline{\M} \mu + \OO(\varepsilon)+\OO_2(\mu,\varepsilon),
\]
where $c=(c_1,\dots,c_{N-1})^\top$, $\hat \D_c = \mathrm{diag}\,(c_1,\dots,c_{N-1})$ and the coefficients
\[
c_j = - a_j/b_j, \qquad j = 1,\dots, N-1
\]
are given in Lemmas~\ref{lem:DmutildeG} and~\ref{lem:DthetatildeG_rang}.
\end{thm}

\begin{proof} 
\emph{Preliminary considerations.}
We recall that the solutions of~\eqref{def:cc_system2} satisfy
\begin{equation}
\label{def:equacio_CM}
\sum_{j=0} (1+\mu_j) q_j + m q_N = 0,
\end{equation}
that is, they must have the center of mass at the origin. 
This implies that, if~\eqref{def:equacio_CM} is satisfied, one of the equations of~\eqref{def:cc_system2} is redundant. We discard the last one, that is, the one with $j=N$.

We introduce the new unknowns $(\rho_j,\theta_j)$ by setting $q_j = \rho_j e^{i\theta_j}$, $j=0,\dots, N-1$, and define the function $q_N = \varepsilon Q(\rho,\theta,\mu)$ through~\eqref{def:equacio_CM}, that is,
\begin{equation}
    \label{def:QN}
    Q (\rho,\theta,\mu) = - \sum_{j=0}^{N-1} (1+\mu_j) \rho_j e^{i\theta_j}.
\end{equation}
We  denote the polygonal central configuration
\begin{equation}
\label{def:configuracio_central_poligonal}
    P_\varepsilon = \{\rho_j =1, \;\theta_j = 2\pi j/N, \; j=0,\dots,N-1, \;\|\mu\| = 0 \}.
\end{equation}
Clearly, $Q_{\mid P_\varepsilon} = 0$.

We first observe that, in view of~\eqref{def:lambda_cc} and~\eqref{def:VmU1U2},
\begin{equation}
\label{varepsilon_lambda}
-\varepsilon \lambda =  \frac{1}{N}\sum_{j=0}^{N-1} \frac{1+\mu_j}{\rho_j} + \varepsilon \Lambda(\rho,\theta,\mu,\varepsilon),
\end{equation}
where
\begin{multline}
\Lambda(\rho,\theta,\mu,\varepsilon) = 
\frac{1}{\varepsilon N}\sum_{j=0}^{N-1} \left(\frac{1+\mu_j}{[(\rho_j -\varepsilon Q (\rho,\theta,\mu)e^{-i\theta_j})(\rho_j -\varepsilon \overline{Q (\rho,\theta,\mu)} e^{i\theta_j})]^{1/2}} -\frac{1+\mu_j}{\rho_j}\right) \\
+ \frac{1}{N}\sum_{0\le j < k \leq N-1} \frac{(1+\mu_j)(1+\mu_k)}{[(\rho_j-\rho_k e^{i(\theta_k-\theta_j)})(\rho_j-\rho_k e^{-i(\theta_k-\theta_j)})]^{1/2}} 
\end{multline}
is analytic in all its variables in a neighborhood of the polygonal central configuration with infinity mass at the origin,
\[
P = \{\rho_j =1, \;\theta_j = 2\pi j/N, \; j=0,\dots,N-1, \;\|\mu\| = 0, \;\varepsilon = 0 \}.
\]

Now we perform the substitution of the new unknowns into equations~\eqref{def:cc_system2}, multiply  equation~$j$, for $j=0,\dots,N-1$, by $2e^{i\theta_j}$, divide it by $1+\mu_j$ and multiply all the equations by $\varepsilon$. Equations~\eqref{def:cc_system2} (recall we do not need to consider the last one) become, 
\begin{equation}
\label{def:equacions_CC_def}
 \frac{1}{N}\left(\sum_{k=0}^{N-1} \frac{1+\mu_k}{\rho_k} \right) \rho_j - \frac{1}{\rho_j^2} + 
    \varepsilon F_j(\rho,\theta,\mu,\varepsilon)  =  0, \qquad j = 0,\dots, N-1 ,
\end{equation}
where, for $j=0,\dots,N-1$,
\begin{multline}
\label{def:F_j}
F_j(\rho,\theta,\mu,\varepsilon) =  \Lambda(\rho,\theta,\mu,\varepsilon) \rho_j -\sum_{k\neq j} \frac{(1+\mu_k)}{[(\rho_j-\rho_k e^{i(\theta_k-\theta_j)})(\rho_j-\rho_k e^{-i(\theta_k-\theta_j)})]^{3/2}} (\rho_j -\rho_k e^{-i(\theta_k-\theta_j)}) \\ 
 - \varepsilon^{-1}\left(\frac{1}{[(\rho_j e^{i\theta_j}-\varepsilon Q(\rho,\theta,\mu))(\rho_j e^{-i\theta_j}-\varepsilon \overline{Q (\rho,\theta,\mu)})]^{3/2}} (\rho_j -\varepsilon \overline{Q (\rho,\theta,\mu)} e^{i\theta_j})) - \frac{1}{\rho_j^2}\right).
\end{multline}
The functions $F_j$, $j=0,\dots,N$, are analytic in all their variables in a neighborhood of $P$.
It is immediate to check that
\begin{equation}
\label{eq:Fj_en_poligon}
F_{j\mid \rho_j = 1,\, j=0,\dots,N-1,\, \mu = 0, \varepsilon = 0}  = A(\theta) + B_j(\theta), \qquad j=0,\dots N-1,
\end{equation}
where, using~\eqref{def:QN},
\begin{align}
\label{def:A_theta}
A(\theta) & = \frac{1}{N} \sum_{0\le \ell < k \leq N-1} \frac{1}{[(1-e^{i(\theta_k-\theta_\ell)})(1-e^{-i(\theta_k-\theta_\ell)})]^{1/2}}, \\
\label{def:B_j_theta}
B_j(\theta) & = - \sum_{k=0, \; k\neq j}^{N-1} \frac{1-e^{-i(\theta_k-\theta_j)}}{[(1-e^{i(\theta_k-\theta_j)})(1-e^{-i(\theta_k-\theta_j)})]^{3/2}} - \sum_{k=0}^{N-1} e^{-i(\theta_k-\theta_j)}.
\end{align}

We remark that~\eqref{def:equacions_CC_def} represents a set of $N$ complex equations, while $\rho$ and $\theta$ are $2N$ real parameters. However, if  $(\rho,\theta)$ is a solution of~\eqref{def:equacions_CC_def}, also is $(\rho,\tilde \theta)$,
where $\tilde \theta_j = \theta_j + \varphi$, for any $\varphi$. This is due to the symmetry by rotations. 

Since  for all $\varepsilon$,  $(\rho,\theta,\mu)$ of the form
\[
\rho_j = 1, \quad \theta_j = 2\pi j/N, \; j=0,\dots, N-1,  \quad \mu = 0,
\]
are solutions of~\eqref{def:cc_system2}, so they are of~\eqref{def:equacions_CC_def}. This can also be checked by a direct substition into~\eqref{def:equacions_CC_def}. 

We will proceed as follows.
\begin{itemize}
    \item First, we check that the real part of equations~\eqref{def:equacions_CC_def}, for $j=0,\dots,N-1$, defines an analytic function $\tilde \rho(\theta,\mu,\varepsilon)$.
    \item Second, we insert $\tilde \rho$ into the imaginary part of equations~\eqref{def:equacions_CC_def}, for $j=0,\dots,N-1$. These equations have an additional degenary, due to the invariance by rotations. Setting $\theta_0 = 0$, it is enough to consider a subset of $N-1$ equations. We check that they define 
    an analytic function $(\theta_1,\dots,\theta_{N-1}) = \hat \theta(\mu,\varepsilon)$. Hence, the first $N$ equations of~\eqref{def:equacions_CC_def} define functions
    \[
    (\rho,\theta) = (\tilde \rho(\hat \theta(\mu,\varepsilon),\mu,\varepsilon), 0, \hat \theta(\mu,\varepsilon)).
    \]
\end{itemize}

\emph{First step. The real part of~\eqref{def:equacions_CC_def}.}  We denote by $\tilde F$ the real part of the left hand side of~\eqref{def:equacions_CC_def}, that is,
\[
\tilde F_j (\rho,\theta,\mu,\varepsilon) = 
\frac{1}{N} \left(\sum_{k=0}^{N-1} \frac{1+\mu_k}{\rho_k} \right) \rho_j - \frac{1}{\rho_j^2} + \varepsilon \mathrm{Re}\, F_j(\rho,\theta,\mu,\varepsilon).
\]
Clearly,
\[
\partial_{\rho_k}\tilde F_{j\mid P} = 
\begin{cases}
    2, \qquad & \text{if $k=j$}, \\
    -1/N, \qquad & \text{if $k\neq j$}.
\end{cases}
\]
Hence, by Gershgorin circle theorem, $\det D_\rho \tilde F_{\mid P} \neq 0$. As a consequence, the equation $\tilde F = 0$ defines $\rho$ as a function $\tilde \rho$ of $(\theta,\mu,\varepsilon)$. It is clear that
\begin{equation}
\label{def:tilde_rho}
\tilde \rho (\theta,\mu,\varepsilon)= \rho_0(\mu) + \varepsilon \rho_1 (\theta,\mu,\varepsilon)
\end{equation}
is an analytic function of its variables, where the $j$ component of $\rho_0$ is
\begin{equation}
\label{def:rho_0}
\rho_{0,j}(\mu) = \frac{1}{\left(1+ \sum_{k=0}^{N-1} \frac{\mu_k}{N} \right)^{1/2}}
= 1 -\frac{1}{2} \sum_{k=0}^{N-1} \frac{\mu_k}{N} + \OO (\|\mu\|^2),
\end{equation}
which is independent of $j$, provided that $\sum_{k=0}^{N-1} \frac{\mu_k}{N} \neq -1$.

\emph{Second step. The imaginary part of~\eqref{def:equacions_CC_def}.} Let  $\varepsilon \tilde G$ be the imaginary part of the left hand side of the first $N$ equations of~\eqref{def:equacions_CC_def}, with the substitution $\rho = \tilde \rho (\theta,\mu,\varepsilon)$,
that is,
\begin{equation}
\label{def:tildeG}
\tilde G_j (\theta,\mu,\varepsilon) = \textrm{Im}\, F_j(\tilde \rho (\theta,\mu,\varepsilon),\theta,\mu,\varepsilon), \qquad j=0,\dots,N-1.
\end{equation}
At $\mu = 0$, $\varepsilon = 0$, since $\tilde \rho(\theta,0,0) = (1,\dots,1)^\top$, using~\eqref{eq:Fj_en_poligon}, \eqref{def:A_theta} and~\eqref{def:B_j_theta}, we have that
\[
\tilde G_j ( \theta,0,0) = \textrm{Im}\, B_j(\theta) =  - \frac{1}{2^{3/2}} \sum_{k=0, \; k\neq j}^{N-1} \frac{\sin(\theta_k-\theta_j)}{(1-\cos (\theta_k-\theta_j))^{3/2}} + \sum_{k=0, \; k\neq j}^{N-1} \sin(\theta_k-\theta_j),
\]
where the function $B_j$ was introduced in~\eqref{def:B_j_theta}.
Hence, for $0\le j \le N-1$,
\begin{equation}
\label{eq:partial_theta_j_tildeG_l_dif}
\partial_{\theta_\ell} \tilde G_j (\theta,0,0) =  \frac{1}{2^{5/2}} 
\frac{3+\cos(\theta_\ell-\theta_j)}{(1-\cos (\theta_\ell-\theta_j))^{3/2}} + \cos(\theta_\ell-\theta_j), \qquad \ell \neq j,
\end{equation}
and
\begin{equation}
\label{eq:partial_theta_j_tildeG_l_igual}
\partial_{\theta_j} \tilde G_j (\theta,0,0) =- \frac{1}{2^{5/2}} \sum_{k=0, \; k\neq j}^{N-1} \frac{3+\cos(\theta_k-\theta_j)}{(1-\cos (\theta_k-\theta_j))^{3/2}}-\sum_{k=0, \; k\neq j}^{N-1} \cos(\theta_k-\theta_j).
\end{equation}

\begin{lem}
\label{lem:DthetatildeG_rang} Let $\M$ be the matrix introduced in~\eqref{def:M}. Then, $D_\theta \tilde G_{\mid \theta = \theta^0,\, \mu = 0, \, \varepsilon = 0}$ satisfies
\[
(\overline{\MM} D_\theta \tilde G_{\mid \theta = \theta^0,\, \mu = 0, \, \varepsilon = 0} \MM)_{i,j} = 
\begin{cases}
    0, & \quad \text{if $i-j \neq 0 \mod N$}, \\
    b_j, & \quad \text{if $i-j = 0 \mod N$},
\end{cases}
\]
where, for $j = 0,\dots N-1$, $j\neq 1, N-1$,
\[
b_j = \frac{N}{2^{5/2}} \sum_{k=1}^{N-1} \frac{3+(\omega^k+\omega^{-k})/2}{(1-(\omega^k+\omega^{-k})/2)^{3/2}} (\omega^{kj}-1) =-\frac{N}{4} \sum_{k=1}^{N-1} \frac{1+\cos^2(\pi k /N)}{\sin^3 (\pi k/N)} \sin^2(\pi k j/N), 
\]
and
\[
\begin{aligned}
b_1 & = -\frac{N}{4} \sum_{k=1}^{N-1} \frac{1+\cos^2(\pi k /N)}{\sin^3 (\pi k/N)} \sin^2(\pi k /N) +\frac{N
^2}{2},\\
b_{N-1} & = -\frac{N}{4} \sum_{k=1}^{N-1} \frac{1+\cos^2(\pi k /N)}{\sin^3 (\pi k/N)} \sin^2(\pi k /N) +\frac{N
^2}{2}.
\end{aligned}
\]
In particular, $b_j = 0$ if and only if $j = 0$. Hence, $\textrm{rank}\, D_\theta \tilde G_{\mid \theta = \theta^0} = N-1$.
\end{lem}
\begin{rem}
    It can be readily proved that $b_1 = b_{N-1}$ are not zero at least for large $N$. In fact, from the proof of Proposition~\ref{prop:negativity_Ck}, this is the case for $N\ge 12$.
\end{rem}
\begin{proof}[Proof of Lemma~\ref{lem:DthetatildeG_rang}]
The expression of the coefficients of  $\overline{\MM} D_\theta \tilde G_{\mid \theta = \theta^0} \MM$ is obtained by a straightforward computation from~\eqref{eq:partial_theta_j_tildeG_l_dif}, \eqref{eq:partial_theta_j_tildeG_l_igual} and~\eqref{def:M}. The last claim is immediate from the expression of $b_j$.
\end{proof}

Now we compute $D_\mu \tilde G_{\mid \theta = \theta^0,\, \mu = 0, \, \varepsilon = 0}$. 
\begin{lem}
  \label{lem:DmutildeG}
The matrix $D_\mu \tilde G_{\mid \theta = \theta^0,\, \mu = 0, \, \varepsilon = 0}$ satisfies
\[
(\overline{\MM} D_\mu \tilde G_{\mid \theta = \theta^0,\, \mu = 0, \, \varepsilon = 0} \MM)_{j,k} = 
\begin{cases}
    0, & \quad \text{if $j-k \neq 0 \mod N$}, \\
    a_k, & \quad \text{if $j-k = 0 \mod N$},
\end{cases}
\]
where, for $k = 0,\dots N-1$, $k\neq 1, N-1$,
\[
a_k = - \frac{iN}{4} \sum_{\ell = 1}^{N-1} \frac{\cos(\pi \ell/N) \sin (2\pi k\ell/N)}{\sin^2(\pi \ell/N)}, 
\]
and
\[
\begin{aligned}
a_1 & = - \frac{iN}{4} \sum_{\ell = 1}^{N-1} \frac{\cos(\pi \ell/N) \sin (2\pi \ell/N)}{\sin^2(\pi \ell/N)}+\frac{i N^2}{2}, \\
a_{N-1} & = - \frac{iN}{4} \sum_{\ell = 1}^{N-1} \frac{\cos(\pi \ell/N) \sin (2\pi (N-1)\ell/N)}{\sin^2(\pi \ell/N)}-\frac{i N^2}{2}.
\end{aligned}
\]
In particular, $a_0 = 0$. 
\end{lem}
\begin{proof}
From~\eqref{def:tildeG} and~\eqref{def:F_j} (and using that some dependence of $F_j$ on $\rho$, $\theta$, $\mu$ and $\varepsilon$ is through $Q$ and $\overline{Q}$ in~\eqref{def:QN}), we introduce
\begin{equation}
\label{def:tildeF}
\tilde F(\rho,\theta,Q,\overline Q,\mu,\varepsilon) = \textrm{Im}\, F(\rho,\theta,Q,\overline Q,\mu,\varepsilon).
\end{equation}
We have that
\begin{align}
\label{DrhotildeFDmutilderho}
D_\mu \tilde G (\theta^0,0,0)  & = D_\rho \tilde F(\rho^0,\theta^0,0,0,0,0) D_\mu \tilde \rho(\theta^0,0,0) \\
\label{DQtildeF}
& + D_Q \tilde F(\rho^0,\theta^0,0,0,0,0) \left( D_\mu Q (\rho^0,\theta^0,0)+ D_\rho Q (\rho^0,\theta^0,0) D_\mu \tilde \rho (\theta^0,0,0) \right) \\
\label{DbarQtildeF}
& + D_{\overline{Q}} \tilde F(\rho^0,\theta^0,0,0,0,0) \left( D_\mu \overline{Q} (\rho^0,\theta^0,0)+ D_\rho \overline{Q} (\rho^0,\theta^0,0) D_\mu \tilde \rho (\theta^0,0,0) \right) \\
\label{Dmutildef}
& + D_\mu \tilde F(\rho^0,\theta^0,0,0,0,0).
\end{align}
We compute each term separately.

We first observe that, in view of \eqref{def:tilde_rho} and~\eqref{def:rho_0},
\begin{equation}
\label{partial_mu_k_tilde_rho}
\partial_{\mu_k} \tilde \rho_{j \mid \mu = 0, \,\varepsilon = 0} = -\frac{1}{2N}, \qquad j,k = 0, \dots, N-1.
\end{equation}

We begin with~\eqref{DrhotildeFDmutilderho}. From~\eqref{def:tildeF} and~\eqref{def:F_j}, we have that, for $\ell \neq j$
\[
\partial_{\rho_\ell} \tilde F_j(\rho^0,\theta^0,0,0,0,0) =-\frac{\sin (2\pi(\ell-j)/N)}{2^{5/2}\left(1-\cos(2\pi(\ell-j)/N)\right)^{3/2}}
\]
and
\[
\partial_{\rho_j} \tilde F_j(\rho^0,\theta^0,0,0,0,0)  = - \sum_{n\neq j} \frac{3 \sin (2\pi(n-j)/N)}{2^{5/2}\left(1-\cos(2\pi(n-j)/N)\right)^{3/2}}
\]
Hence, for~\eqref{DrhotildeFDmutilderho}, 
\[
\begin{aligned}
(D_\rho \tilde F(\rho^0,\theta^0,0,0,0,0) D_\mu \tilde \rho(\theta^0,0,0))_{j,k}& = 
\frac{1}{2^{7/2} N} \sum_{n\neq j} \frac{4 \sin (2\pi(n-j)/N)}{\left(1-\cos(2\pi(n-j)/N)\right)^{3/2}} \\
& = \frac{1}{2^{7/2}N} \sum_{n=1}^{N-1} \frac{4 \sin (2\pi n/N)}{\left(1-\cos(2\pi n/N)\right)^{3/2}}    = 0.
\end{aligned}
\]

Now we deal with~\eqref{DQtildeF} and~\eqref{DbarQtildeF}. 
We first observe that, in view of~\eqref{def:tilde_rho} and~\eqref{def:rho_0},
\[
\partial_{\rho_j} Q (\rho^0,\theta^0,0) = - \omega^j, \qquad \partial_{\rho_j} \overline{Q} (\rho^0,\theta^0,0) = - \omega^{-j}, \qquad j=0,\dots, N-1,
\]
which, using also~\eqref{partial_mu_k_tilde_rho}, implies
\[
(D_\rho Q (\rho^0,\theta^0,0) D_\mu \tilde \rho (\theta^0,0,0))_{j,k} = -\frac{1}{2N} \sum_{j=0}^{N-1} \omega^j = 0
\]
and
\[
(D_\rho \overline{Q} (\rho^0,\theta^0,0) D_\mu \tilde \rho (\theta^0,0,0))_{j,k} = -\frac{1}{2N} \sum_{j=0}^{N-1} \omega^{-j} = 0.
\]
Moreover,  since $\textrm{Im}\, (\overline{Q} \omega^j) = (\overline{Q} \omega^j - Q \omega^{-j})/(2i)$, in view of~\eqref{def:F_j},
\[
\begin{aligned}
\partial_Q \tilde F_j(\rho^0,\theta^0,0,0,0,0) & = -\frac{1}{2i}  \omega^{-j}\\
\partial_{\overline{Q}} \tilde F_j(\rho^0,\theta^0,0,0,0,0) & = \frac{1}{2i}  \omega^{j},
\end{aligned}
\]
and, by~\eqref{def:QN},
\[
\partial_{\mu_k} Q(\rho^0,\theta^0,0) = - \omega^{k}, \qquad \partial_{\mu_k} \overline{Q}(\rho^0,\theta^0,0) = - \omega^{-k}, \qquad k = 0,\dots,N-1,
\]
we have that, for $j,k = 0,\dots, N-1$,
\begin{multline*}
\left(D_Q \tilde F(\rho^0,\theta^0,0,0,0,0) \left( D_\mu Q (\rho^0,\theta^0,0)+ D_\rho Q (\rho^0,\theta^0,0) D_\mu \tilde \rho (\theta^0,0,0) \right)\right)_{j,k} \\
 = \left(D_Q \tilde F(\rho^0,\theta^0,0,0,0,0)  D_\mu Q (\rho^0,\theta^0,0)\right)_{j,k}
 = \frac{1}{2i} \omega^{k-j}
\end{multline*}
and
\begin{multline*}
\left(D_{\overline{Q}} \tilde F(\rho^0,\theta^0,0,0,0,0) \left( D_\mu \overline{Q} (\rho^0,\theta^0,0)+ D_\rho \overline{Q} (\rho^0,\theta^0,0) D_\mu \tilde \rho (\theta^0,0,0) \right)\right)_{j,k} \\
 = \left(D_{\overline{Q}} \tilde F(\rho^0,\theta^0,0,0,0,0)  D_\mu  \overline{Q}(\rho^0,\theta^0,0)\right)_{j,k}
 = -\frac{1}{2i} \omega^{-(k-j)}.
\end{multline*}

Finally we deal with~\eqref{Dmutildef}. In view of~\eqref{def:F_j}, we have that, if $k\neq j$,
\[
\partial_{\mu_k} \tilde F_j (\rho^0,\theta^0,0,0,0,0) = \frac{\textrm{Im}\, \omega^{-(k-j)}}{[(1-\omega^{k-j})(1-\omega^{-(k-j)})]^{3/2}} = - \frac{1}{2 i} \frac{\omega^{k-j}- \omega^{-(k-j)}}{[(1-\omega^{k-j})(1-\omega^{-(k-j)})]^{3/2}}
\]
and, for $k=j$,
\[
\partial_{\mu_j} \tilde F_j (\rho^0,\theta^0,0,0,0,0) = 0.
\]
Summarizing, we have proven that, if $k\neq j$,
\begin{equation}
\label{coeficients_D_mu_tilde_G_j_diferent_de_k}
(D_\mu \tilde G (\theta^0,0,0))_{j,k} = - \frac{1}{2 i}\frac{\omega^{k-j}- \omega^{-(k-j)}}{[(1-\omega^{k-j})(1-\omega^{-(k-j)})]^{3/2}}  + \frac{1}{2i}(\omega^{k-j}-\omega^{-(k-j)})
\end{equation}
and
\begin{equation}
\label{coeficients_D_mu_tilde_G_j_igual_a_k}
(D_\mu \tilde G (\theta^0,0,0))_{j,j} = 0.
\end{equation}

We finally compute $\overline{\M} D_{\mu} \tilde G(\theta^0,0,0) \M$, where $\M$ was defined in~\eqref{def:M}. First, from~\eqref{coeficients_D_mu_tilde_G_j_diferent_de_k} and~\eqref{coeficients_D_mu_tilde_G_j_igual_a_k},
\[
\begin{aligned}
(D_{\mu} \tilde G(\theta^0,0,0) \M)_{j,k} & = -\frac{1}{2i} \sum_{\ell = 0, \ell\neq j}^{N-1} 
\left(\frac{\omega^{\ell-j}- \omega^{-(\ell -j)}}{[(1-\omega^{\ell-j})(1-\omega^{-(\ell-j)})]^{3/2}} - (\omega^{\ell-j}-\omega^{-(\ell-j)}) \right) \omega^{k\ell}  \\
& =  -\frac{1}{2i} \sum_{\ell = 1}^{N-1} 
\left(\frac{\omega^{\ell}- \omega^{-\ell }}{[(1-\omega^{\ell})(1-\omega^{-\ell})]^{3/2}} - (\omega^{\ell}-\omega^{-\ell}) \right) \omega^{k(\ell+j)}.
\end{aligned}
\]
Hence, 
\[
\begin{aligned}
(\overline{\M} D_{\mu} \tilde G(\theta^0,0,0) \M)_{j,k} & = -\frac{1}{2i} \sum_{n=0}^{N-1} \omega^{-n j} \sum_{\ell = 1}^{N-1} 
\left(\frac{\omega^{\ell}- \omega^{-\ell }}{[(1-\omega^{\ell})(1-\omega^{-\ell})]^{3/2}} - (\omega^{\ell}-\omega^{-\ell}) \right) \omega^{k(\ell+n)} \\
& = -\frac{1}{2i}   \sum_{\ell = 1}^{N-1} 
\left(\frac{\omega^{\ell}- \omega^{-\ell }}{[(1-\omega^{\ell})(1-\omega^{-\ell})]^{3/2}} - (\omega^{\ell}-\omega^{-\ell}) \right) \omega^{k\ell} \sum_{n=0}^{N-1}  \omega^{n(k-j)} \\
& = 
\begin{cases}
 N \sum_{\ell = 1}^{N-1} 
a_{\ell,k} \omega^{k\ell}, & \; k = j, \\
0, & k \neq j,
\end{cases}
\end{aligned}
\]
where
\[
a_{\ell,k} = -\frac{1}{2i}  
\left(\frac{\omega^{\ell}- \omega^{-\ell }}{[(1-\omega^{\ell})(1-\omega^{-\ell})]^{3/2}} - (\omega^{\ell}-\omega^{-\ell}) \right) , \quad \ell = 1,\dots, N-1, \; k = 0,\dots,N-1.
\]
Note that the above expression for $(\overline{\M} D_{\mu} \tilde G(\theta^0,0,0) \M)_{j,k}$ can be further simplified. Indeed, first we remark that, since $a_{N-\ell,k} = - a_{\ell,k}$, and $\omega^{k(N-\ell)} = \omega^{-k\ell}$, we have that
\[
(\overline{\M} D_{\mu} \tilde G(\theta^0,0,0) \M)_{k,k} = \frac{1}{2}\sum_{\ell = 1}^{N-1} a_{\ell,k} (\omega^{k\ell}-\omega^{-k\ell}).
\]
Then, using that
\[
\frac{1}{2i}   \sum_{\ell = 1}^{N-1} 
 (\omega^{\ell}-\omega^{-\ell}) \omega^{k\ell} = 
 \begin{cases}
 0, \qquad & k \neq 1, N-1,\\
 -N/2i \qquad & k = 1,\\
 N/2i \qquad & k = N-1,\\
 \end{cases}
\]
we have that, for $k\neq 1,N-1$,
\begin{multline*}
(\overline{\M} D_{\mu} \tilde G(\theta^0,0,0) \M)_{k,k} = -\frac{N}{4i}\sum_{\ell = 1}^{N-1} 
\frac{(\omega^{\ell}- \omega^{-\ell })(\omega^{k\ell}-\omega^{-k\ell})}{[(1-\omega^{\ell})(1-\omega^{-\ell})]^{3/2}}  \\ 
= - \frac{iN}{2^{3/2}} \sum_{\ell = 1}^{N-1} \frac{\sin(2\pi \ell/N) \sin (2\pi k\ell/N)}{(1-\cos(2\pi \ell/N))^{3/2}}  = - \frac{iN}{4} \sum_{\ell = 1}^{N-1} \frac{\cos(\pi \ell/N) \sin (2\pi k\ell/N)}{\sin^2(\pi \ell/N)},
\end{multline*}
while
\[
(\overline{\M} D_{\mu} \tilde G(\theta^0,0,0) \M)_{1,1} = - \frac{iN}{4} \sum_{\ell = 1}^{N-1} \frac{\cos(\pi \ell/N) \sin (2\pi \ell/N)}{\sin^2(\pi \ell/N)}-\frac{N^2}{2i}
\]
and
\[
(\overline{\M} D_{\mu} \tilde G(\theta^0,0,0) \M)_{N-1,N-1} = - \frac{iN}{4} \sum_{\ell = 1}^{N-1} \frac{\cos(\pi \ell/N) \sin (2\pi(N-1) \ell/N)}{\sin^2(\pi \ell/N)}+\frac{N^2}{2i}.
\]
Hence the claim is proven.
In particular,  $(\overline{\M} D_{\mu} \tilde G(\theta^0,0,0) \M)_{0,0} = 0$.
\end{proof}

After Lemmas~\ref{lem:DthetatildeG_rang} and~\ref{lem:DmutildeG}, we can finish the proof of Theorem~\ref{thm:posicions_respecte_a_les_masses}.

We have seen in the \emph{first step}  that the real part of~\eqref{def:equacions_CC_def} defines $\rho = \tilde \rho(\theta,\mu,\varepsilon)$, in~\eqref{def:tilde_rho}, as an analytic function of $(\theta,\mu,\varepsilon)$ in a neighborhood of $(\theta^0,0,0)$. 

Due to the conservation of the angular momentum, one of the equations of the imaginary part of~\eqref{def:equacions_CC_def} is redundant. This is equivalent to fixing one of the angles. Then, 
 Lemma~\ref{lem:DthetatildeG_rang} implies that a subset of $N-1$ equations of the imaginary part of~\eqref{def:equacions_CC_def} defines $\theta = \Theta(\mu,\varepsilon)$ as an analytic function of $(\mu,\varepsilon)$ in a neighborhood of $(0,0)$. Let us compute the first order in $\mu$ of $\Theta$. We introduce
 $\phi = \theta-\theta^0$. In view of Lemmas~\ref{lem:DthetatildeG_rang} and~\ref{lem:DmutildeG}, finding the first order of $\Theta$ is equivalent to solve $\phi$ as a function of $\mu$ from
 \begin{equation}
 \label{primer_ordre_de_phi_en_mu}
 D_\theta \tilde G(\theta^0,0,0) \phi = - D_\mu  \tilde G(\theta^0,0,0) \mu
 \end{equation}
Introducing $\varphi$ and $\nu$ by setting $\phi = \M \varphi$ and $\mu = \M \nu$, respectively, $\hat \varphi = (\varphi_1,\dots,\varphi_{N-1})^\top$, $\hat \nu = (\nu_1,\dots,\nu_{N-1})^\top$,
 $\D_a = \mathrm{diag}\,(a_0,\dots, a_{N-1})$, $\D_b = \mathrm{diag}\,(b_0,\dots, b_{N-1})$, $\hat \D_a = \mathrm{diag}\,(a_1,\dots, a_{N-1})$, and $\hat \D_b = \mathrm{diag}\,(b_1,\dots, b_{N-1})$, where the coefficients $a_j$, $b_j$, $j=0,\dots,N-1$ are given in Lemmas~\ref{lem:DmutildeG} and~\ref{lem:DthetatildeG_rang}, respectively, equation~\eqref{primer_ordre_de_phi_en_mu} becomes
\begin{equation}
 \label{primer_ordre_de_varphi_en_nu}
 \D_b \varphi = -\D_a \nu.
 \end{equation}
Notice that, since $b_0 = a_0 = 0$, in~\eqref{primer_ordre_de_varphi_en_nu} $\varphi_0$ is free and $\varphi$ does not depend on $\nu_0$. 
Equation~\eqref{primer_ordre_de_varphi_en_nu} becomes
\begin{equation}
 \label{primer_ordre_de_hat_varphi_en_hat_nu}
 \hat \varphi =  \hat \D_c \hat \nu.
 \end{equation}
 where
 \begin{equation}
\hat \D_c =  \mathrm{diag}\,( c_1,\dots,c_{N-1}) = -  \hat \D_b^{-1} \hat \D_a =  - \mathrm{diag}\,( a_1/b_1,\dots,a_{N-1}/b_{N-1}).
 \end{equation}
Since $\varphi$ does not depend on $\nu_0$, we can take $\nu_0 = 0$. Since, by the definition of $\nu$, $\nu_0 = N^{-1} \sum_{\ell = 0}^{N-1} \mu_\ell$, this choice,  by~\eqref{def:tilde_rho} and~\eqref{def:rho_0}, implies that $\tilde \rho (\theta,\mu, \varepsilon)_{\mid \nu_0 = 0} = 1+ \varepsilon \rho_1(\theta,\mu, \varepsilon)_{\mid \nu_0 = 0} $. 

Hence, we obtain that
\[
\phi = \M \varphi = \M
\begin{pmatrix}
    \varphi_0 \\
    \hat \D_c \hat \nu 
\end{pmatrix}.
\]
Since $\varphi_0$ is free, we can choose it in such a way that $\phi_0 = 0$, that is, 
\[
\varphi_0 = - \sum_{\ell= 1}^{N-1} c_\ell \nu_\ell.
\]
This is equivalent to fix $\theta_0 = 0$. Hence, since, with the current choice of $\nu_0$,
\[
\frac{1}{N} \overline{\M} \mu = \begin{pmatrix}
    0 \\ \tilde \nu
\end{pmatrix},
\]
denoting $c = (c_1,\dots,c_{N-1})^\top$, we have
\[
\phi = \M 
\begin{pmatrix}
        -c^\top \\ \hat \D_c
    \end{pmatrix}  \hat \nu = \frac{1}{N} \M 
\begin{pmatrix}
        -c^\top \\ \hat \D_c
    \end{pmatrix} ( 0,\Id) \,\overline{\M} \mu,
\]
with the condition $\sum_{\ell=0}^{N-1} \mu_\ell = 0$.
\end{proof}

\subsection{Symmetric central configurations}
\label{sec:cc_simetriques}

Here we prove that it is possible to choose the central configuration to be symmetric with respect to the real axis by taking symmetric masses.

Define the index permutation $S \colon \{0, 1,\ldots,N-1\} \to \{0, 1,\ldots,N-1\} $, $S(\ell) = N-\ell \pmod{N}$ and the reflection with respect to the horizontal axis $\widehat{S} \colon \mathbb{R}^2 \to \mathbb{R}^2$, $\widehat{S}(x,y)=(x,-y)$.

Let $\mathfrak{S}(\mu_0,\mu_1,\ldots,\mu_{N-1},m) = (\mu_{S(0)},\mu_{S(1)},\ldots,\mu_{S(N-1)},m)$ and let $\widehat{\mathfrak{S}}(q_0,q_1,\ldots,q_{N-1},q_N) = (\widehat{S}(q_{S(0)}),\widehat{S}(q_{S(1)}),\ldots,\widehat{S}(q_{S(N-1)}),\widehat{S}(q_N))$.

Then, as a consequence of the uniqueness of solution that the implicit function theorem provides for a nondegenerate central configuration, Theorem~\ref{thm:posicions_respecte_a_les_masses} implies

 \begin{prop}\label{prop:symmetry_CC}
 Let $\mathcal{W}=\{(\mu, m) \colon \lVert \mu \rVert < \mu^\ast, m > m^\ast\}$ be the domain of definition of the central configuration $q(\mu, m)$ obtained in Theorem~\ref{thm:posicions_respecte_a_les_masses}. Assume $\tilde{\mathcal{W}} \subset \mathcal{W}$ satisfies $\mathfrak{S}(\tilde{\mathcal{W}}) = \tilde{\mathcal{W}}$. Then, for $(\mu, m) \in \tilde{\mathcal{W}}$, the central configuration $q=q(\mu, m)$ satisfies that $\widehat{\mathfrak{S}}(q(\mu, m)) = q(\mu, m)$.
\end{prop}

Therefore, by choosing the masses symmetrically in the ring, we obtain central configurations which are symmetric with respect to the horizontal axis.

\section{Invariant manifolds of infinity for the RPC(N+2)BP}\label{sec:RPCN2BP}

\subsection{Estimates for the distance between the invariant manifolds of $\gamma_{G_0}$}
\label{sec:prova_distancia}

In this section we will use some of the results in~\cite{GMS2016} to deduce Theorem~\ref{thm:Melnikov_domina}. In here the strategy is much simpler, because  $\varepsilon$ is a small parameter, while in~\cite{GMS2016} the analogous parameter was $\OO(1)$.

First, since the invariant manifolds $W_{\infty, G_0}^{s,u}$ are Lagrangian, they can be locally parameterized as graphs of generating functions $S(r, \phi; \mu,\varepsilon, G_0)$, which are solutions of the Hamilton-Jacobi equation
\begin{equation}\label{eq:HJ_RPCN2BP}
    \mathcal{H}(r, \phi, \partial_{r} S, \partial_\phi S; \tmm, G_0) = -G_0^3 , 
\end{equation}
where $\mathcal{H}$ is the Hamiltonian in \eqref{def:Hamiltonian_rescaled_rotating}, with appropriate boundary conditions.

For $\varepsilon = 0$, the Hamilton-Jacobi equation~\eqref{eq:HJ_RPCN2BP} becomes
\[
\frac12 \left(\partial_{r} S\right)^2 + \frac{1}{2 r^2} \left(\partial_\phi S\right)^2 -\frac{1}{r} - G_0^3 \partial_\phi S = -G_0^3.
\]
Thus, the generating function of the invariant manifolds (a homoclinic loop for $\varepsilon = 0$) is given by
\[
S_0(r,\phi) = \phi+ h(r),
\]
where $h$ is any function satisfying
\[
\frac12 \left(h'\right)^2 + \frac{1}{2 r^2}  -\frac{1}{r}  = 0.
\]
The strategy in~\cite{GMS2016} consists of finding solutions of~\eqref{eq:HJ_RPCN2BP} close to $S_0$, introducing $S_1$ by setting $S= S_0+S_1$. To do so, using the parametrization of the homoclinic given by Lemma~\ref{lem:rhyhalphah}, we introduce the change
\begin{equation}
    \label{def:canvi_a_v_xi}
    (r,\phi) = (r_h(v),\xi +\alpha_h(v))
\end{equation}
and the new generating function
\[
T(v,\xi) = S(r_h(v),\xi +\alpha_h(v)),
\]
which we write as $T=T_0+T_1$, where
\[
T_0(v,\xi) = S_0(r_h(v),\xi +\alpha_h(v)), \qquad T_1(v,\xi) = S_1(r_h(v),\xi +\alpha_h(v)).
\]
The associated Hamilton-Jacobi equation is
\begin{multline}
\label{eq:HJ}
\mathcal{H}\left(r_h(v),\xi+\alpha_h(v),y_h^{-1}(v)\left(\partial_v (T_0+T_1)(v,\xi)-r_h(v)^{-2}\partial_\xi (T_0+T_1)(v,\xi)\right),\partial_\xi (T_0+T_1)(v,\xi)\right) \\ = -G_0^3,
\end{multline}
with boundary conditions
\[
\begin{aligned}
\lim_{v\to -\infty} \partial_v T_1^u(v,\xi) & = \lim_{v\to -\infty} \partial_\xi T_1^u(v,\xi) = 0, \\
\lim_{v\to +\infty} \partial_v T_1^s(v,\xi) & = \lim_{v\to +\infty} \partial_\xi T_1^s(v,\xi) = 0,
\end{aligned}
\]
for the unstable and stable manifolds, respectively. 
If solutions $T^{u,s} = T_0+T_1^{u,s}$ of this equation are found, with the corresponding boundary conditions are found,  then
\[
\Gamma^{u,s}(v,\xi) = 
\begin{pmatrix}
    r_h(v)\\
    y_h(v) + y_h(v)^{-1}\left( \partial_v T^{u,s}_1(v,\xi)-r_h(v)^{-2}\partial_\xi T^{u,s}_1(v,\xi)\right)\\
    \xi + \alpha_h(v) \\
    1 + \partial_\xi T^{u,s}_1(v,\xi)
\end{pmatrix}
\]
are parametrizations of the unstable and stable manifolds of $\gamma_{G_0}$.

\begin{prop}
    \label{prop:Delta_de_T} Let $G_0 >1$. Then, for any $0 < v_- < v_+$, the Hamilton-Jacobi equation~\eqref{eq:HJ} has solutions $T^{u,s} = T_0+T_1^{u,s}$, defined and analytic in $(-\infty,v_+)\times \T$ and $(v_-,\infty)\times \T$, respectively, satisfying the corresponding boundary conditions. They satisfy, for $(v,\xi) \in (v_-,v_+)\times \T$,
    \[
    |T_1^u(v,\xi,\mu,\varepsilon,G_0)-T_1^s(v,\xi,\mu,\varepsilon,G_0) - \varepsilon L(\xi+G_0^3v,\mu,\varepsilon,G_0) -E| \le C \varepsilon^2,
    \]
    where $L$ is the Melnikov potential defined in~\eqref{def:Melnikov}.
\end{prop}

\begin{proof}
    The proof is an immediate consequence of the results in~\cite{GMS2016}. Let us briefly explain the derivation. We first observe that the perturbative potential $\varepsilon V$ in~\eqref{def:V_N2} satisfies the hypothesis of Lemma~5.2 in~\cite{GMS2016}. We emphasize that the role of the parameter $\varepsilon$ here is played by the parameter $\mu$ in~\cite{GMS2016}. This implies that Theorem~3.2 of~\cite{GMS2016} holds in our case. This, in turn, implies the proposition. 
\end{proof}

Fixing now a section $\{\phi=\phi_0\}$, the two-dimensional invariant manifolds $W_{\infty, G_0}^{s,u}$ become one-dimensional invariant curves $\gamma^{s,u}$ of the Poincaré map 
$\mathcal{P}_{G_0, \phi_0}$ in~\eqref{def:mapa_de_Poincare}. Hence, the intersection of the invariant manifolds of $\gamma_{G_0}$ (in the energy level $-G_0^3$) with the section gives rise to the invariant curves $\gamma^{s,u}$ with parameterizations of the form 
\begin{equation}\label{eq:param_inv_curves_RPCN2BP} 
\begin{aligned}
r &= r_h(v) \\
y &= Y_{\phi_0}^{s,u}(v) \\ 
&= y_h(v)- y_h(v)^{-1}\left(\partial_v T_1^{s,u}(v,\xi) - r_h(v)^{-2} \partial_\xi T_1^{s,u}(v,\xi)\right)\Big|_{\xi=\phi_0-\alpha_h(v)}.
\end{aligned}
\end{equation}
Hence, Theorem~\ref{thm:Melnikov_domina} follows immediately from Proposition~\ref{prop:Delta_de_T} and~\eqref{eq:param_inv_curves_RPCN2BP}.

\subsection{Expansion of the potential $\varepsilon V$}

In this section we present a first expansion of the potential $\varepsilon V$ that we will need to compute the Melnikov potential. Recall that the rescaled mass parameters $\tmm =(\tilde m_0,\dots, \tilde m_N)$ are functions of $\mu$ and $\varepsilon$ through the formulas in~\eqref{relacions_masses_tilde_i_epsilon} and the positions of the primaries in the central configuration are given by $\tilde q_j(\mu,\varepsilon)$ in Proposition~\ref{prop:solucio_circular_periode_2pi} and Theorem~\ref{thm:posicions_respecte_a_les_masses_s2}.

\begin{prop}
\label{prop:primera_expansió_del_potencial}
The potential $\varepsilon V$ in~\eqref{def:V_N2} can be written as
\begin{equation}
\label{eq:potentialV_series}
 V(r, \phi; \mu,\varepsilon, G_0) = \frac{1}{2} V^{[0]} (r; \mu,\varepsilon, G_0) + \sum_{k \in \mathbb{N}} V^{[k]}(r; \mu,\varepsilon, G_0) \cos(k \phi)
\end{equation}
with 
\begin{equation}\label{eq:potentialV_series_coefficients}
    V^{[k]}(r; \mu,\varepsilon, G_0) = 2  \sum_{\ell \ge F(k)} \binom{-1/2}{\ell -k} \binom{-1/2}{\ell} \varepsilon^{-1} D_{\ell-k,\ell}( \mu,\varepsilon) \frac{1}{G_0^{2(2\ell-k)} r^{2\ell-k+1}}
\end{equation}
where
\begin{equation}
	D_{n,s} ( \mu,\varepsilon) =  (-1)^{n+s} \sum_{j=0}^{N} \tilde{m}_j \, \tilde{q}_j^n(\mu,\varepsilon) \overline{\tilde{q}_j}^s(\mu,\varepsilon).
\end{equation}
and 
\begin{equation}\label{def:F_expansionV}
    F(k) \coloneqq \begin{cases}
	1, &k = 0 \\
	2, &k = 1 \\
	k, &k \ge 2
\end{cases} .
\end{equation}
Furthermore,
\begin{equation}\label{eq:symmetry_Dns}
	D_{n,s}=\overline{D_{n,s}}=D_{s,n} .
\end{equation}
\end{prop}

\begin{proof}
We define, for $A \in \mathbb{R}$ and $B \in \mathbb{C}$, 
\begin{equation}
	\label{def:VAB}
	V_{A, B} (r,\phi)= \frac{A}{r \, \left|1 + B \, \frac{1}{G_0^2 r} e^{-i \phi}\right|} .
\end{equation}
Using  $\displaystyle {(1+z)}^{-\frac{1}{2}} = \sum_{l=0}^{\infty} \binom{-1/2}{l} z^l$, for $|z| < 1$, we have that
\begin{align*}
	V_{A, B} (r,\phi)& = \frac{A}{r} \frac{1}{\left( 1 + B \, \frac{1}{G_0^2 r} e^{-i \phi} \right)^{1/2}  \left( 1 + \overline{B} \, \frac{1}{G_0^2 r} e^{i \phi} \right)^{1/2}}  \\
	&= \frac{A}{r} \sum_{l \geq 0} \binom{-1/2}{l} B^l \frac{1}{G_0^{2l} r^l} e^{-i l (\alpha- t)} \sum_{s \geq 0} \binom{-1/2}{s} \overline{B}^s \frac{1}{G_0^{2m} r^s} e^{i s \phi} \\
	&= \sum_{l, s \geq 0} \binom{-1/2}{l} \binom{-1/2}{s} A \, B^{l} \overline{B}^s \frac{1}{G_0^{2(l+s)} r^{l+s+1}} e^{ -i (l-s) \phi} \\
	&= \sum_{n\ge 0} \sum_{ s =0}^n  \binom{-1/2}{n-s} \binom{-1/2}{s} A \, B^{n-s} \overline{B}^s \frac{1}{G_0^{2n} r^{n+1}} e^{ -i (n-2s) \phi} .
\end{align*}

Hence, defining
\[
	A_j = \tilde{m}_j \,, \quad B_j = -\tilde{q}_j(\mu,\varepsilon),
\]
we have that
\begin{equation}\label{eq:potential_V_expanded}
\begin{aligned}
      V(r, \phi; \mu,\varepsilon, G_0) &= - \frac{1}{r} + \sum_{j=0}^{N} V_{A_j, B_j} (r,\phi) \\
     &= - \frac{1}{r} + \sum_{n\ge 0} \sum_{ s =0}^n  \binom{-1/2}{n-s} \binom{-1/2}{s} \varepsilon^{-1} D_{n-s,s} (\mu,\varepsilon) \frac{1}{G_0^{2n} r^{n+1}} e^{ -i (n-2s) \phi}
\end{aligned}
\end{equation}
where
\begin{equation}\label{eq:def_Dns}
	D_{n,s}(\mu,\varepsilon) = \sum_{j=0}^{N} A_j \, B_j^{n} \overline{B_j}^s = (-1)^{n+s} \sum_{j=0}^{N} \tilde{m}_j \, \tilde{q}_j^n (\mu,\varepsilon)\overline{\tilde{q}_j}^s(\mu,\varepsilon).
\end{equation}

Note also that by definition $\overline{D_{n,s}}=D_{s,n}$ and that, being the central configuration symmetric with respect to the horizontal axis, we have that $D_{n,s}$ is always a real number. In conclusion, 
\[
D_{n,s}=\overline{D_{n,s}}=D_{s,n} .
\]

Taking into account that $D_{0,0}=\sum_{j=0}^N \tilde{m}_j$ and that, since the center of mass of the primaries is located at the origin, $D_{1,0}=-\sum_{j=0}^{N} \tilde{m}_j \tilde{q}_j =0$ and $D_{0,1}=-\sum_{j=0}^{N} \tilde{m}_j \overline{\tilde{q}_j} =0$, a straightforward computation shows that
\[
V(r, \phi; \mu,\varepsilon, G_0) = \sum_{n\ge 2} \sum_{ s =0}^n  \binom{-1/2}{n-s} \binom{-1/2}{s} \varepsilon^{-1} D_{n-s,s} (\mu,\varepsilon) \frac{1}{G_0^{2n} r^{n+1}} e^{ -i (n-2s) \phi} .
\]

Introducing the new set of indices
\[
\left\{
\begin{aligned}
	k & = -n+ 2s, \\
	\ell & = s,
\end{aligned}
\right.
\quad \text{if $k\ge 0$}, \qquad
\left\{
\begin{aligned}
	k & = -n + 2s, \\
	\ell & = n-s,
\end{aligned}
\right.
\quad \text{if $k< 0$}, 
\]
we have that
\begin{equation}
\label{eq:potentialV_series_prelimversion_1}
\begin{aligned}
 V(r, \phi; \mu,\varepsilon, G_0) & = \sum_{k \ge 0, \; \ell \ge |k|}   \binom{-1/2}{\ell -k} \binom{-1/2}{\ell} \varepsilon^{-1} D_{\ell-k,\ell}(\mu,\varepsilon) \frac{1}{G_0^{2(2\ell-k)} r^{2\ell-k+1}} e^{ i k \phi} \\
& + \sum_{k \le -1, \; \ell \ge |k|}   \binom{-1/2}{\ell} \binom{-1/2}{k+\ell} \varepsilon^{-1} D_{\ell,k+\ell} (\mu,\varepsilon)\frac{1}{G_0^{2(2\ell+k)} r^{2\ell+k+1}} e^{ i k \phi},
\end{aligned}
\end{equation}
or, equivalently,
\begin{equation}
\label{eq:potentialV_series_prelimversion_2}
\begin{aligned}
\varepsilon V(r, \phi; \mu,\varepsilon, G_0) & = \sum_{k \ge 0, \; \ell \ge k}   \binom{-1/2}{\ell -k} \binom{-1/2}{\ell} D_{\ell-k,\ell}(\mu,\varepsilon) \frac{1}{G_0^{2(2\ell-k)} r^{2\ell-k+1}} e^{ i k \phi} \\
& + \sum_{k \ge 1, \; \ell \ge k}   \binom{-1/2}{\ell} \binom{-1/2}{\ell-k} D_{\ell,\ell-k}(\mu,\varepsilon) \frac{1}{G_0^{2(2\ell-k)} r^{2\ell-k+1}} e^{ -i k \phi}.
\end{aligned}
\end{equation}

In the above expressions, the terms corresponding to $2\ell - |k|= 0$ and to $2\ell - |k|= 1$ are not present, since they correspond respectively to the terms coming from $n=0$ and $n=1$ in \eqref{eq:potential_V_expanded}. Taking this into account, together with the symmetry~\eqref{eq:symmetry_Dns}, we finally arrive at the claimed expression. 
\end{proof}

\subsection{First expansion of the Melnikov potential}

In this section we find a suitable first order of the Melnikov function $L$ in 

\begin{prop}\label{prop:Fourier_expansion_Melnikov_RPCN2BP}
    The function $L(\beta; \mu,\varepsilon, G_0)$ defined in \eqref{def:Melnikov} satisfies
    \begin{equation}\label{eq:melnikov_final}
          L(\beta; \mu,\varepsilon, G_0) = L^{[0]}(\mu,\varepsilon, G_0) + 2 \sum_{k \in \mathbb{N}} L^{[k]}(\mu,\varepsilon, G_0) \cos(k \beta) ,
    \end{equation}
    where
    \begin{equation}\label{eq:melnikov_final_coeffs}
    \begin{aligned}
	L^{[1]}(\mu,\varepsilon, G_0) &= \varepsilon^{-1} D_{1,2}(\mu,\varepsilon) G_0^{-1} N(1,2,1) \left(1+\bigO\left(G_0^{-1}\right)\right) \\
    &= C_1 \varepsilon^{-1} D_{1,2}(\mu,\varepsilon) G_0^{-3/2} e^{-G_0^3/3} \left(1+\bigO\left(G_0^{-1}\right)\right),\\
	L^{[k]}(\mu,\varepsilon, G_0) &= \varepsilon^{-1} D_{0,k}(\mu,\varepsilon) G_0^{-1} N(k,k,0) \left(1+\bigO\left(G_0^{-1}\right)\right) \\
    &= C_k \varepsilon^{-1} D_{0,k}(\mu,\varepsilon) G_0^{k-3/2} e^{-k G_0^3/3} \left(1+\bigO\left(G_0^{-1}\right)\right), \quad k \ge 2 ,
    \end{aligned}
    \end{equation}
    for some constants $C_k >0$ independent of $\mu$, $\varepsilon$ and $G_0$.
\end{prop}

\begin{proof}
From~\eqref{eq:potentialV_series} and~\eqref{eq:potentialV_series_coefficients}, it is clear that the potential $\varepsilon V$ defined in \eqref{def:V_N2} satisfies
\[
V(r,\phi;\mu,\varepsilon, G_0)=  \mathcal{O}\left( \frac{1}{G_0^4 r^3}\right).
\]
Note that  $V$  satisfies $\mathcal{H} = \mathcal{H}_0 + \varepsilon V$.
We recall that the Melnikov potential in~\eqref{def:Melnikov} is
\begin{equation}\label{def:Melnikov_RPCN2BP}
\begin{aligned}
    L(u,\beta) &\coloneqq \int_{-\infty}^{\infty} V \circ \Phi_0^s \circ \Gamma_0 (u,\beta)\, ds = \int_{-\infty}^{\infty} V  \circ \Gamma_0 (u+s,\beta)\, ds \\
    &= \int_{-\infty}^{\infty} V(r_h(u+s), \beta + \alpha_h(u+s) - G_0^3 (u+s) )\, ds \\
    &= \int_{-\infty}^{\infty} V(r_h(s), \beta + \alpha_h(s) - G_0^3 s )\, ds \eqqcolon L(\beta),
\end{aligned}
\end{equation}
where, to simplify notation, we have skipped the dependence of $V$ on $\mu$, $\varepsilon$ and $G_0$.

Define as in \cite{DKdlRS2019} the integrals
\begin{equation}\label{def:N}
    N(q, m, n) \coloneqq \frac{2^{m+n}}{G_0^{2m+2n-1}} \binom{-1/2}{m} \binom{-1/2}{n} \int_{-\infty}^{\infty} \frac{e^{i q (\tau + \tau^3/3) G_0^3/2}}{(\tau-i)^{2m} (\tau+i)^{2n}} d \tau  .
\end{equation}
Note that, for any $q \in \Z$, $n, m \ge 0$, 
\[N(q, m, n) = \overline{N(q, m, n)} = N(-q, n, m) \,.\]

In order to compute the Melnikov potential, we first take into account \eqref{eq:potentialV_series} so that
\begin{multline*}
 V(r_h(s), \beta + \alpha_h(s) - G_0^3 s; \mu,\varepsilon, G_0) = \sum_{\ell \ge 1} \binom{-1/2}{\ell} \binom{-1/2}{\ell} \varepsilon^{-1} D_{\ell,\ell}(\mu,\varepsilon) \frac{1}{G_0^{4\ell} r_h^{2\ell+1}(s)} \\
    + \sum_{k \in \mathbb{N}} \sum_{\ell \ge F(k)} \binom{-1/2}{\ell -k} \binom{-1/2}{\ell} \varepsilon^{-1} D_{\ell-k,\ell}(\mu,\varepsilon) \frac{e^{i k (\beta + \alpha_h(s) - G_0^3 s)}+e^{-i k (\beta + \alpha_h(s) - G_0^3 s)}}{G_0^{2(2\ell-k)} r_h^{2\ell-k+1}(s)}.
\end{multline*}

We now focus on computing
\[
\begin{aligned}
    I_{k,\ell}(\beta)  &\coloneqq \int_{-\infty}^{\infty} \binom{-1/2}{\ell -k} \binom{-1/2}{\ell} \frac{e^{i k (\beta + \alpha_h(s) - G_0^3 s)}}{G_0^{2(2\ell-k)} r_h^{2\ell-k+1}(s)} ds \\
    &= \frac{1}{G_0^{2(2\ell-k)}}\binom{-1/2}{\ell -k} \binom{-1/2}{\ell} e^{i k \beta} \int_{-\infty}^{\infty} \frac{e^{i k (\alpha_h(s) - G_0^3 s)}}{r_h^{2\ell-k+1}(s)} ds \\
    &= \frac{2^{2 \ell - k}}{G_0^{2(2\ell-k)}}\binom{-1/2}{\ell -k} \binom{-1/2}{\ell} e^{i k \beta} \int_{-\infty}^{\infty}  \frac{e^{-i k (\tau + \tau^3/3) G_0^3/2}}{(\tau-i)^{2(\ell-k)} (\tau+i)^{2 \ell}}  d \tau \\
    &= G_0^{-1} N(-k, \ell-k, \ell) e^{i k \beta} = G_0^{-1} N(k, \ell, \ell-k) e^{i k \beta},
\end{aligned}
\]
where we have used the change of variables $s=(\tau+\tau^3/3)/2$ from Lemma~\ref{lem:rhyhalphah}.

It is immediate to see that $I_{k,\ell}(-\beta) = \overline{I_{k, \ell}(\beta)}$.

Therefore we have that
\begin{equation}\label{eq:Melnikov_computed}
    \begin{aligned}
    L(\beta,\mu,\varepsilon,G_0) = L^{[0]} + 2 \sum_{k \in \mathbb{N}} L^{[k]}(\mu,\varepsilon,G_0) \cos(k \beta)
\end{aligned}
\end{equation}
with
\begin{equation}\label{eq:harmonicsL}
    L^{[k]}(\mu,\varepsilon,G_0) \coloneqq \sum_{\ell \ge F(k)} \varepsilon^{-1} D_{\ell-k,\ell}(\mu,\varepsilon) G_0^{-1} N(k, \ell, \ell - k)
\end{equation}
where, from \eqref{def:F_expansionV},
\[
    F(k) = \begin{cases}
	1, &k = 0 \\
	2, &k = 1 \\
	k, &k \ge 2
\end{cases} 
\]
and, from \eqref{eq:def_Dns} and \eqref{eq:symmetry_Dns},
\[
	D_{n,s}(\mu,\varepsilon) = (-1)^{n+s} \sum_{j=0}^{N} \tilde{m}_j \, \tilde{q}_j^n (\mu,\varepsilon) \overline{\tilde{q}_j}^s (\mu,\varepsilon)\in \mathbb{R},
\]
where the rescaled mass parameters $\tmm =(\tilde m_0,\dots, \tilde m_N)$ are functions of $\mu$ and $\varepsilon$ through the formulas in~\eqref{relacions_masses_tilde_i_epsilon} and  $\tilde q_j(\mu,\varepsilon)$ are defined in Proposition~\ref{prop:solucio_circular_periode_2pi}.  Hence,
\[
D_{n,s}(\mu,\varepsilon) = \frac{M^{-(n+s)/3} \SA(\mm, N)^{n+s}}{M} (-1)^{n+s} \sum_{j=0}^{N} m_j \, q_j^n(\mu,\varepsilon) \overline{q_j}^s(\mu,\varepsilon),
\]
where the positions $q_j(\mu,\varepsilon)$ are given in Theorem~\ref{thm:posicions_respecte_a_les_masses_s2}. We recall that $m_j = 1+\mu_j$, $j=0,\dots, N-1$,
and $m_N = 1/\varepsilon$.

Given $(m_0,\ldots,m_{N-1})$ masses of the bodies in the ring, define the permutation
\[\mathfrak{P}(m_0,\ldots,m_{N-1})=(m_{N-1},m_0,m_1,\ldots,m_{N-2}).
\]
By the definition of $\SA$, it is clear that 
\[
\SA(m_{N-1},m_0,m_1,\ldots,m_{N-2}, m)=\SA(\mathfrak{P}(m_0,\ldots,m_{N-1}),m)=\SA(m_0,\ldots,m_{N-1},m).
\]
Here it must be taken into account that the permutation $\mathfrak{P}$ of the masses of the bodies in the ring only causes at most a rotation of the whole central configuration, which does not affect the value of $\SA$.

The aforementioned symmetry of $\SA$ implies that in the expansion in $\mu$ of $\SA$, the linear terms in $\mu$ can only appear in the form $c \sum_{j=0}^{N-1} \mu_j$, for some common constant $c$. 

Therefore, assuming $\sum_{j=0}^{N-1} \mu_j = 0$, we have that 
\[
\begin{aligned}
    \frac{M^{-(n+s)/3} \SA(\mm, N)^{n+s}}{M} = \frac{(\varepsilon^{-1}+N)^{-(n+s)/3} \varepsilon^{-(n+s)/3}}{\varepsilon^{-1}+N} \left(1 + \mathcal{O}(\lVert \mu \rVert^2) + \mathcal{O}(\varepsilon) \right) .
\end{aligned}
\]

In conclusion, 
\begin{equation}\label{eq:simplified_values_Dns}
    D_{n,s}(\mu,\varepsilon) = (-1)^{n+s} \varepsilon \left(\sum_{j=0}^{N} m_j \, q_j^n(\mu,\varepsilon) \overline{q_j}^s(\mu,\varepsilon) \right) \left(1 + \mathcal{O}(\lVert \mu \rVert^2) + \mathcal{O}(\varepsilon) \right).
\end{equation}

Now, just out of convenience, we make the change of indices $n=\ell-k$ in \eqref{eq:harmonicsL}, arriving at
\[L^{[k]}(\mu,\varepsilon) = \sum_{n \ge \delta_0(k)} \varepsilon^{-1} D_{n, n+k}(\mu,\varepsilon) G_0^{-1} N(k, n+k, n) ,\]
where 
\[\delta_0(k) = \begin{cases}
    1 & \text{if} \quad k=0,1 \\
	0 & \text{if} \quad k \ge 2 .
\end{cases}\]

From Proposition~19 in~\cite{DKdlRS2019}, we obtain that there exists a constant $C>0$ such that, for any $k \ge 1$ and $n \ge 0$,
\[|N(k, n+k, n)| \le C^{n+k} G_0^{k-n-1/2} e^{-k G_0^3/3}\]
and, therefore,
\[|L^{[k]}| \le {(\tilde{C} G_0)}^{k-3/2} e^{-k G_0^3/3},
\]
for some constant $\tilde{C} > 0$.
These bounds are indeed sharp, as can be checked from Proposition~22 of~\cite{DKdlRS2019}.
\end{proof}

\subsection{Expansion of the coefficients of the Melnikov potential in the masses}
Using the expressions for the positions of the primaries as a function of their masses obtained in Section~\ref{sec:family_CC}, here we expand to first order in $\mu$ the coefficients $D_{1,2}$ and $D_{0,k}$, $k\ge2$ that appear in the Melnikov potential in Proposition~\ref{prop:Fourier_expansion_Melnikov_RPCN2BP}, see~\eqref{eq:melnikov_final}, \eqref{eq:melnikov_final_coeffs}, and~\eqref{eq:simplified_values_Dns}. These coefficients depend on $\mu$ and $\varepsilon$, but are independent of $G_0$.

\begin{prop}\label{prop:expansion_Melnikov_positions}
Let $D=(D_1,D_2,\ldots,D_{N-1})^\top$, where $D_1 \coloneqq D_{1,2}$ and $D_k \coloneqq D_{0,k}$ for $k=2,\ldots,N-1$ are the coefficients that appear in the Melnikov potential in Proposition~\ref{prop:Fourier_expansion_Melnikov_RPCN2BP}.

Then
\begin{equation}\label{eq:expansion_Melnikov_positions}
    D = \mathrm{diag}\,(d_1,\ldots,d_{N-1}) \hat{\overline{\M}} \mu,
\end{equation}
where $\mu=(\mu_0,\ldots,\mu_{N-1})^\top$, the matrix $\hat{\overline{\M}}$ is defined by $\hat{\overline{\M}}_{k,j} = \omega^{-kj}$, $k= 1,\dots,N-1$, $j=0,\ldots,N-1$, with $\omega = e^{2 \pi i/N}$, and 
\[
d_k \coloneqq (-1)^{k}\varepsilon \left(1- i k c_{k}\right) \left(1+\mathcal{O}_2(\mu) + \mathcal{O}(\varepsilon) \right) , \quad k=1,\ldots,N-1
\]
where $c_k$ are given in Theorem~\ref{thm:posicions_respecte_a_les_masses}.
\end{prop}

\begin{proof}
    First, from Theorem~\ref{thm:posicions_respecte_a_les_masses}, it is immediate to obtain that
\[
\begin{aligned}
    q_N &= \mathcal{O}(\varepsilon) , \\
    q_j &= \omega^j(1+i \tilde{\theta}_j) + \mathcal{O}(\varepsilon)+ \mathcal{O}_2(\mu), \quad j=0,\ldots,N-1,
\end{aligned}
\]
where
\[
\tilde{\theta}_j = \sum_{r=0}^{N-1} \left(\frac{1}{N} \sum_{k=1}^{N-1} c_k (\omega^{k(j-r)} - \omega^{-k r}) \right) \mu_r , \quad j=0,\ldots,N-1 ,
\]
with $c_k$ given in Theorem~\ref{thm:posicions_respecte_a_les_masses}.

We want to expand, up to first order in $\mu$, the expression in \eqref{eq:simplified_values_Dns} given by
\[
D_{n,s} = (-1)^{n+s} \varepsilon \left(\sum_{j=0}^{N} m_j \, q_j^n \overline{q_j}^s \right) \left(1 + \mathcal{O}_2(\mu) + \mathcal{O}(\varepsilon) \right) .
\]

From now on we assume that $n, s \ge 0$, $n+s \ge 2$, $1 \le s-n \le N-1$, as is the case for the coefficients we are interested in.

From the expressions of the positions, we arrive at 
\[
    D_{n,s} = (-1)^{n+s} \varepsilon \left(\sum_{j=0}^{N-1} \omega^{j(n-s)} (\mu_j+i (n-s) \tilde{\theta}_j ) +\mathcal{O}_2(\mu) + \mathcal{O}(\varepsilon) \right) .
\]

Define
\[
\nu_k = \sum_{j=0}^{N-1} \omega^{-k j} \mu_j , \quad k=0,\ldots,N-1.
\]
We have that
\[
\begin{aligned}
    \sum_{j=0}^{N-1} \omega^{j(n-s)} \tilde{\theta}_j &= \sum_{j=0}^{N-1} \omega^{j(n-s)} \sum_{r=0}^{N-1} \left(\frac{1}{N} \sum_{k=1}^{N-1} c_k (\omega^{k(j-r)} - \omega^{-k r}) \right) \mu_r \\
    &= \frac{1}{N} \sum_{r=0}^{N-1} \left(\sum_{k=1}^{N-1} c_k \sum_{j=0}^{N-1} (\omega^{k(j-r)+j(n-s)} - \omega^{-k r+j(n-s)}) \right) \mu_r \\
    &= \frac{1}{N} \sum_{r=0}^{N-1} \left(\sum_{k=1}^{N-1} \omega^{-k r} c_k \sum_{j=0}^{N-1} (\omega^{j(k+n-s)} - \omega^{j(n-s)}) \right) \mu_r \\
    &= \frac{1}{N} \sum_{r=0}^{N-1} \left(\sum_{k=1}^{N-1} \omega^{-k r} c_k \sum_{j=0}^{N-1} \omega^{j(k+n-s)} \right) \mu_r = \frac{1}{N} \sum_{k=1}^{N-1} c_k \sum_{j=0}^{N-1} \omega^{j(k+n-s)} \sum_{r=0}^{N-1} \omega^{-k r} \mu_r \\
    &= \frac{1}{N} \sum_{k=1}^{N-1} c_k \sum_{j=0}^{N-1} \omega^{j(k+n-s)} \nu_k = \frac{1}{N} \sum_{k=1}^{N-1} c_k  \nu_k \sum_{j=0}^{N-1} \omega^{j(k+n-s)} = c_{s-n} \nu_{s-n} .
\end{aligned}
\]
In conclusion,
\[
 D_{n,s} = (-1)^{s-n}\varepsilon \left(1- i (s-n) c_{s-n} \right) \nu_{s-n} \left(1+\mathcal{O}_2(\mu) + \mathcal{O}(\varepsilon) \right).
\]
Therefore, for $k=1,\ldots,N-1$,
\[
\begin{aligned}
    D_k &= (-1)^{k}\varepsilon \left(1- i k c_{k} \right) \left(1+\mathcal{O}_2(\mu) + \mathcal{O}(\varepsilon) \right) \nu_{k}\\
    &= (-1)^{k}\varepsilon \left(1- i k c_{k}\right) \left(1+\mathcal{O}_2(\mu) + \mathcal{O}(\varepsilon) \right) \sum_{j=0}^{N-1} \omega^{-kj} \mu_j.
\end{aligned}
\]

\end{proof}

\subsection{The rank of the Jacobian matrix of Melnikov coefficients is maximal}
\begin{thm}\label{thm:maximal_rank_jacobian}
Let $D$ be as in Proposition~\ref{prop:expansion_Melnikov_positions}. Then, for any integer $N\ge 3$ and $\mu$ small enough, the Jacobian matrix $\partial_\mu D$ has maximal rank, $\mathrm{rank}(\partial_\mu D)= N-1$.
\end{thm}
Taking into account Proposition~\ref{prop:expansion_Melnikov_positions}, in order to prove Theorem~\ref{thm:maximal_rank_jacobian} it is enough to check that \begin{equation}\label{eq:def_Ck}
    \mathcal{C}_k(N) \coloneqq i k c_k    
\end{equation}
satisfies $\mathcal{C}_k(N) \neq 1$ for every $N\ge 3$ and $k\in\{1,\dots,N-1\}$.

Let us start by considering, for $N \ge 3$ and $k\in\{1,\dots,N-1\}$,
\begin{equation}\label{eq:def_Rk}
R_k(N) \coloneqq k\,\frac{\mathcal A_k(N)}{\mathcal B_k(N)},
\end{equation}
where
\begin{equation}\label{eq:def_Ak}
    \mathcal{A}_k(N) \coloneqq \sum_{s=1}^{N-1}\frac{\cos(\pi s/N)}{\sin^2(\pi s/N)}\,
\sin(2 k \pi s/N)
\end{equation}
and
\begin{equation}\label{eq:def_Bk}
    \mathcal{B}_k(N) \coloneqq \sum_{s=1}^{N-1}\frac{1+\cos^2(\pi s/N)}{\sin^3(\pi s/N)}\,
\sin^2(k \pi s/N).
\end{equation}

\begin{lem}\label{lem:BkA_identity}
    For any $N \ge 3$ and $k\in\{1,\ldots,N-1\}$, the following identity holds:
    \[
    \mathcal{B}_k(N) - k \mathcal{A}_k(N) = \sum_{j=1}^{k} (2j-1)(2k-2j+1) \cot\left(\frac{(2j-1) \pi}{2 N}\right),
    \]
    where $\cot$ denotes the cotangent.
\end{lem}
\begin{proof}
    Fix $N\ge 3$ and denote
    \[
    x_s \coloneqq \frac{\pi s}{N}.
    \]
    Taking into account the trigonometric identity
    \[
    \sin((2k+1) x) + \sin((2k-1)x) = 2 \sin(2k x) \cos(x),
    \]
    we can write $\mathcal{A}_k$ defined in \eqref{eq:def_Ak} as
    \begin{equation}\label{eq:AkTk}
        \mathcal{A}_k = \frac{T_{k+1}+T_{k}}{2},
    \end{equation}
    where
    \begin{equation}\label{eq:def_T}
    T_k \coloneqq \sum_{s=1}^{N-1} \frac{\sin((2 k - 1) x_s)}{\sin^2(x_s)}, \quad k \ge 1.
    \end{equation}
    Let
     \begin{equation}\label{eq:def_E}
    E_k \coloneqq \sum_{s=1}^{N-1} \frac{\cos(2 k x_s)}{\sin(x_s)}, \quad k \ge 0 .
    \end{equation}
    
    Then, taking into account the trigonometric identity
    \[
    \sin((2k+1)x) - \sin((2k-1)x) = 2 \sin(x) \cos(2 k x) ,
    \]
    it is immediate that
    \begin{equation}\label{eq:TkEk}
        T_{k+1} - T_k = 2 E_k , \quad k \ge 1 .
    \end{equation}
    Note that $E_0 = T_1$. Thus, from \eqref{eq:TkEk}, we arrive at
    \begin{equation}\label{eq:TE}
    T_k = E_0 + 2 \sum_{j=1}^{k-1} E_j , \quad k \ge 1,
    \end{equation}
    where, for $k=1$, we interpret the empty sum as $0$.
    
    In conclusion, from \eqref{eq:AkTk} and \eqref{eq:TE}, we obtain that
    \begin{equation}\label{eq:AE}
    \mathcal{A}_k = E_0 + E_{k} + 2 \sum_{j=1}^{k-1} E_j . 
    \end{equation}

    For $\mathcal{B}_k$ defined in \eqref{eq:def_Bk}, we use that $\cos^2(x)= 1 - \sin^2(x)$ together with the identity
    \[
    \sum_{j=1}^{k} \sin((2j-1) x) = \frac{\sin^2(k x)}{\sin(x)}
    \]
    to obtain
    \begin{equation}\label{eq:BTU}
    \mathcal{B}_k = 2 \sum_{j=1}^{k} T_j - \sum_{j=1}^{k} U_j ,
    \end{equation}
    where
    \begin{equation}\label{eq:def_U}
    U_j \coloneqq \sum_{s=1}^{N-1} \sin((2j-1) x_s), \quad j \ge 1.
    \end{equation}
    Note that
    \begin{equation}\label{eq:U_cot}
        U_j = \cot\left(\frac{(2j-1)\pi}{2N}\right).
    \end{equation}

    Moreover, taking into account \eqref{eq:def_E}, the trigonometric identity
    \[
    \cos(2 (j-1) x)-\cos(2 j x) = 2 \sin(x) \sin((2j-1) x)
    \]
    and \eqref{eq:def_U}, we deduce that
    \begin{equation}\label{eq:EjUj}
    E_{j-1}-E_{j} = 2 \sum_{s=1}^{N-1} \sin((2 j -1) x_s) = 2 U_j .
    \end{equation}
    Therefore, from \eqref{eq:BTU}, \eqref{eq:TE} and \eqref{eq:EjUj}
    \begin{equation}\label{eq:BE}
    \begin{aligned}
    \mathcal{B}_k &= 2 \sum_{j=1}^{k} \left(E_0 + 2 \sum_{r=1}^{j-1} E_r \right) - \frac{1}{2} \sum_{j=1}^{k} (E_{j-1}-E_{j}) \\
    &= 2 k E_0 - \frac{E_0-E_k}{2} + 4 \sum_{j=1}^{k-1} (k-j) E_j .
    \end{aligned}
    \end{equation}
    From \eqref{eq:AE} and \eqref{eq:BE}, we obtain
    \[
    \mathcal{B}_k - k \mathcal{A}_k = (k - 1/2) E_0 - (k-1/2) E_k + 2 \sum_{j=1}^{k-1} (k-2j) E_j 
    \]
    If we call $h_{j,k}\coloneqq (2j-1)(2k-2j+1)$, then
    \[
    \begin{aligned}
        \mathcal{B}_k - k \mathcal{A}_k &= \frac{1}{2} h_{1,k} E_0 - \frac{1}{2} h_{k,k} E_k + \frac{1}{2} \sum_{j=1}^{k-1} (h_{j+1,k}-h_{j,k}) E_{j} \\
        &= \frac{1}{2}  \sum_{j=1}^{k} h_{j,k} (E_{j-1}-E_{j}) .
    \end{aligned}
    \]
    Finally, taking into account \eqref{eq:EjUj} and \eqref{eq:U_cot},
    \[
    \mathcal{B}_k - k \mathcal{A}_k = \sum_{j=1}^{k} h_{j,k} \cot\left(\frac{(2j-1)\pi}{2N}\right),
    \]
    with $h_{j,k}=(2j-1) (2k-2j+1)$.
\end{proof}

\begin{thm}\label{thm:Rk_not_1}
    For any $N \ge 3$ and $k\in\{1,\ldots,N-1\}$, we have that $R_k(N) < 1$.
\end{thm}
\begin{proof}
    Clearly, from its definition in \eqref{eq:def_Bk}, we have that $\mathcal{B}_k(N) > 0$. Thus $R_k(N) < 1$ is equivalent to $\mathcal{B}(N) - k \mathcal{A}_k(N) > 0$. 

    From Lemma~\ref{lem:BkA_identity}, we have that
    \[
    \mathcal{B}_k(N) - k \mathcal{A}_k(N) = \sum_{j=1}^{k} h_{j,k} \cot\left(\frac{(2j-1) \pi}{2 N}\right),
    \]
    where the coefficients $h_{j,k} \coloneqq (2j-1)(2k-2j+1)$ satisfy the symmetry $h_{k+1-j,k}=h_{j,k}$.

    Therefore
    \[
    \mathcal{B}_k(N) - k \mathcal{A}_k(N) = \frac{1}{2} \sum_{j=1}^{k} h_{j,k} \left(\cot\left(\frac{(2j-1) \pi}{2 N}\right) +\cot\left(\frac{(2k-2j+1) \pi}{2 N}\right) \right).
    \]

    Using the trigonometric identity $\cot(x)+\cot(y)=\frac{\sin(x+y)}{\sin(x) \sin(y)}$, we arrive at
    \[
    \mathcal{B}_k(N) - k \mathcal{A}_k(N) = \frac{1}{2} \sin(k \pi/N) \sum_{j=1}^{k} (2j-1)(2k-2j+1) \frac{1}{\sin\left(\frac{(2j-1) \pi}{2 N}\right)} \frac{1}{\sin\left(\frac{(2k-2j+1) \pi}{2 N}\right)} > 0.
    \]
\end{proof}

\begin{prop}\label{prop:negativity_Ck}
    For $N \ge 12$ and $k=N-1$, we have that $\mathcal{C}_k(N)$ defined in \eqref{eq:def_Ck} satisfies 
    \[
    \mathcal{C}_{N-1}(N) < 0 , \quad \text{ for } \, N \ge 12.
    \]
\end{prop}
\begin{proof}
    From its definition,
    \[
    \mathcal{C}_{N-1}(N) = (N-1) \frac{\mathcal{A}_{N-1}(N) + 2N}{\mathcal{B}_{N-1}(N) - 2N} ,
    \]
    with
    \[
    \begin{aligned}
        \mathcal{A}_{N-1}(N) &= -2 \sum_{s=1}^{N-1} \frac{\cos^2(\pi s/N)}{\sin(\pi s/N)} ,\\
        \mathcal{B}_{N-1}(N) &= \sum_{s=1}^{N-1} \frac{1+\cos^2(\pi s/N)}{\sin(\pi s/N)} .
    \end{aligned}
    \]
    Hence
    \[
    \mathcal{A}_{N-1}(N)+\mathcal{B}_{N-1}(N) = \sum_{s=1}^{N-1} \frac{1-\cos^2(\pi s/N)}{\sin(\pi s/N)} = \sum_{s=1}^{N-1} \sin(\pi s/N) > 0.
    \]
    Therefore $\mathcal{B}_{N-1}(N) > -\mathcal{A}_{N-1}(N)$ and the result will be obtained once we prove that $\mathcal{A}_{N-1}(N) <- 2N$ for $N \ge 12$.

    First, pairing the equal terms with indices $s$ and $N-s$, we obtain
    \begin{equation}\label{eq:A1_pairing}
    \mathcal{A}_{N-1}(N) = - 4 \sum_{s=1}^{\lfloor(N-1)/2 \rfloor} \frac{\cos^2(\pi s/N)}{\sin(\pi s/N)} .
    \end{equation}

    Now, defining the auxiliary function $f \colon (0, \pi/2) \to \mathbb{R}^+$,
    \[
    f(t) \coloneqq \frac{t \cos^2(t)}{\sin(t)} ,
    \]
    we can write
    \begin{equation}\label{eq:A1_auxf}
        \mathcal{A}_{N-1}(N) = - \frac{4N}{\pi} \sum_{s=1}^{\lfloor(N-1)/2 \rfloor} \frac{f(\pi s/N)}{s}.
    \end{equation}

    A direct computation gives that $f$ is strictly decreasing on $(0, \pi/2)$.

    In what follows we assume $N \ge 12$. Note that $0 < \pi s/N < \pi/2$ for $1 \le s \le \lfloor (N-1)/2 \rfloor$ and that, since $f$ is decreasing, $f(\pi s/N) \ge f(\pi s/12)$.

    From \eqref{eq:A1_auxf} it is immediate that
    \[
    \mathcal{A}_{N-1}(N) \le -\frac{4N}{\pi}\sum_{s=1}^{5} \frac{f(\pi s/N)}{s} \le -\frac{4N}{\pi} \sum_{s=1}^{5} \frac{f(\pi s/12)}{s} = -\frac{N}{3} \sum_{s=1}^{5} \frac{\cos^2(\pi s/12)}{\sin(\pi s/12)} = \frac{N}{12} \mathcal{A}_{11}(12),
    \]
    where in the last equality we have used \eqref{eq:A1_pairing}.

    In conclusion,
    \[
    \mathcal{A}_{N-1}(N) \le \frac{N}{12} \mathcal{A}_{11}(12).
    \]

    Since $\mathcal{A}_{11}(12) = -(6+2 \sqrt{2} + 2\sqrt{3}/3 + 6 \sqrt{6}) < -24$, we conclude that $\mathcal{A}_{N-1}(N) < -2N $ for $N \ge 12$.
\end{proof}

\begin{proof}[Proof of Theorem~\ref{thm:maximal_rank_jacobian}]
    We recall that, by Proposition~\ref{prop:expansion_Melnikov_positions}, it is enough to check that $\mathcal{C}_k(N) \coloneqq i k c_k$ satisfies $\mathcal{C}_k(N) \neq 1$ for every $N\ge 3$ and $k\in\{1,\dots,N-1\}$.
    
    Note that $\mathcal{C}_k(N) = R_k(N)$ for $2 \le k \le N-2$, while for $k=1, N-1$ we have that
    \[
    \begin{aligned}
        \mathcal{C}_1(N) &= \frac{\mathcal{A}_1(N) -2N}{\mathcal{B}_1(N) - 2N} , \\
        \mathcal{C}_{N-1}(N) &= (N-1) \frac{\mathcal{A}_{N-1}(N) +2N}{\mathcal{B}_{N-1}(N) - 2N} .
    \end{aligned}
    \]
    Therefore, Theorem~\ref{thm:Rk_not_1} immediately gives that $\mathcal{C}_k(N) \neq 1$ for $1 \le k \le N-2$. 

    Finally, Proposition~\ref{prop:negativity_Ck} gives that $\mathcal{C}_{N-1}(N) \neq 1$ for $N \ge 12$. The cases $N=3,\ldots,11$ can be checked through numerical computations with sufficient precision.
\end{proof}

\section{Tangencies of arbitrarily high order between the invariant manifolds of infinity}\label{sec:tangencies}

The following theorem implies Theorem~\ref{thm:tangencies_sec2}.

\begin{thm}
\label{thm:tangencies}
Let $\ell \in \N$ be odd and fixed  and  $N = 2 \ell +1$.
    For any $G_0\gg 1$, there exists $\varepsilon_0>0$ such that, for any value of the central mass $m > 1/\varepsilon_0$, there exist $\nu_j = \mathcal{O}(G_0^{\ell+1} e^{- G_0^3/3})+\mathcal{O}(1/m)$, $j=1,\dots,\ell$, such that the invariant manifolds of infinity of the $(N+1)+1$ restricted planar circular problem associated to the central configuration with masses $m_j = 1+\mu_j$, $j=0,\dots, N-1$, $m_N = m$, where
    \begin{equation}
    \label{def:mu_nu}
    \mu_j = \begin{cases}
        \nu_j, \qquad & j = 1,\dots,\ell,\\
        \nu_{N-j}, \qquad & j = \ell+1,\dots,N-1,
    \end{cases}
    \end{equation}
    and $\mu_0 = - \sum_{j=1}^{N-1} \mu_j$
    have a tangency of order $\ell$ that splits generically. 
\end{thm}

\begin{rem}
    The claim also holds for even $\ell$, taking $N = 2 \ell$ instead of $N = 2 \ell + 1$ and adapting the proof accordingly.
\end{rem}

\begin{proof} 
We recall that $\varepsilon = 1/m$. In our setting, if $\varepsilon = 0$, the invariant manifolds coincide. Consider the Melnikov potential $L$ in~\eqref{def:Melnikov}. Let $T^{u,s} = T_0+T_1^{u,s}$,
solutions of the Hamilton-Jacobi equation~\eqref{eq:HJ}, be the generating functions of the unstable and stable invariant manifolds of $\gamma_{G_0}$ given by Proposition~\ref{prop:Delta_de_T}. Their difference $\varepsilon \Delta = T^u-T^s$, by Proposition~\ref{prop:Delta_de_T}, satisfies
\begin{equation}
\label{def:Delta_L}
\Delta(v,\xi,\mu(\nu),\beta,\varepsilon,G_0 )  =  L(\xi+G_0^3v,\mu(\nu),\beta,\varepsilon,G_0) + \mathcal{O}_1 (\varepsilon),
\end{equation}
where $\mu(\nu)$ is defined through~\eqref{def:mu_nu}.
The manifolds have a tangency of order $\ell$ at  $\nu = \nu^0$, $\beta = \beta_0$, for some $G_0$ and $\varepsilon $, if 
\begin{equation}
\label{eq:tangency_order_l}
    \mathcal{S}(\nu^0, \beta^0,G_0,\varepsilon) = \begin{pmatrix}\mathcal{S}_1 \\ \vdots \\ \mathcal{S}_{\ell+1} \end{pmatrix}_{\mid  \;\nu  = \nu^0, \beta  = \beta^0 } = \begin{pmatrix}\partial_\beta \Delta\\ \vdots \\ \partial^{\ell+1}_\beta \Delta\end{pmatrix}_{\mid  \;\nu  = \nu^0, \beta  = \beta^0 } = 0
\end{equation}
and $\partial^{\ell+2}_\beta {\Delta}_{\mid  \;\nu  = \nu^0, \beta  = \beta^0 } \neq 0$.
It splits if $\mathrm{rank}\, \partial  (\mathcal{S}_1,\dots,\mathcal{S}_\ell) / \partial \nu_{\mid \nu^0,\beta_0} = \ell$. As a matter of fact, we will see that  $\mathrm{rank}\, \partial  \mathcal{S} / \partial ( \nu,\beta)_{\mid \nu^0,\beta_0} = \ell+1$.

By the standard implicit function theorem, it is enough to check that~\eqref{eq:tangency_order_l} holds for $\varepsilon = 0$, that is, for the Melnikov function $L$.

The rest of the proof consists in finding non-degenerate solution of the corresponding equation for the Melnikov funcion (see~\eqref{eq:tangency} below). This is done  by solving a fixed point equation  around a suitable approximate solution. We emphasize that the equation~\eqref{eq:tangency} is highly non-linear in one of its unknowns. Hence, this suitable approximation is obtained by exploiting the symmetries of the problem. 

We introduce $k\in \N$, $k \ge 1$ by $\ell = 2k -1$.

Let 
\begin{equation}
\label{def:LmenoroigualqueN}
L_{\le N-1} (\beta,\mu,\varepsilon, G_0)= L^{[0]}(\mu,\varepsilon, G_0) + 2 \sum_{k =1}^{N-1} L^{[k]}(\mu,\varepsilon, G_0) \cos(k \beta)
\end{equation}
be the truncation up to harmonic $N-1$ of the Melnikov potential in~\eqref{def:Melnikov_RPCN2BP}  and $\mathcal{R}$ and $L_{>N}$ be such that
\begin{equation}
\label{def:R}
 L = L_{\le N-1} + \mathcal{R}
\end{equation}
and
\begin{equation}
\label{def:LmesgranqueN}
 \mathcal{R}(\beta,\mu,\varepsilon, G_0) = L^{[N]}(\mu,\varepsilon, G_0) \cos(N \beta)+ L_{> N}(\beta,\mu,\varepsilon, G_0).
\end{equation}
In view of Proposition~\ref{prop:Fourier_expansion_Melnikov_RPCN2BP}, we write 
\begin{equation}
    \label{def:DCGamma}
    L_{\le N-1} - L^{[0]} = \Gamma^\top \mathcal{C}  \mathcal{D}
\end{equation}
where
\begin{equation}
\begin{aligned}
    \mathcal{D}^\top & = (D_{1}, D_{2},\dots,D_{N-1}), \\
    \mathcal{C} & = \mathrm{diag}\,(L^{[1]}/D_{1},\dots,L^{[N-1]}/D_{N-1}), \\
    \Gamma(\beta)^\top & = (\cos \beta, \dots, \cos (N-1) \beta),
\end{aligned}
\end{equation}
and the coefficients $D_j$, $j=1,\dots N-1$ where introduced in Proposition~\ref{prop:expansion_Melnikov_positions}. We remark that $\mathcal{D}$ collects all the dependence on the masses and is independent of $G_0$. 

First we show that we can use $\mathcal{D}_{\le\ell}^\top  = (D_{1}, D_{2},\dots,D_{\ell})$ as parameters, instead of $\nu$.  To do so, we introduce
\begin{equation}
\label{def:C_menor_o_igual_que_ell}
\begin{aligned}
\mathcal{C}_{\le \ell} & = \mathrm{diag}\,(L^{[1]}/D_{1},\dots,L^{[\ell]}/D_{\ell}), \\
\mathcal{C}_{> \ell} & = \mathrm{diag}\,(L^{[\ell+1]}/D_{\ell+1},\dots,L^{[2\ell]}/D_{2\ell}).
\end{aligned}
\end{equation}
\begin{lem}
\label{lem:from_masses_to_coefficients}
Let $L^{[\le \ell]} = (L^{[1]}, \dots,L^{[\ell]})^\top$. It satisfies
\[
L^{[\le \ell]} = \mathcal{C}_{\le \ell} \mathcal{D}_{\le\ell} =  \mathcal{C}_{\le \ell} ( \mathcal A \nu + \OO_2(\nu)),
\]
where $\mathcal A$ is a constant matrix, depending only on $\ell$, satisfying $\det \mathcal A \neq 0$.

Consequently, for any $j>\ell$, we can write $D_j$ as a function $D_j (\mathcal{D}_{\le\ell})$, independent of $G_0$. They satisfy $D_j(0) = 0$, $j \neq 0 \mod N$, $D_N(0) \neq 0$.
\end{lem}

\begin{proof}[Proof of Lemma~\ref{lem:from_masses_to_coefficients}]
We first observe that the relation among the mass parameters $\mu$ and $\nu$ can be written as
$\mu = \mathcal{I} \nu$, where
\[
\mathcal{I}^{\top} = (\Id_{\ell\times \ell},\widetilde{\Id}_{\ell\times \ell})
\]
and $\widetilde{\Id}_{\ell\times \ell}$ is the $\ell \times \ell$ exchange matrix, that is, with $1$'s at the antidiagonal and $0$'s otherwise.

Then, by Proposition~\ref{prop:expansion_Melnikov_positions},
\begin{equation}
\label{eq:D_en_funcio_de_nu}
\mathcal{D}_{\le\ell} = (\Id_{\ell\times \ell},0) \mathcal{D} = (\Id_{\ell\times \ell},0) \widetilde{\mathcal{D}} \hat{\overline{\M}} \mu = (\Id_{\ell\times \ell},0) \widetilde{\mathcal{D}} \hat{\overline{\M}} \mathcal{I} \nu ,
\end{equation}
where
\[
\widetilde{\mathcal{D}} = \mathrm{diag}\,(d_1,\ldots,d_{N-1})
\]
satisfies $\det \widetilde{\mathcal{D}}_{\mu = 0} \neq 0$.

\begin{lem}
    \label{lem:sobre_M_times_tilde_I} Let $A,B$ be the $\ell \times \ell$ matrices such that $(\hat{\overline{\M}} \mathcal{I})^\top = (A^\top,B^\top)$. Then, $\det A \neq 0$.
\end{lem}
\begin{proof}[Proof of Lemma~\ref{lem:sobre_M_times_tilde_I}] An immediate computation shows that
\[
(A^2)_{i,j} = 
\begin{cases}
    2 \ell-1, \qquad &\textrm{if $i=j$},\\
    -2, \qquad &\textrm{if $i\neq j $}
\end{cases}
\]
Then, the claim follows from Gershgorin circle theorem.
\end{proof}  
Hence, using Lemma~\ref{lem:sobre_M_times_tilde_I} and~\eqref{eq:D_en_funcio_de_nu}, we have that
\[
\mathcal{D}_{\le\ell} = \mathrm{diag}\,(d_1,\ldots,d_{\ell}) A \nu = \mathrm{diag}\,(d_1,\ldots,d_{\ell})_{\nu = 0} A \nu + \OO_2(\nu).
\]
Finally, the vanishing properties of the coefficients $D_j$ follow from the symmetry properties of the $N$-gon central configuration when $\nu = 0$.
\end{proof}

From now on, we will use $\mathcal{D}_{\le\ell}$ as the parameters in $L$, provided that $\mathcal{D}_{\le\ell}$ are small enough.

The manifolds have a tangency of order $\ell$ at  $\mathcal{D}_{\le\ell} = \mathcal{D}_{\le\ell}^0$, $\beta = \beta_0$ if
\begin{equation}
\label{eq:tangency}
    F(\mathcal{D}_{\le\ell}^0, \beta^0,G_0) = \begin{pmatrix}\partial_\beta L \\ \vdots \\ \partial^{\ell+1}_\beta L \end{pmatrix}_{\mid  \;\mathcal{D}_{\le\ell}  = \mathcal{D}_{\le\ell}^0, \beta  = \beta_0 } = 0
\end{equation}
and $\partial^{\ell+2}_\beta L_{\mid  \;\mathcal{D}_{\le\ell}  = \mathcal{D}_{\le\ell}^0,\beta = \beta_0} \neq 0$. It splits if $\mathrm{rank}\, \partial  F / \partial ( \mathcal{D}_{\le\ell},\beta)_{\mid \mathcal{D}_{\le\ell}^0,\beta_0} = \ell+1$.


We split $\Gamma^\top = (\Gamma_{\le \ell}^{\top},\Gamma_{>\ell}^\top)$, where $\Gamma_{\le \ell}^{\top} = (\cos \beta, \dots, \cos \ell \beta)$, and write
\[
\mathcal{C} = \begin{pmatrix} \mathcal{C}_{\le \ell} & 0 \\
0 & \mathcal{C}_{>\ell} 
\end{pmatrix}, 
\]
where $\mathcal{C}_{\le \ell}$ and $\mathcal{C}_{> \ell}$ were defined in~\eqref{def:C_menor_o_igual_que_ell}.

We introduce $W_{\le \ell}$ and $W_{> \ell}$ by
\begin{equation}
\label{def:WmenysWmes}
\left\{
\begin{aligned}
(W_{\le \ell})_{i,j} & = \partial^i_\beta (\cos j \beta),\\
(W_{> \ell})_{i,j} & = \partial^{i}_\beta (\cos (j+\ell) \beta), 
\end{aligned}
\right. \qquad 1  \le j \le \ell, \; 1\le i \le \ell+1.
\end{equation}
and $\tilde W_{\le \ell}$, the $\ell \times \ell$ matrix such that $(\tilde W_{\le \ell})_{i,j} = (W_{\le \ell})_{i,j}$, $1\le i,j\le \ell$.

\begin{lem}
    \label{lem:sobre_W_Gamma} $\det \tilde W_{\le \ell} = 0$ if and only if $\beta = k \pi$, $k\in \Z$.
\end{lem}

\begin{proof}[Proof of Lemma~\ref{lem:sobre_W_Gamma}] 
A direct computation shows that
$\det \tilde W _{\le \ell} = c_\ell (\sin \beta) ^{\ell(\ell+1)/2}$ with constant $c_\ell = (-1)^{\ell (\ell+1)/2} 2^{\ell (\ell-1)/2} \prod_{k=1}^{\ell} k! \neq 0$.
\end{proof}
Hence, we can rewrite~\eqref{eq:tangency} as
\begin{equation}
\label{eq:tangencies_2}
F = W_{\le \ell} \mathcal{C}_{\le \ell} \mathcal{D}_{\le \ell}+
W_{> \ell}  \mathcal{C}_{> \ell} \mathcal{D}_{> \ell} + \Delta_\beta \mathcal{R} = 0,
\end{equation}
where $\mathcal{D}^\top = (\mathcal{D}_{\le \ell}^\top, \mathcal{D}_{> \ell}^\top)$ and
\[
\Delta_\beta \mathcal{R} = (\partial_\beta \mathcal{R}, \dots, \partial_\beta^{\ell+1} \mathcal{R})^\top.
\]
We emphasize that, in~\eqref{eq:tangencies_2}, $\mathcal{D}_{> \ell}$ is a function of $\mathcal{D}_{\le \ell}$ and $\Delta_\beta \mathcal{R}$ is a function of $\mathcal{D}_{\le \ell}$ and $\beta$.

In order to solve~\eqref{eq:tangencies_2}, we first solve the auxiliary equation
\begin{equation}
    \label{eq:tangencies_simplificades}
    \tilde F = W_{\le \ell} \mathcal{E}_{\le \ell} + \Delta_\beta \mathcal{R}_{N\mid \mathcal{D}_{\le \ell} = 0} = 0,
\end{equation}
where $\mathcal{E}_{\le \ell}^\top = (E_1,\dots,E_{\ell})$ and $\mathcal{R}_N = \mathcal{R} - L_{>N}$, that is, 
\begin{equation}
\label{def:RN}
\mathcal{R}_N(D_{N},\beta,G_0) = L^{[N]}(D_N,G_0) \cos N \beta.
\end{equation}
We emphasize that equation~\eqref{eq:tangencies_simplificades} has unknowns $\mathcal{E}_{\le \ell}$ and $\beta$ and has a nonlinear dependence on $\beta$.

We recall that we are assuming that $N = 2\ell+1$ and $\ell = 2k-1$, $k\ge 1$.
\begin{lem}
\label{lem:tangencies_simplificades}
Equation~\eqref{eq:tangencies_simplificades} admits a solution $( \mathcal{E}_{\le \ell},\beta) = ( \mathcal{E}_{\le \ell}^0,\pi/2)$, where
\[
E_{2j}^0 = 0, \qquad 1 \le j \le k-1,
\]
and $\|\mathcal{E}_{\le \ell}^0\| \le K G_0^{N-3/2} e^{-N G_0^3/3}$, for some constant $K$ independent of $G_0$. Furthermore, 
\[
\partial \tilde F/\partial (\mathcal{E},\beta)_{\mid (\mathcal{E}_{\le \ell}^0,\pi/2)} 
= \begin{pmatrix}
\tilde W_{\le \ell \mid \beta = \pi/2} & 0\\
  b^\top & \sigma
\end{pmatrix},
\]
where $\sigma$ and $b$
satisfy, for some constant $\tilde K >0$ independent of $G_0$,
\[
|\sigma| \ge \tilde K G_0^{N-3/2} e^{-N G_0^3/3}, \qquad \|b\| \le \tilde K G_0^{N-3/2} e^{-N G_0^3/3}.
\]
Consequently,
\[
\left(\partial \tilde F/\partial (\mathcal{E},\beta)_{\mid (\mathcal{E}_{\le \ell}^0,\pi/2)} \right)^{-1}
= \begin{pmatrix}
\tilde W_{\le \ell \mid \beta = \pi/2}^{-1} & 0\\
  \tilde b^\top & \sigma^{-1}
\end{pmatrix},
\]
where
\[
\|\tilde b\| \le \tilde K.
\]
\end{lem}

\begin{proof}[Proof of Lemma~\ref{lem:tangencies_simplificades}]
Observe that, for any $i,j\in \N$, $i\ge 1$, $j\ge 0$,
\[
\partial_\beta^{2i}(\cos (2j+1) \beta)_{\mid \beta = \pi/2} = (-1)^i (2j+1)^{2i} \cos ((2j+1) \pi/2) =  0.
\]
Hence, taking $\beta = \pi/2$ and $E_{2j} = 0$, $j=1,\dots, k-1$, since $N$ is odd, the even components of~$\tilde F$ in~\eqref{eq:tangencies_simplificades} vanish. Denoting $\mathcal{E}_{\le \ell}^{\mathrm{odd}} = (E_{2j-1})_{1 \le j \le k}^\top$, the odd components of $\mathcal{E}_{\le \ell}$, the odd components of~$\tilde F$ are
\[
W_{\le \ell\mid \beta = \pi/2}^{\mathrm{odd}} \mathcal{E}_{\le \ell}^{\mathrm{odd}} + \Delta_\beta \mathcal{R}_{N\mid \beta = \pi/2}^{\mathrm{odd}} = 0,
\]
where
\[
(W_{\le \ell\mid \beta = \pi/2}^{\mathrm{odd}})_{i,j} = (-1)^i (2j-1)^{2i-1} \sin ((2j-1) \beta)_{\mid \beta = \pi/2} = (-1)^{i+j+1} (2j-1)^{2i-1}, \quad 1 \le i,j \le k.
\]
Hence, $W_{\le \ell\mid \beta = \pi/2}^{\mathrm{odd}}$ is invertible, with norm only depending on $N$. Moreover, by Proposition~\ref{prop:Fourier_expansion_Melnikov_RPCN2BP}, 
\[
\|\Delta_\beta \mathcal{R}_{N\mid \beta = \pi/2}^{\mathrm{odd}}\| \le K G_0^{N-3/2} e^{-N G_0^3/3},
\]
which proves  the existence of  the solution $(\mathcal{E}_{\le \ell}^0,\pi/2)$.

To compute $\partial \tilde F/\partial (\mathcal{E},\beta)_{\mid (\mathcal{E}_{\le \ell}^0,\pi/2)}$ we first remark that, by construction, $\tilde F_{j+1} = \partial_\beta \tilde F_j$, $j= 1, \dots, \ell$. Hence,
$\partial_\beta \tilde F_{j\mid ( \mathcal{E}_{\le \ell}^0,\pi/2)} = 0$, for $j= 1, \dots, \ell$. 
Also, in view of~\eqref{eq:tangencies_simplificades},  $\partial (\tilde F_1, \dots, \tilde F_{\ell})^\top/\partial \mathcal{E}_{\le \ell} = \tilde W_{\le \ell}$. The claim on the coefficient $\sigma$
follows from the bounds on $\mathcal{E}_{\le \ell}^0$ and the fact that the matrix
$
(W_{\ell},w),
$
where $w_j = \partial^j_\beta (\cos (N\beta))$, $j=1,\dots, \ell+1$, is non-singular at $\beta = \pi/2$.
\end{proof}

The following lemma is an immediate consequence of Proposition~\ref{prop:Fourier_expansion_Melnikov_RPCN2BP}.

\begin{lem} 
\label{lem:fita_de_C}There exists $K>0$ such that the matrices $\mathcal{C}_{\le \ell}$ and $\mathcal{C}_{>\ell}$ satisfy
\[
\|\mathcal{C}_{\le \ell}^{-1}\| \le K G_0^{3/2-\ell} e^{\ell G_0^3/3},\qquad
\|\mathcal{C}_{> \ell}\| \le K G_0^{\ell-3/2} e^{-(\ell+1) G_0^3/3}.
\]
\end{lem}

Finally, we solve~\eqref{eq:tangencies_2}.

\begin{prop}
\label{prop:tangencies}
    Let $G_0\gg 1$. Then, equation~\eqref{eq:tangencies_2} has a solution 
    \[
    (\mathcal{D}_{\le \ell}^0,\beta^0) = (\mathcal{C}_{\le \ell}^{-1} (\mathcal{E}_{\le \ell}^0+\tilde{\mathcal{E}}_{\le \ell}^0),\pi/2+\sigma^{-1}\alpha_0)
    \]
with $\| \tilde{\mathcal{E}}_{\le \ell}^0 \|, |\alpha_0| \le K G_0^{2\ell+1/2} e^{-(2\ell+2) G_0^3/3}$.
Moreover, $\det DF_{\mid (\mathcal{D}_{\le \ell}^0,\beta^0)} \neq 0$.
\end{prop}

\begin{proof}[Proof of Proposition~\ref{prop:tangencies}]

To do so, we introduce 
$\tilde{\mathcal{E}}_{\le \ell}$ 
and $\alpha$ by setting 
$\mathcal{D}_{\le \ell} = \mathcal{C}_{\le \ell}^{-1} (\mathcal{E}_{\le \ell}^0 + \tilde{\mathcal{E}}_{\le \ell})$ and $\beta = \pi/2+\sigma^{-1} \alpha$, where  $\mathcal{E}_{\le \ell}^0$ and $\sigma$ are given in Lemma~\ref{lem:tangencies_simplificades}.
We write the function $F$ in~\eqref{eq:tangencies_2} as
\begin{equation}
\label{def:F1F2F3F4}
  F (\mathcal{C}_{\le \ell}^{-1} (\mathcal{E}_{\le \ell}^0 + \tilde{\mathcal{E}}_{\le \ell}), \pi/2+\sigma^{-1} \alpha) = \sum_{r=1}^4 F_r(\tilde{\mathcal{E}}_{\le \ell},\alpha) 
\end{equation}
where
\begin{align}
    \label{def:F1}
    F_1(\tilde{\mathcal{E}}_{\le \ell},\alpha) & = \tilde F( \mathcal{E}_{\le \ell}^0 +  \tilde{\mathcal{E}}_{\le \ell}, \pi/2+\sigma^{-1}\alpha) \\
    \label{def:F2}
    F_2(\tilde{\mathcal{E}}_{\le \ell},\alpha) & = \Delta_\beta \mathcal{R}_N(\mathcal{C}_{\le \ell}^{-1} (\mathcal{E}_{\le \ell}^0 + \tilde{\mathcal{E}}_{\le \ell}), \pi/2+\sigma^{-1}\alpha) - \Delta_\beta \mathcal{R}_N(0, \pi/2+\sigma^{-1}\alpha)\\
    F_3(\tilde{\mathcal{E}}_{\le \ell},\alpha) & = W_{>\ell}(\pi/2+\sigma^{-1}\alpha) \mathcal{C}_{> \ell} \mathcal{D}_{>\ell} \left(\mathcal{C}_{\le \ell}^{-1} (\mathcal{E}_{\le \ell}^0 + \tilde{\mathcal{E}}_{\le \ell}) \right)\\
    F_4(\tilde{\mathcal{E}}_{\le \ell},\alpha) & = L_{>N} (\mathcal{C}_{\le \ell}^{-1} (\mathcal{E}_{\le \ell}^0 + \tilde{\mathcal{E}}_{\le \ell}), \pi/2+\sigma^{-1}\alpha)
\end{align}
where we recall that $\tilde F$ was defined in~\eqref{eq:tangencies_simplificades}, $\mathcal{R}_N$ in~\eqref{def:RN}, $W_{>  \ell}$ in~\eqref{def:WmenysWmes}, and $L_{> N}$ in~\eqref{def:LmesgranqueN}.
The following lemma gathers the bounds of $F_j$, $j=1,\dots,4$, in a suitable ball.

\begin{lem}
    \label{lem:propietats_de_F1F2F3F4} 
    \label{lem:fites_de_F_j_tangencies_melnikov} 
    There exists $C,K>0$, independent of $G_0$, such that, for all $(\tilde{\mathcal{E}}_{\le \ell},\alpha)$ satisfying $\|\tilde{\mathcal{E}}_{\le \ell}\|,\|\alpha\| \le C G_0^{2\ell+1/2} e^{-(2\ell+2) G_0^3/3}$, the following holds.
\begin{enumerate}
    \item 
    \[
    F_1(\tilde{\mathcal{E}}_{\le \ell},\alpha) = (\partial \tilde F/\partial (\mathcal{E}_{\ge \ell},\beta))_{\mid (\mathcal{E}_{\ge \ell}^0,\pi/2)}
    \begin{pmatrix}
         \tilde{\mathcal{E}}_{\le \ell} \\ \sigma^{-1}\alpha
    \end{pmatrix} + \OO_2( \tilde{\mathcal{E}}_{\le \ell},\sigma^{-1}\alpha),
    \]
    \item 
    \[
    \begin{aligned}
    \|F_2(0,0)\| & \le K G_0^{3\ell+1/2} e^{-(3\ell+2) G_0^3/3}, \\
    \left\|\partial_{\tilde{\mathcal{E}}_{\le \ell}} F_2(\tilde{\mathcal{E}}_{\le \ell},\alpha)\right\|, \left\|\partial_{\alpha} F_2(\tilde{\mathcal{E}}_{\le \ell},\alpha)\right\|& \le K G_0^{\ell+1} e^{-(\ell+1) G_0^3/3},
    \end{aligned}
    \]
    \item
    \[
    \begin{aligned}
    \|F_3(0,0)\| & \le K G_0^{2\ell+1/2} e^{-(2\ell+2) G_0^3/3}, \\
    \left\|\partial_{\tilde{\mathcal{E}}_{\le \ell}} F_3(\tilde{\mathcal{E}}_{\le \ell},\alpha)\right\|, \left\|\partial_{\alpha} F_3(\tilde{\mathcal{E}}_{\le \ell},\alpha)\right\|& \le K G_0 e^{-G_0^3/3},
    \end{aligned}
    \]
    \item 
    \[
    \begin{aligned}
    \|F_4(0,0)\| & \le K G_0^{2\ell+1/2} e^{-(2\ell+2) G_0^3/3}, \\
    \left\|\partial_{\tilde{\mathcal{E}}_{\le \ell}} F_4(\tilde{\mathcal{E}}_{\le \ell},\alpha)\right\|, \left\|\partial_{\alpha} F_4(\tilde{\mathcal{E}}_{\le \ell},\alpha)\right\|& \le K G_0^{\ell + 2} e^{-(\ell+2) G_0^3/3}.
    \end{aligned}
    \]
\end{enumerate}
\end{lem}

\begin{proof}[Proof of Lemma~\ref{lem:propietats_de_F1F2F3F4}]
We first remark that, from Lemmas~\ref{lem:tangencies_simplificades}  and~\ref{lem:fita_de_C},
\begin{equation}
    \label{C_menor_o_igual_que_ell_meny_u_E_super_0}
    \|\mathcal{C}_{\le \ell}^{-1} \mathcal{E}_{\le \ell}^0\| \le K G_0^{\ell+1} e^{-(\ell+1) G_0^3/3}.
\end{equation}
We will use this bound several times along this proof.

The proof of 1 follows from 
\[
F_1(\tilde{\mathcal{E}}_{\le \ell},\alpha) = \tilde F( \mathcal{E}_{\le \ell}^0 + \tilde{\mathcal{E}}_{\le \ell}, \pi/2+ \sigma^{-1}\alpha) - \tilde F( \mathcal{E}_{\le \ell}^0 , \pi/2),
\]
Lemma~\ref{lem:tangencies_simplificades} the definition of the functions $\tilde F$, Lemmas~\ref{lem:fita_de_C},  and~\ref{lem:tangencies_simplificades}.

Now we check 2. We first observe that, by Propositions~\ref{prop:Fourier_expansion_Melnikov_RPCN2BP} and~\ref{prop:expansion_Melnikov_positions}, taking into account the definition or $\mathcal{R}_N$ in~\eqref{def:RN}, there exists $\rho >0$ such that
\begin{equation}
\label{fita_derivades_KN}
    \sup_{\|(D_N,\beta)\|< \rho} \| \partial_{D_N} \Delta_\beta \mathcal{R}_N \|, \sup_{\|(D_N,\beta)\|< \rho} \| \partial_{\beta} \Delta_\beta \mathcal{R}_N \| \le K G_0^{2\ell-1/2} e^{-(2\ell+1) G_0^3/3}.
\end{equation}
Hence, combining~\eqref{C_menor_o_igual_que_ell_meny_u_E_super_0} and~\eqref{fita_derivades_KN}, we get
\[
\|F_2(0,0)\| = \|\Delta_\beta \mathcal{R}_N(\mathcal{C}_{\le \ell}^{-1} \mathcal{E}_{\le \ell}^0 , \pi/2) - \Delta_\beta \mathcal{R}_N(0, \pi/2)\| \le K G_0^{3\ell+1/2} e^{-(3\ell+2)G_0^3/3}.
\]
Moreover, since
\[
\partial_{\tilde{\mathcal{E}}_{\le \ell}} F_2 (\tilde{\mathcal{E}}_{\le \ell},\alpha) =
\partial_{D_N} L^{[N]}_{\mid D_n(\mathcal{C}_{\le \ell}^{-1} (\mathcal{E}_{\le \ell}^0 + \tilde{\mathcal{E}}_{\le \ell}))} \cos N(\pi/2+\sigma^{-1} \alpha) \partial_{\mathcal{D}_{\le \ell}} {D_N}_{\mid \mathcal{C}_{\le \ell}^{-1} (\mathcal{E}_{\le \ell}^0 + \tilde{\mathcal{E}}_{\le \ell})} \mathcal{C}_{\le \ell}^{-1},
\]
and, by~\eqref{eq:D_en_funcio_de_nu}, the derivative $\partial_{\mathcal{D}_{\le \ell}} {D_N}$ only depends on $\ell$, the second part of the claim follows immediately from Propositions~\ref{prop:Fourier_expansion_Melnikov_RPCN2BP} and~\ref{prop:expansion_Melnikov_positions}.

The last part of 2 follows immediately from~\eqref{fita_derivades_KN} and the bound of $\sigma^{-1}$ provided by Lemma~\ref{lem:tangencies_simplificades}.

Now we check 3. Recall that, by~\eqref{eq:D_en_funcio_de_nu}, $\mathcal{D}_{\le \ell} = 0$ if and only if $\mu=0$, for $\|\mu\| \le \mu^*$. Then, since, by the symmetry of the regular $(2\ell+1)$-gon central configuration, ${\mathcal{D}_{>\ell}}_{\mid\mu = 0} = 0$, we have that
\[
\begin{aligned}
\|F_3(0,0)\| & = \|W_{>\ell}(\pi/2) \mathcal{C}_{> \ell} \mathcal{D}_{>\ell} \left(\mathcal{C}_{\le \ell}^{-1} \mathcal{E}_{\le \ell}^0  \right) - W_{>\ell}(\pi/2) \mathcal{C}_{> \ell} \mathcal{D}_{>\ell} \left(0\right)\| \\
& \le K \|\mathcal{C}_{> \ell}\| \|\mathcal{C}_{\le \ell}^{-1} \mathcal{E}_{\le \ell}^0 \|.
\end{aligned}
\]
Then, the claim on $F_3(0,0)$ follows from Lemma~\ref{lem:fita_de_C} and~\eqref{C_menor_o_igual_que_ell_meny_u_E_super_0}.

For the derivative with respect to $\tilde{\mathcal{E}}_{\le \ell}$, we have that, by Lemma~\ref{lem:fita_de_C} and using that $\| W_{>\ell}\|$ and $\| \partial_{\mathcal{D}_{\le \ell}} \mathcal{D}_{>\ell} \|$ are bounded independently of $G_0$ in the domain we are considering,
\[
\left\|\partial_{\tilde{\mathcal{E}}_{\le \ell}} F_3(\tilde{\mathcal{E}}_{\le \ell},\alpha)\right\|
\le \| W_{>\ell}\| \| \mathcal{C}_{> \ell}\| \| \partial_{\mathcal{D}_{\le \ell}} \mathcal{D}_{>\ell} \|\|\mathcal{C}_{\le \ell}^{-1}\| \le K G_0 e^{-G_0^3/3}.
\]
For the derivative with respect to $\alpha$, we use again that ${\mathcal{D}_{>\ell}}_{\mid \mu = 0} = 0$, inequality~\eqref{C_menor_o_igual_que_ell_meny_u_E_super_0} and Lemmas~\ref{lem:tangencies_simplificades} and~\ref{lem:fita_de_C} to obtain
\[
\begin{aligned}
\left\|\partial_{\alpha} F_3(\tilde{\mathcal{E}}_{\le \ell},\alpha)\right\| & = \|\partial_{\beta} W_{>\ell}(\pi/2+ \sigma^{-1}\alpha) \mathcal{C}_{> \ell} \left(\mathcal{D}_{>\ell} \left(\mathcal{C}_{\le \ell}^{-1} (\mathcal{E}_{\le \ell}^0 + \tilde{\mathcal{E}}_{\le \ell})   \right) - \mathcal{D}_{>\ell} \left(0\right)\right)\sigma^{-1}\| \\
& \le K \sigma^{-1} \|\mathcal{C}_{> \ell}\| \sup_{\|\tilde{\mathcal{E}}_{\le \ell}\| \le C K G_0^{2\ell+1/2} e^{-(2\ell+2) G_0^3/3}} \|\mathcal{C}_{\le \ell}^{-1} (\mathcal{E}_{\le \ell}^0 + \tilde{\mathcal{E}}_{\le \ell})  \| \\
& \le K G_0 e^{-G_0^3/3}.
\end{aligned}
\]

Finally, we check 4. The claim on $F_4(0,0)$ follows immediately from  Propositions~\ref{prop:Fourier_expansion_Melnikov_RPCN2BP} and~\ref{prop:expansion_Melnikov_positions}.

As for the derivative with respect $\tilde{\mathcal{E}}_{\le \ell}$, using that we have written $\mu$ as a function of $\mathcal{D}_{\le \ell}$, Propositions~\ref{prop:Fourier_expansion_Melnikov_RPCN2BP} and~\ref{prop:expansion_Melnikov_positions} and Lemma~\ref{lem:fita_de_C}, we have that
\[
\begin{aligned}
\|\partial_{\tilde{\mathcal{E}}_{\le \ell}} F_4(\tilde{\mathcal{E}}_{\le \ell},\alpha)\| & = 
\| \partial_{\mu} {L_{>N}}_{\mid \mathcal{C}_{\le \ell}^{-1} (\mathcal{E}_{\le \ell}^0 + \tilde{\mathcal{E}}_{\le \ell})),\pi/2+\sigma^{-1}\alpha} \partial_{\mathcal{D}_{\le \ell}} \mu_{\mid \mathcal{C}_{\le \ell}^{-1} (\mathcal{E}_{\le \ell}^0 + \tilde{\mathcal{E}}_{\le \ell}))} \mathcal{C}_{\le \ell}^{-1}\| \\
& \le K G_0^{\ell + 2} e^{-(\ell+2) G_0^3/3}.
\end{aligned}
\]
To deal with the derivative with respect to $\alpha$, we first remark that, by the symmetry of the regular $(2\ell+1)$-gon central configuration and Proposition~\ref{prop:Fourier_expansion_Melnikov_RPCN2BP}, 
\[
\|\partial_\beta L_{>N} (0, \pi/2+\sigma^{-1}\alpha)\| \le K G_0^{2(2\ell+1)-3/2} e^{-2(2\ell+1) G_0^3/3}.
\]
Then, using the bound of $\sigma^{-1}$ in Lemma~\ref{lem:tangencies_simplificades}, Proposition~\ref{prop:Fourier_expansion_Melnikov_RPCN2BP} and~\eqref{C_menor_o_igual_que_ell_meny_u_E_super_0}
\[
\begin{aligned}
    \|\partial_{\alpha} F_4(\tilde{\mathcal{E}}_{\le \ell},\alpha)\| \le &  \|\partial_\beta L_{>N} (\mathcal{C}_{\le \ell}^{-1} (\mathcal{E}_{\le \ell}^0 + \tilde{\mathcal{E}}_{\le \ell}), \pi/2+\sigma^{-1}\alpha)-\partial_\beta L_{>N} (0, \pi/2+\sigma^{-1}\alpha)\| |\sigma^{-1}| \\
    & + \|\partial_\beta L_{>N} (0, \pi/2+\sigma^{-1}\alpha)\| |\sigma^{-1}| \\
    \le & K G_0^{\ell+2} e^{-(\ell+2) G_0^3/3} + K  G_0^{2\ell+1} e^{-(2\ell+1) G_0^3/3}.
\end{aligned}
\]
\end{proof}
Using~\eqref{def:F1F2F3F4}, introducing the matrices
\begin{equation}
    \label{def:A1}
    A_1 = (\partial \tilde F/\partial (\mathcal{E}_{\ge \ell},\beta))_{\mid (\mathcal{E}_{\ge \ell}^0,\pi/2)} I_{\sigma^{-1}}, \qquad I_a = \mathrm{diag}\,(1,\dots,1,a),
\end{equation}
and denoting $\xi = (\tilde{\mathcal{E}}_{\le \ell},\alpha)^\top$, we rewrite equation~\eqref{eq:tangencies_2} as the fixed point equation
\begin{equation}
    \label{eq:punt_fix_tangencies_pel_Melnikov}
    \xi =  \mathcal{F} (\xi)
\end{equation}
where, 
\begin{equation}
    \label{def:F_punt_fix_tangencies_pel_Melnikov}
    \mathcal{F} (\xi) = - A_1^{-1} (F(\xi)-A_1 \xi).
\end{equation}
We simply check that $\mathcal{F}$ is a contraction on $\{\xi = (\tilde{\mathcal{E}}_{\le \ell},\alpha)^\top \mid \,\|\xi \| \le C G_0^{2\ell+1/2} e^{-(2\ell+2) G_0^3/3}\}$, where $\|\xi\| = \|\tilde{\mathcal{E}}_{\le \ell}\|+\|\alpha\|$.

Indeed, we first observe that, by Lemma~\ref{lem:fites_de_F_j_tangencies_melnikov},
\[
\|\mathcal{F}(0) \| \le \sum_{j=2}^4 \|F_j(0)\| \le K  G_0^{2\ell+1/2} e^{-(2\ell+2) G_0^3/3}.
\]

Next, since, by Lemma~\ref{lem:tangencies_simplificades},
\[
A_1^{-1} = \begin{pmatrix}
\tilde W_{\le \ell \mid \beta = \pi/2}^{-1} & 0\\
  \tilde b^\top & 1
\end{pmatrix},
\]
where $\|\tilde b\|$ is bounded independently of $G_0$, we have that $\|A_1^{-1}\|$ is bounded independently of $G_0$. 

Now we bound the Lipschitz constant of $F(\xi)-A_1 \xi$ for $\|\xi \| \le C G_0^{2\ell+1/2} e^{-(2\ell+2) G_0^3/3}$.
Observe that, from~\eqref{def:F_punt_fix_tangencies_pel_Melnikov},
\[
F(\xi)-A_1 \xi = F_1(\xi) - A_1 \xi + \sum_{j=2}^4 F_j(\xi).
\]
We bound the Lipschitz constant of each of the terms above.

By 1 of Lemma~\ref{lem:fites_de_F_j_tangencies_melnikov} and the bound of $\sigma^{-1}$ provided by Lemma~\ref{lem:tangencies_simplificades},
\[
\sup_{\|\xi\| \le C G_0^{2\ell+1/2} e^{-(2\ell+2) G_0^3/3}} \|D F_1(\xi) - A_1 \| \le K G_0^{2\ell+1/2} e^{-(2\ell+2) G_0^3/3} |\sigma^{-1}| \le K G_0 e^{-G_0^3/3}.
\]
Also, from 2, 3, and 4 of Lemma~\ref{lem:fites_de_F_j_tangencies_melnikov}, for $\|\xi \| \le C G_0^{2\ell+1/2} e^{-(2\ell+2) G_0^3/3}$,
\[
\|D F_j(\xi)\| \le K \max\{K G_0 e^{-G_0^3/3}, G_0^{\ell+1} e^{-(\ell+1) G_0^3/3}, G_0^{\ell + 2} e^{-(\ell+2) G_0^3/3},\}.
\]
Taking $G_0$ large enough, the first claim is proven. The second one is immediate.
\end{proof}
\end{proof}

\section{Renormalization of a high-order tangency in the parabolic setting}\label{uni_sec:renormalization}
This section follows the lines of Section~6 in \cite{GMP2025}.
\subsection{Preliminaries}
Let $\ord \in \mathbb{N}$, $\ord \ge 3$ and suppose that we have $\ord$ free parameters 
\[
\varepsilon = (\varepsilon_0, \varepsilon_1, \ldots, \varepsilon_{\ord-1}) \in \mathbb{R}^\ord.
\]
In our case $\varepsilon$ consists of some of the mass parameters $\mu_j$.

Suppose, as proved in our case in Section~\ref{sec:tangencies}, that the one-dimensional stable and unstable invariant manifolds of infinity $\gamma^{s,u}$ of the Poincaré map $\mathcal{P}_{\varepsilon}$ in \eqref{def:mapa_de_Poincare} have a tangency of order $\ord$ at $\varepsilon=0$ which unfolds generically in $\varepsilon$. 

As in \cite{GMP2025}, after applying McGehee compactification, we obtain a new map $f_\varepsilon=f_\varepsilon(q,p)$ for which the origin $O=(0,0)$ is a fixed point with one-dimensional stable and unstable manifolds $W^{s,u}(O)$ displaying a tangency of order $\ord$ which unfolds generically. Moreover, since $D f_\varepsilon (O) = \mathrm{Id}$, the origin $O$ is a degenerate (parabolic) saddle. 

In spite of this degeneracy, the fact that the potential $V=V(r,\phi)$ defined in \eqref{def:V_N2} is even in $\phi$ allows us to apply the normal form in \cite{GMP2025} and therefore obtain asymptotics for the Shilnikov maps; these are recalled in Appendix~\ref{uni_sec:Shilnikov}. Denote by
$\mathcal{K}_{\varepsilon}$ and $\Psi_{\mathcal{K}_{\varepsilon}}$, respectively, the Hamiltonian and the Poincaré return map in the new coordinates after applying the normal form; in the notation we have made explicit their dependence on the parameters $\varepsilon$.

Moreover, by virtue again of the aforementioned normal form, we may assume that the invariant manifolds are straightened and $W_{loc}^s(O)=\{p=0\}$, $W_{loc}^u(O)=\{q=0\}$.

In summary, we have the existence of two points, $(a, 0) \in W_{loc}^s(O)$ and $(0, \tilde{a}) \in W_{loc}^u(O)$, satisfying:
\begin{itemize}
    \item There exists a (unique) number $L \in \mathbb{N}$ for which $\Psi_{\mathcal{K}_{\varepsilon}}^L(0,\tilde{a})=(a,0)$.
    \item For any neighborhood $R$ of $(a,0)$, there exists a small connected arc $\gamma_R \subset W_{loc}^u(O)$ which passes through $(0,\tilde{a})$, satisfies that $\Psi_{\mathcal{K}_{\varepsilon}}^L(\gamma_R)\subset R$ and moreover has a tangency of order $\ord$ with $W_{loc}^s(O)$ at $(a,0)$ when $\varepsilon=0$.
    \item The tangency of order $\ord$ between $\Psi_{\mathcal{K}_{\varepsilon}}^L(\gamma_R)$ and $W_{loc}^s(O)$ when $\varepsilon=0$ unfolds generically in $\varepsilon$.
\end{itemize}

\subsection{Global map}

In the next lemma, we describe the local structure of the map
\begin{equation}\label{uni_eq:definitionglobalmap}
\Psi_{glob, \varepsilon}:=\Psi_{\mathcal{K}_{\varepsilon}}^L
\end{equation}
from a neighborhood of $(0,\tilde{a})$ to a neighborhood of $(a,0)$. More concretely we define the complex balls
\begin{equation}\label{uni_eq:ballsglobalmap}
\begin{split}
R^u_{\rho}=&\{z=(q,p)\in \mathbb{C}^2\colon |z-(0,\tilde a)|\leq \rho\},\\
R^s_{\rho}=&\{z=(q,p)\in \mathbb{C}^2\colon |z-(a,0)|\leq \rho\}.
\end{split}
\end{equation}
\begin{lem}\label{uni_lem:globalmap}
        Let $(a,0)\in W^{s}_{loc}$ and $(0,\tilde a)\in W^{u}_{loc}$ be the homoclinic points constructed above and let $\Psi_{glob,\varepsilon}$ be the map in \eqref{uni_eq:definitionglobalmap}. Then, for $\rho>0$ small enough,  the map $\Psi_{glob,\varepsilon}:R^{u}_{\rho}\to R^s_{2\rho}$ is well defined, real-analytic and the dependence on $\varepsilon$ is real-analytic. Moreover, uniformly for $(x,\tilde a+y)\in R^{u}_{\rho}$, and $\varepsilon$ in a sufficiently small complex disk centered at zero
        \begin{multline*}
    \Psi_{glob,\varepsilon}:(x,\tilde a+y)\mapsto (a+b x+c y+X(x,y;\varepsilon), \\ \varepsilon_0 + d x+\varepsilon_1 y +\cdots+\varepsilon_{\ord-1} y^{\ord-1} -y^{\ord+1}+Y(x,y;\varepsilon)),  
        \end{multline*}
        for some constants $b,c$ and $d>0$ and 
        \[
        X(x,y;\varepsilon)=\mathcal O_2(\varepsilon)+\mathcal{O}_2(x,y),\quad\quad Y(x,y;\varepsilon)=\mathcal O_2(\varepsilon)+\mathcal{O}(x(x+y))+\mathcal{O}_{\ord+2}(x,y).
        \]
    \end{lem}

We provide estimates in complex domains in order to make use of Cauchy estimates to obtain bounds for certain derivatives in the next section. 

\subsection{Renormalization}\label{uni_subsec:renormalization}
In this section we prove Proposition~\ref{uni_prop:mainresultrenormalization}, which ensures that a suitably renormalized iteration of the return map to a neighborhood of a homoclinic tangency of order $\ord$ converges, in the real-analytic topology, to the H\'enon map of degree $\ord+1$. The proof is divided in several steps, in subsequent lemmas.

\begin{prop}\label{uni_prop:mainresultrenormalization}
    Let $f_{\varepsilon}$ denote the Poincar\'{e} map and assume that, for $\varepsilon=0$, the invariant manifolds of the degenerate saddle $O$, $W^s(O)$ and $W^u(O)$, possess a homoclinic tangency of order $\ord$ which unfolds generically in $\varepsilon$. Denote by $z$ the corresponding  homoclinic tangency for the map $f_{\varepsilon}$. Then,  there exists $L\in\mathbb{N}$ such that, for any $T\in\mathbb N$ sufficiently large, there exist
    \begin{itemize}
        \item an affine reparametrization $\varphi_{T}:[-2,2]^\ord \to\mathbb{R}^\ord$, which maps $\varphi_{T}:\kappa\mapsto \varepsilon=\varphi_{T}(\kappa)$, for $\kappa=(\kappa_0, \kappa_1, \ldots, \kappa_{\ord-1})$ and
        \item a real-analytic change of variables $\Psi_{\kappa,T}: B_{10}\to \mathbb{C}^2$
    \end{itemize}
    such that, as $T\to\infty$,
    \[
    \varphi_{T}([-2,2]^\ord)\to 0,\qquad\qquad \Psi_{\kappa,T}(B_{10})\to z,
    \] 
    and,  uniformly for $(Q,P)\in B_{10}$, \begin{equation}\label{uni_eq:convergenceHenon}
    \Psi_{\kappa,T}^{-1}\circ f_{\varphi_{T}(\kappa)}^{L+T}\circ \Psi_{\kappa,T}(Q,P)=(P,\kappa_0-Q+\kappa_1 P + \cdots + \kappa_{\ord-1} P^{\ord-1} -P^{\ord+1})+\mathcal{O}(T^{-\beta}),
    \end{equation}
    for some positive real number $\beta>0$.
\end{prop}

\begin{rem}
    We state Proposition~\ref{uni_prop:mainresultrenormalization} for $f_\varepsilon$ instead that for $\mathcal{P}_\varepsilon$ because for the former we can properly make sense of the degenerate saddle $O$ as a fixed point; of course an analogous proposition holds for the map $\mathcal{P}_\varepsilon$ and the invariant manifolds of infinity.
\end{rem}

Let $\Psi^{u,s}_T:\Sigma_\delta\to\mathbb{C}^2$ be the unstable and stable Shilnikov maps constructed in Proposition~\ref{uni_prop:Shilnikov}.

Then, for $T$ large enough,  the map $\Psi_\varepsilon:\mathcal B_\delta\to\mathbb{C}^2$ given by 
\begin{equation}\label{uni_eq:Psirenorm}
\Psi_\varepsilon(\xi,\eta) \coloneqq (\Psi_T^s)^{-1}\circ \Psi_{glob,\varepsilon}\circ\Psi_T^u(a+\xi,\tilde a +\eta)
\end{equation}
is real-analytic on 
\[
\mathcal B_{\delta}=\{z\in\mathbb C^2\colon \mathrm{dist}(z,[-\sigma,\sigma]^2)\leq \delta\},
\]
for some constant $\sigma$ depending only on the quantity $a$ in the definition of $\Sigma$ in \eqref{uni_eq:dfnsigmadelta}. We recall that these are small but fixed constants.
In Lemma~\ref{uni_lem:renormfirststep} we claim that, for any $T\in\mathbb N$ large enough, there exists a certain value $\varepsilon_T$ for which there exists a fixed point for the map $\Psi_{\varepsilon_T}$ and, moreover, the image under the map $\Psi_{\varepsilon_T}$ of a vertical segment passing through this fixed point  is an arc which has a tangency of order $\ord$ with a horizontal segment. Then, we will see in Lemma~\ref{uni_lem:rescalingtoHenon} that, locally around this point, a rescaling of the map $\Psi_{\varepsilon_T}$ is close to the (critical) H\'{e}non map of degree $\ord+1$. The following notation will be used. We write  
\[
\tilde x_T(\xi,\eta)=x_T( a+\xi,\tilde a+\eta)\qquad\qquad\text{and}\qquad\qquad\tilde y_T(\xi,\eta)=y_T( a+\xi, \tilde a+\eta),
\]
where $x_T, y_T$ where given in Proposition~\ref{uni_prop:Shilnikov}.

The following lemma was proved in \cite{GMP2025} for the case of a quadratic tangency (i.e., a tangency of order $1$). The proof can be readily adapted for tangencies of any order.

\begin{lem}\label{uni_lem:renormfirststep}
    For any $T\in\mathbb N$ large enough, there exist real  $\xi_\ast,\eta_\ast,\varepsilon_\ast=\mathcal{O}(T^{-2/3})$, depending on $T$, such that $\Psi_{\varepsilon_\ast}(\xi_\ast,\eta_\ast)=(\xi_\ast,\eta_\ast)$ and the $\eta$-coordinate of the derivative with respect to $\eta$ of $\Psi_{\mathcal{K}_{\varepsilon}}$ evaluated at $(\xi_\ast, \eta_\ast; \varepsilon_\ast)$ equals zero.
\end{lem}

\begin{rem}
    In fact, in what follows, we only use that $\Psi_{\varepsilon_\ast}(\xi_\ast,\eta_\ast)=(\xi_\ast,\eta_\ast)$. The condition on the derivative of $\Psi_\varepsilon$ will not be used in our case, at the expense of assuming that the order of the tangency is high enough; more concretely, it will be enough to assume $\ord \ge 3$.
\end{rem}

\begin{figure}[h]
\begin{center}
\includegraphics[width=.6\textwidth]{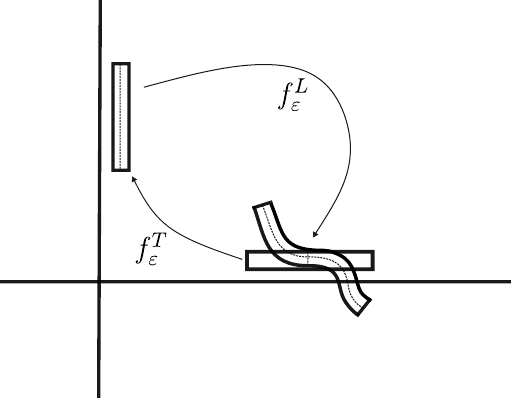}
\end{center}
\caption{Local and global maps in the case of a tangency of order $\ord=4$.}\label{fig:renormalizacion_quintic}
\end{figure}

We now define a suitable rescaling of the map $\Psi_\varepsilon$ in \eqref{uni_eq:Psirenorm} around the point $(\xi_*,\eta_*)$ for values of $\varepsilon$ close to $\varepsilon_*$. More precisely, we prove the following result.

\begin{lem}\label{uni_lem:rescalingtoHenon}
Let $\ord \in \mathbb{N}$, $\ord \ge 3$, be the order of the tangency. 

Denote by $B_{10}\subset \mathbb{C}^2$ the complex ball around the origin of radius $10$. 

Let $h_\alpha:(Q,P)\mapsto (\xi,\eta)$ be the change of variables given by
\[
h_\alpha (Q,P)=(\xi_\ast + \alpha Q, \eta_\ast + \alpha P)
\]
where $\alpha^\ord=\partial_\eta \tilde y_T(\xi_*,\eta_*)$.

Moreover, let $\varepsilon(\kappa) \coloneqq \varepsilon_* + \upsilon(\kappa)$, where $\upsilon$ is a certain function with components $\upsilon=(\upsilon_0, \upsilon_1, \ldots, \upsilon_{\ord-1})$ satisfying $\upsilon_j(\kappa)= \mathcal{O}(T^{-2/3}) + \alpha^{\ord+1-j} \kappa_j$ for $j=0,1,\ldots,\ord-1$.

Then, uniformly for $(Q,P)\in B_{10}\subset \mathbb{C}^2$ and $\kappa\in[-2,2]^{\ord}$ we have
\[
(h_\alpha)^{-1}\circ \Psi_{\varepsilon(\kappa)}\circ h_\alpha (Q,P)= H_\kappa(Q,P) +\mathcal{O}(T^{-\beta}),
\]
where $\beta > 0$, $H_\kappa(Q,P)= (c P, \kappa_0 + d \gamma Q + \kappa_1 P + \cdots + \kappa_{\ord-1} P^{\ord-1} - P^{\ord+1})$ and $\gamma=\displaystyle \frac{\partial_\xi \tilde x_T(\xi_*,\eta_*)}{\partial_\eta \tilde y_T(\xi_*,\eta_*)}$.
\end{lem}

\begin{proof}
We want to show that there exists a function $W=W(Q,P)$ with components $W=(U, V)$ satisfying, for some $\beta > 0$, that $W=\mathcal{O}\left(T^{-\beta} \right)$ uniformly for $(Q,P)\in B_{10}$ and such that, on the same domain,
\begin{equation}\label{uni_eq:renormcondproof}
   \Psi_{glob,\varepsilon}\circ \Psi^u_T(a+\xi_\ast+\alpha Q, \tilde{a}+\eta_\ast+\alpha P)= \Psi _T^s\circ h_\alpha \circ (H_\kappa + W)(Q,P).
\end{equation}

Throughout the proof it is important to have in mind that
\[
\begin{aligned}
    &\tilde{x}_T= \mathcal{O}(T^{-2/3}), \quad \partial_\xi \tilde{x}_T = \mathcal{O}(T^{-5/3}), \quad \partial_\eta \tilde{x}_T = \mathcal{O}(T^{-2/3}) ,  \\
    &\tilde{y}_T= \mathcal{O}(T^{-2/3}), \quad \partial_\xi \tilde{y}_T = \mathcal{O}(T^{-2/3}), \quad \partial_\eta \tilde{y}_T = \mathcal{O}(T^{-5/3}) ,
\end{aligned}
\]
as stated in Proposition~\ref{uni_prop:Shilnikov} and that
\[
    \xi_\ast, \eta_\ast, \varepsilon_\ast = \mathcal{O}(T^{-2/3})
\]
as stated in Lemma~\ref{uni_lem:renormfirststep}.

We write the left-hand side of \eqref{uni_eq:renormcondproof} in components as $(\Psi_{glob,\varepsilon}\circ \Psi^u_T)(a+\xi, \tilde{a}+\eta)=(\mathcal{Q}(\xi, \eta; \varepsilon), \mathcal{P}(\xi, \eta; \varepsilon))$ and evaluated at $\xi=\xi_\ast+\alpha Q$ and $\eta=\eta_\ast+\alpha P$.

Note that
\begin{equation}\label{uni_eq:components_glob_circ_unst}
    \begin{aligned}
  \mathcal{Q}(\xi,\eta;\varepsilon)=&b\tilde x_T(\xi,\eta)+c \eta+X(\tilde x_T(\xi,\eta),\eta;\varepsilon)\\
\mathcal P(\xi,\eta;\varepsilon)=&\varepsilon_0 + d \tilde{x}_T(\xi,\eta)+\varepsilon_1 \eta +\cdots+\varepsilon_{\ord-1} \eta^{\ord-1} -\eta^{\ord+1}+Y(\tilde{x}_T(\xi,\eta),\eta;\varepsilon).
\end{aligned}
\end{equation}

In this way, we first work with the left-hand side of \eqref{uni_eq:renormcondproof}. On the one hand, its first component has the expression 
\begin{multline}\label{uni_eq:first_component_left_hand_ast}
\Pi_Q (\Psi_{glob,\varepsilon}\circ \Psi^u_T(a+\xi_\ast+\alpha Q, \tilde{a}+\eta_\ast+\alpha P)) \\
\begin{aligned}
     =& \mathcal{Q}(\xi_\ast+\alpha Q, \eta_\ast + \alpha P; \varepsilon) \\
    =& \mathcal{Q}(\xi_\ast, \eta_\ast; \varepsilon)+ \partial_\xi \mathcal{Q}(\xi_\ast, \eta_\ast; \varepsilon) \alpha Q + \partial_\eta \mathcal{Q}(\xi_\ast, \eta_\ast; \varepsilon) \alpha P \\
    &+ \sum_{j=2}^{\ord} \frac{1}{j!}\partial_\eta^{(j)} \mathcal{Q}(\xi_\ast, \eta_\ast; \varepsilon) \alpha^j P^j  + \mathcal{O}(\alpha^2 T^{-5/3}) + \mathcal{O}(\alpha^{\ord+1})\\
    =& \mathcal{Q}(\xi_\ast, \eta_\ast; \varepsilon) + b \partial_\xi \tilde{x}_T(\xi_\ast, \eta_\ast) \alpha Q + \mathcal{O}(\alpha T^{-2/3} T^{-5/3})  + c \alpha P + \mathcal{O}(\alpha T^{-2/3}) P \\
    &+ \sum_{j=2}^{\ord} \frac{1}{j!}\partial_\eta^{(j)} \mathcal{Q}(\xi_\ast, \eta_\ast; \varepsilon) \alpha^j P^j  + \mathcal{O}(\alpha^2 T^{-5/3}) + \mathcal{O}(\alpha^{\ord+1})
\end{aligned}
\end{multline}

On the other hand, its second component is
\begin{multline}\label{uni_eq:second_component_left_hand_ast}
\Pi_P (\Psi_{glob,\varepsilon}\circ \Psi^u_T(a+\xi_\ast+\alpha Q, \tilde{a}+\eta_\ast+\alpha P)) \\
\begin{aligned}
     =&  \mathcal{P}(\xi_\ast+\alpha Q, \eta_\ast + \alpha P; \varepsilon) \\
    =& \mathcal{P}(\xi_\ast, \eta_\ast; \varepsilon) + \partial_\xi \mathcal{P}(\xi_\ast, \eta_\ast; \varepsilon) \alpha Q + \mathcal{O}(\alpha^2 T^{-5/3}) + \sum_{j=1}^{\ord+1} \frac{1}{j!} \partial_\eta^{(j)} \mathcal{P}(\xi_\ast, \eta_\ast; \varepsilon) \alpha^j P^j + \mathcal{O}(\alpha^{\ord+2})\\
    =& \mathcal{P}(\xi_\ast, \eta_\ast; \varepsilon) + d \partial_\xi \tilde{x}_T(\xi_\ast, \eta_\ast) \alpha Q + \sum_{j=1}^{\ord+1} \frac{1}{j!} \partial_\eta^{(j)} \mathcal{P}(\xi_\ast, \eta_\ast; \varepsilon) \alpha^j P^j  \\
    &+ \mathcal{O}(\alpha^2 T^{-5/3}) + \mathcal{O}(\alpha T^{-2/3} T^{-5/3})+ \mathcal{O}(\alpha^{\ord+2}) .
\end{aligned}
\end{multline}

We are going to use that, according to Lemma~\ref{uni_lem:renormfirststep}, we have that $(\Psi_T^s)^{-1} \circ \Psi_{glob, \varepsilon_\ast} \circ \Psi_T^u (a+\xi_\eta, \tilde{a}+\eta_\ast)=(\xi_\ast, \eta_\ast)$, from where we deduce that $\mathcal{Q}(\xi_\ast, \eta_\ast; \varepsilon_\ast)=\xi_\ast$ and that $\mathcal{P}(\xi_\ast, \eta_\ast; \varepsilon_\ast)=\tilde{y}_T(\xi_\ast, \eta_\ast)$.
Now we expand the left-hand side of \eqref{uni_eq:renormcondproof} around $\varepsilon=\varepsilon_\ast$. We get
\begin{multline}\label{uni_eq:left_hand}
\Psi_{glob,\varepsilon}\circ \Psi^u_T(a+\xi_\ast+\alpha Q, \tilde{a}+\eta_\ast+\alpha P)  \\
    \begin{aligned}
    =& \Psi_{glob,\varepsilon_\ast}\circ \Psi^u_T(a+\xi_\ast+\alpha Q, \tilde{a}+\eta_\ast+\alpha P) \\
    &+\sum_{j=0}^{\ord-1} \partial_{\varepsilon_j}|_{\varepsilon=\varepsilon_\ast} (\Psi_{glob,\varepsilon}\circ \Psi^u_T(a+\xi_\ast+\alpha Q, \tilde{a}+\eta_\ast+\alpha P)) \upsilon_j(\kappa) 
    + \mathcal{O}_2(\upsilon).
    \end{aligned}
\end{multline}
Its first component, taking into account \eqref{uni_eq:first_component_left_hand_ast}, is
\begin{multline}\label{uni_eq:first_component_left_hand}
\Pi_Q(\Psi_{glob,\varepsilon}\circ \Psi^u_T(a+\xi_\ast+\alpha Q, \tilde{a}+\eta_\ast+\alpha P)) \\
    \begin{aligned}
    =& \Pi_Q (\Psi_{glob,\varepsilon_\ast}\circ \Psi^u_T(a+\xi_\ast+\alpha Q, \tilde{a}+\eta_\ast+\alpha P)) + \mathcal{O}(\varepsilon_\ast) \mathcal{O}(\upsilon) \\
    =& \xi_\ast + b \partial_\xi \tilde{x}_T(\xi_\ast, \eta_\ast) \alpha Q + \mathcal{O}(\alpha T^{-2/3} T^{-5/3})  + c \alpha P + \mathcal{O}(\alpha T^{-2/3}) P \\
    &+ \sum_{j=2}^{\ord} \frac{1}{j!}\partial_\eta^{(j)} \mathcal{Q}(\xi_\ast, \eta_\ast; \varepsilon_\ast) \alpha^j P^j  + \mathcal{O}(\alpha^2 T^{-5/3}) + \mathcal{O}(\alpha^{\ord+1}) \\
    &+ \mathcal{O}(\varepsilon_\ast) \mathcal{O}(\upsilon).
    \end{aligned}
\end{multline}
Its second component, taking into account \eqref{uni_eq:second_component_left_hand_ast}, is 
\begin{multline}\label{uni_eq:second_component_left_hand}
\Pi_P (\Psi_{glob,\varepsilon}\circ \Psi^u_T(a+\xi_\ast+\alpha Q, \tilde{a}+\eta_\ast+\alpha P)) \\
    \begin{aligned}
    =& \Pi_P (\Psi_{glob,\varepsilon_\ast}\circ \Psi^u_T(a+\xi_\ast+\alpha Q, \tilde{a}+\eta_\ast+\alpha P)) \\
    &+ (1+\mathcal{O}(T^{-2/3})) \upsilon_0(\kappa) + (\mathcal{O}(T^{-2/3})+\alpha P) \upsilon_1(\kappa) \\
    &+ \sum_{j=2}^{\ord-1} (\mathcal{O}(T^{-2/3}) + \mathcal{O}(\alpha^j) P^j) \upsilon_j(\kappa) + \mathcal{O}_2(\upsilon) \\
    =&\tilde{y}_T(\xi_\ast, \eta_\ast) + d \partial_\xi \tilde{x}_T (\xi_\ast, \eta_\ast) \alpha Q + \sum_{j=1}^{\ord+1} \frac{1}{j!} \partial_\eta^{(j)} \mathcal{P}(\xi_\ast, \eta_\ast; \varepsilon_\ast) \alpha^j P^j \\
    &+ \mathcal{O}(\alpha^2 T^{-5/3}) + \mathcal{O}(\alpha T^{-2/3} T^{-5/3})+ \mathcal{O}(\alpha^{\ord+2}) \\
    &+ (1+\mathcal{O}(T^{-2/3})) \upsilon_0(\kappa) + (\mathcal{O}(T^{-2/3})+\alpha P) \upsilon_1(\kappa)  \\
    &+ \sum_{j=2}^{\ord-1} (\mathcal{O}(T^{-2/3}) + \mathcal{O}(\alpha^j) P^j) \upsilon_j(\kappa) + \mathcal{O}_2(\upsilon).
    \end{aligned}
\end{multline}
Now we expand the right-hand side of \eqref{uni_eq:renormcondproof}, obtaining 
\begin{multline*}
\Psi_T^s \circ h_\alpha \circ (H_\kappa + W)(Q,P) \\
    \begin{aligned}
        =& \Psi_T^s \circ h_\alpha (c P + U(Q,P), \kappa_0 + d \gamma Q  + \kappa_1 P + \cdots \kappa_{\ord-1} P^{\ord-1} - P^{\ord+1} + V(Q,P)) \\
        =& \Psi_T^s(\xi_\ast + \alpha c P + \alpha U(Q,P), \eta_\ast + \alpha (\kappa_0 + d \gamma Q + \kappa_ 1 P + \cdots \kappa_{\ord-1} P^{\ord-1} - P^{\ord+1} + V(Q,P))).
    \end{aligned}
\end{multline*}
Then we have that its first component is equal to 
\begin{equation} \label{uni_eq:first_component_right_hand}
    \Pi_Q(\Psi_T^s \circ h_\alpha \circ (H_\kappa + W)(Q,P)) = \xi_\ast + \alpha c P + \alpha U(Q,P) .
\end{equation}
Its second component is
\begin{multline}\label{uni_eq:second_component_right_hand}
\Pi_P(\Psi_T^s \circ h_\alpha \circ (H_\kappa + W)(Q,P)) \\
    \begin{aligned}
        =& \tilde{y}_T(\xi_\ast, \eta_\ast) + \partial_\xi \tilde{y}_T(\xi_\ast, \eta_\ast) (\alpha c P + \alpha U) \\
        &+ \partial_\eta \tilde{y}_T(\xi_\ast, \eta_\ast) \alpha (\kappa_0 + d \gamma Q + \kappa_1 P + \cdots + \kappa_{\ord-1} P^{\ord-1} - P^{\ord+1} + V) \\
        &+ \sum_{j=2}^{\ord} \frac{1}{j!} \partial_\xi^{(j)} \tilde{y}_T(\xi_\ast, \eta_\ast) (\alpha c P + \alpha U)^j + \mathcal{O}(\alpha^2 T^{-5/3}) + \mathcal{O}(\alpha^{\ord+1} T^{-2/3}) .
    \end{aligned}
\end{multline}
Comparing \eqref{uni_eq:first_component_left_hand} and \eqref{uni_eq:first_component_right_hand}, we obtain that
\begin{equation}\label{uni_eq:U_solved}
    \begin{aligned}
\alpha U(Q, P) =&  b \partial_\xi \tilde{x}_T(\xi_\ast, \eta_\ast) \alpha Q + \mathcal{O}(\alpha T^{-2/3} T^{-5/3}) + \mathcal{O}(\alpha T^{-2/3}) P \\
&+ \sum_{j=2}^{\ord} \frac{1}{j!}\partial_\eta^{(j)} \mathcal{Q}(\xi_\ast, \eta_\ast; \varepsilon_\ast) \alpha^j P^j  + \mathcal{O}(\alpha^2 T^{-5/3}) + \mathcal{O}(\alpha^{\ord+1}) + \mathcal{O}(\varepsilon_\ast) \mathcal{O}(\upsilon).
\end{aligned}
\end{equation}
This completely determines $U(Q,P)$, although this expression still depends on $\upsilon(\kappa)$.

Now we plug the value of $U(Q,P)$ given in \eqref{uni_eq:U_solved} at \eqref{uni_eq:second_component_right_hand} and expand the resulting expression in $Q, P$. We obtain
\begin{multline}\label{uni_eq:second_component_right_hand_noU}
\Pi_P(\Psi_T^s \circ h_\alpha \circ (H_\kappa + W)(Q,P)) \\
    \begin{aligned}
        =& \tilde{y}_T(\xi_\ast, \eta_\ast) + \partial_\xi \tilde{y}_T(\xi_\ast, \eta_\ast) (\alpha c P + \mathcal{O}(\alpha T^{-2/3}) P + \sum_{j=2}^{\ord} \frac{1}{j!}\partial_\eta^{(j)} \mathcal{Q}(\xi_\ast, \eta_\ast; \varepsilon_\ast) \alpha^j P^j) \\
        &+ \partial_\eta \tilde{y}_T(\xi_\ast, \eta_\ast) \alpha (\kappa_0 + d \gamma Q + \kappa_1 P + \cdots + \kappa_{\ord-1} P^{\ord-1} - P^{\ord+1} + V) \\
        &+ \sum_{k=2}^{\ord} \frac{1}{k!} \partial_\xi^{(k)} \tilde{y}_T(\xi_\ast, \eta_\ast) (\alpha c P + \mathcal{O}(\alpha T^{-2/3}) P + \sum_{j=2}^{\ord} \frac{1}{j!}\partial_\eta^{(j)} \mathcal{Q}(\xi_\ast, \eta_\ast; \varepsilon_\ast) \alpha^j P^j)^k \\
        &+\mathcal{O}(\alpha^2 T^{-5/3}) + \mathcal{O}(\alpha^{\ord+1} T^{-2/3}) +\mathcal{O}(\alpha T^{-5/3} T^{-2/3}) +\mathcal{O}( T^{-2/3} T^{-2/3} \upsilon) .
    \end{aligned}
\end{multline}

To choose the values of $\alpha$, $\gamma$ and $\upsilon(\kappa)$, we compare the coefficients of terms with equal powers of $Q, P$ in \eqref{uni_eq:second_component_left_hand} and \eqref{uni_eq:second_component_right_hand_noU}. More concretely, we want $\alpha$, $\gamma$ and $\upsilon$ such that
\begin{equation}\label{uni_eq:left_equal_right}
\begin{aligned}
    &\tilde{y}_T(\xi_\ast, \eta_\ast) + d \partial_\xi \tilde{x}_T (\xi_\ast, \eta_\ast) \alpha Q + \sum_{j=1}^{\ord+1} \frac{1}{j!} \partial_\eta^{(j)} \mathcal{P}(\xi_\ast, \eta_\ast; \varepsilon_\ast) \alpha^j P^j \\
    &+ \mathcal{O}(\alpha^2 T^{-5/3}) + \mathcal{O}(\alpha T^{-2/3} T^{-5/3})+ \mathcal{O}(\alpha^{\ord+2}) \\
    &+ (1+\mathcal{O}(T^{-2/3})) \upsilon_0(\kappa) + (\mathcal{O}(T^{-2/3})+\alpha P) \upsilon_1(\kappa)  \\
    &+ \sum_{j=2}^{\ord-1} (\mathcal{O}(T^{-2/3}) + \mathcal{O}(\alpha^j) P^j) \upsilon_j(\kappa) + \mathcal{O}_2(\upsilon) \\
    =& \tilde{y}_T(\xi_\ast, \eta_\ast) + \partial_\xi \tilde{y}_T(\xi_\ast, \eta_\ast) (\alpha c P + \mathcal{O}(\alpha T^{-2/3}) P + \sum_{j=2}^{\ord} \frac{1}{j!}\partial_\eta^{(j)} \mathcal{Q}(\xi_\ast, \eta_\ast; \varepsilon_\ast) \alpha^j P^j) \\
        &+ \partial_\eta \tilde{y}_T(\xi_\ast, \eta_\ast) \alpha (\kappa_0 + d \gamma Q + \kappa_1 P + \cdots + \kappa_{\ord-1} P^{\ord-1} - P^{\ord+1} + V) \\
        &+ \sum_{k=2}^{\ord} \frac{1}{k!} \partial_\xi^{(k)} \tilde{y}_T(\xi_\ast, \eta_\ast) (\alpha c P + \mathcal{O}(\alpha T^{-2/3}) P + \sum_{j=2}^{\ord} \frac{1}{j!}\partial_\eta^{(j)} \mathcal{Q}(\xi_\ast, \eta_\ast; \varepsilon_\ast) \alpha^j P^j)^k \\
        &+\mathcal{O}(\alpha^2 T^{-5/3}) + \mathcal{O}(\alpha^{\ord+1} T^{-2/3}) +\mathcal{O}(\alpha T^{-5/3} T^{-2/3}) +\mathcal{O}( T^{-2/3} T^{-2/3} \upsilon) .
\end{aligned}
\end{equation}
It is immediate that choosing $\alpha^\ord = \partial_\eta \tilde{y}_T(\xi_\ast, \eta_\ast)$ and $\gamma = \frac{\partial_\xi \tilde{x}_T(\xi_\ast, \eta_\ast)}{\partial_\eta \tilde{y}_T(\xi_\ast, \eta_\ast)}$, then there is $\upsilon_j(\kappa)= \mathcal{O}(T^{-2/3}) + \alpha^{\ord+1-j} \kappa_j$, $j=0,1,\ldots,\ord-1$ such that \eqref{uni_eq:left_equal_right} becomes
\begin{multline}\label{uni_eq:V_solved}
\mathcal{O}(\alpha^2 T^{-5/3}) + \mathcal{O}(\alpha T^{-2/3} T^{-5/3})+ \mathcal{O}(\alpha^{\ord+2}) \\
    \begin{aligned}
    =&\partial_\eta \tilde{y}_T(\xi_\ast, \eta_\ast) \alpha  V(Q,P) 
    \\
    &
    + \mathcal{O}(\alpha^\ord T^{-2/3})P^\ord +\mathcal{O}(\alpha^2 T^{-5/3}) + \mathcal{O}(\alpha^{\ord+1} T^{-2/3}) +\mathcal{O}(\alpha T^{-5/3} T^{-2/3}).
\end{aligned}
\end{multline}
Therefore, for $\ord \in \mathbb{N}$ such that $\frac{-2}{3} < \frac{-5}{3 \ord}$, which is true for $\ord \ge 3$, we obtain that there exists $V=\mathcal{O}
(T^{-\beta})$ for some $\beta >0$ satisfying \eqref{uni_eq:V_solved}.
Moreover, from \eqref{uni_eq:U_solved} it can be readily check that for $\ord \ge 3$ we also have $U=\mathcal{O}(T^{-\beta})$.
\end{proof}

Finally, in order to complete the proof of Proposition~\ref{uni_prop:mainresultrenormalization}, we introduce the linear map $\hat{h}:(Q,P)\mapsto (-(d\gamma)^{-1} Q, P)$
and consider $h_T= h\circ\hat{h}$. Then, recalling the definition of $\Psi_{\mathcal{K}_{\varepsilon}}$ in \eqref{uni_eq:Psirenorm}, we obtain that
\begin{equation}\label{uni_eq:asymptoticHenonproof}
h_T^{-1}\circ(\Psi^s_T)^{-1}\circ \Psi_{\mathcal{K}_{\varphi(\kappa)}}^L\circ\Psi^{u}_T\circ h_T(Q,P)=(-cd\gamma P,\kappa_0 + d \gamma Q + \kappa_1 P + \cdots + \kappa_{\ord-1} P^{\ord-1} - P^{\ord+1})+\mathcal{O}(T^{-\beta}),\qquad\qquad 
\end{equation}
with 
\[
\varphi(\kappa)=\varepsilon_*+\upsilon(\kappa).
\]
Moreover, since $(\xi_*,\eta_*)$ is a fixed point of the map $\Psi_{\varepsilon_*}$ and the map $\Psi_{\varepsilon_*}$ preserves a symplectic form, we must have
\[
1=\mathrm{det} D\Psi_{\varepsilon_*}= -cd\gamma+\mathcal{O}(T^{-2/3}).
\]

The result finally follows by taking into account
\[
\Psi_{\mathcal{K}}^T = \Psi_T^u \circ (\Psi_T^s)^{-1} .
\]

Since the Poincaré maps $\Psi_{\mathcal{K}_\varepsilon}$ and $f_\varepsilon$ are related by the change of variables of the normal form, the result may be equivalently stated for the latter map.

\appendix
\section{Shilnikov maps}\label{uni_sec:Shilnikov}
In this appendix we recall the definition and asymptotics of the Shilnikov maps in the case of a degenerate (parabolic) saddle which were obtained in \cite{GMP2025}.

 Let $\mathcal{K}$ be the Hamiltonian obtained after applying the normal form from \cite{GMP2025} and consider the return map on the section $\{\phi=0\}$ associated to its Hamiltonian flow, which we denote by  $\Psi_{\mathcal{K}}:(q,p)\mapsto (\bar q,\bar p)$. Then $O=(0,0)$ is a (degenerate) fixed point for $\Psi_{\mathcal K}$ and its local stable (unstable) manifold is given by the $q$ axis ($p$ axis). 
 
 Let $a>0$ be a sufficiently small fixed constant and define the subset
\begin{equation}\label{uni_eq:dfnsigmadelta}
\Sigma=\{(\xi,\eta)\in\mathbb R^2\colon |\xi-a|,|\eta-a|\leq a/2\}.
\end{equation}
Define also, for any $(\xi_0,\eta_0)\in \Sigma$, the sections
\[
\sigma^s_{\xi_0}=\{q=\xi_0,\ 0\leq p\leq a/4 \},\qquad\qquad \sigma^u_{\eta_0}=\{p=\eta_0,\  0\leq q\leq a/4\},
\]
which are, respectively, transverse to the local stable and local unstable manifolds.
 
 Let $a>0$ be now such that
\[
\{(q,p)\in[0,3a/2]^2\}\subset V_\rho \times V_\rho,
\]
where $V_\rho$ is the complex sectorial domain $V_\rho \coloneqq \{z\in B_{\rho}(0)\subset\mathbb C \colon |\mathrm{Im}z| < \rho \kappa |\mathrm{Re}z|\}$, with $\rho, \kappa >0$ small, in which the normal form is defined.

For any $n\in\mathbb N$ denote by $\psi^s_{\xi_0,n}$ the unique connected component of $\Psi_{\mathcal K}^n(\sigma^s_{\xi_0})\cap \{(q,p)\in[0,3a/2]^2\}$ which contains the point $\Psi^{n}_{\mathcal K}(\xi_0,0)$ and by $\psi^u_{\eta_0,n}$ the unique connected component of $\Psi^{-n}_{\mathcal K}(\sigma^u_{\eta_0})\cap \{(q,p)\in[0,3a/2]^2\}$ which contains the point $\Psi^{-n}_{\mathcal K}(0,\eta_0)$. 

For $a>0$ sufficiently small and $n\in\mathbb N$ large enough, for any $(\xi_0,\eta_0)\in \Sigma$, the sets 
\begin{equation}\label{uni_eq:Shilnikovassumption}
\psi^s_{\xi_0,n}\cap \sigma^u_{\eta_0} \qquad\qquad \psi^{u}_{\eta_0,n}\cap \sigma^s_{\xi_0},
\end{equation}
consist each of a unique point. Then, as in the case of a hyperbolic fixed point, we define the stable and unstable Shilnikov maps by
\begin{equation}\label{uni_eq:Shilnikovmapsdefinition}
\begin{array}{rl}
    \Psi^s_{n}:\Sigma  & \to\mathbb R^2 \\
    (\xi,\eta) &\mapsto  \psi^{u}_{\eta,n}\cap \sigma^s_\xi
\end{array}\qquad\qquad \begin{array}{rl}
    \Psi^u_{n}:\Sigma  & \to\mathbb R^2 \\
    (\xi,\eta) &\mapsto  \psi^s_{\xi,n}\cap \sigma^u_\eta
\end{array}
\end{equation}

For the renormalization scheme in Section~\ref{uni_subsec:renormalization}, it is crucial to be able to describe the behavior of these maps (and their derivatives) on a suitable complex neighborhood of $\Sigma$ in \eqref{uni_eq:dfnsigmadelta}. With that aim, we define for $\delta>0$ the set 
\begin{equation}\label{uni_eq:complexextensionsigma}
\Sigma_\delta=\{\tau=(\xi,\eta)\in\mathbb{C}^2\colon \mathrm{dist}(\tau,\Sigma)\leq \delta\}.
\end{equation}

The following proposition from \cite{GMP2025} shows that the maps \eqref{uni_eq:Shilnikovmapsdefinition} are indeed well defined and, moreover, asymptotic expressions for them are provided. 

\begin{prop}\label{uni_prop:Shilnikov}
Let $a>0$ be sufficiently small. Then, there exists  $T_0\gg1$ such that for every $\varepsilon$ in some ball, all $T\in\mathbb N$,  $T\geq T_0$ and $\delta\gtrsim T^{-1}$,
the Shilnikov stable and unstable maps $\Psi^{s}_T$, $\Psi^u_T$ in \eqref{uni_eq:Shilnikovmapsdefinition} are well defined real-analytic maps on $\Sigma_\delta$  of the form 
\begin{equation}\label{uni_eq:formShilnikovmaps}
\Psi^{s}_T(\xi,\eta)=(\xi,y_T(\xi,\eta)),\qquad\qquad \Psi^{u}_T(\xi,\eta)=(x_T(\xi,\eta),\eta),
\end{equation}
Moreover, uniformly for $(\xi,\eta)\in \Sigma_\delta$,
\begin{itemize}
    \item They satisfy the asymptotic expressions
\begin{equation}\label{uni_eq:asymptoticShilnikov}
\begin{split}
x_T(\xi,\eta)= &\left(\frac{\pi}{16 T}\right)^{2/3} \eta^{-1}\left( 1+\mathcal{O}(T^{-1})\right),\\
y_T(\xi,\eta)=& \left(\frac{\pi}{16 T}\right)^{2/3} \xi^{-1}\left( 1+\mathcal{O}(T^{-1})\right),
\end{split}
\end{equation}
\item Their first derivatives satisfy
\begin{align}\label{uni_eq:asymptoticShilnikovderiv}
\partial_\eta x_T(\xi,\eta)= &-\left(\frac{\pi}{16 T}\right)^{2/3} \eta^{-2}\left(1+\mathcal O(T^{-1})\right)\\
\partial_\xi x_T(\xi,\eta)=&\frac{2}{3T^{5/3}}\left(\frac{\pi}{16 }\right)^{2/3} (\xi\eta)^{-5/2}\left(1+\mathcal O(a)\right)\label{uni_eq:asymptoticShilnikovderiv2}
\end{align}
and 
\begin{align}\label{uni_eq:asymptoticShilnikovderivy}
\partial_\xi y_T(\xi,\eta)= &-\left(\frac{\pi}{16 T}\right)^{2/3} \xi^{-2}\left(1+\mathcal O(T^{-1})\right)\\
\partial_\eta y_T(\xi,\eta)=&\frac{2}{3T^{5/3}}\left(\frac{\pi}{16 }\right)^{2/3} (\xi\eta)^{-5/2}\left(1+\mathcal O(a)\right)\label{uni_eq:asymptoticShilnikovderiv2y}
\end{align}
\item Bounds for any higher order derivatives can be readily obtained by application of Cauchy's theorem, since the previous expressions are valid uniformly for $(\xi, \eta) \in \Sigma_\delta$.
\end{itemize}
\end{prop}

As can be seen in Section~\ref{uni_subsec:renormalization}, the Shilnikov maps are an extremely useful tool to understand the homoclinic picture at the degenerate saddle $O$. The fact that $\partial_\xi x_T\neq 0$ and  $\partial_\eta y_T\neq 0$ guarantees that $\Psi_{T}^{s,u}$ are local diffeomorphisms. This is crucial for developing the renormalization scheme in Section~\ref{uni_subsec:renormalization}.

\printbibliography
\end{document}